\documentclass[12pt]{amsart}
\usepackage[dvips]{graphicx}
\usepackage{amsmath,graphics}
\usepackage{amsfonts,amssymb,amscd,color}
\theoremstyle{plain}
\newtheorem{theorem*}{Theorem}
\newtheorem*{lemma*}{Lemma}
\newtheorem{corollary*}{Corollary}
\newtheorem*{proposition*}{Proposition}
\newtheorem{conjecture*}{Conjecture}
\newtheorem{Itheorem}{Theorem}
\newtheorem{Iprop}[Itheorem]{Proposition}
\newtheorem*{prop*}{Proposition}
\newtheorem{Icor}[Itheorem]{Corollary}
\newtheorem{theorem}{Theorem}[section]
\newtheorem{lemma}[theorem]{Lemma}
\newtheorem{corollary}[theorem]{Corollary}
\newtheorem{proposition}[theorem]{Proposition}

\newtheorem*{thm*}{Theorem}

\theoremstyle{remark}

\newtheorem*{definition}{Definition}

\newtheorem{example}{Example}[section]
\newtheorem*{example*}{Example}
\newtheorem*{claim}{Claim}

\theoremstyle{definition}

\newtheorem{remark}[theorem]{Remark}
\newtheorem{defn}[theorem]{Definition}
\newtheorem*{defn*}{Definition}

\newtheorem*{conv}{Conventions}

 \def\Q{\Bbb{Q}} \def\F{\Bbb{F}} \def\Z{\Bbb{Z}} \def\R{\Bbb{R}} \def\T{\Bbb{T}}\def\C{\Bbb{C}}
   \def\ll{\langle} \def\rr{\rangle}
 \def\a{\alpha} \def\g{\gamma} \def\tor{\mbox{Tor}} \def\bp{\begin{pmatrix}}
\def\sm{\setminus} \def\ep{\end{pmatrix}} \def\bn{\begin{enumerate}} \def\Hom{\mbox{Hom}}
   \def\en{\end{enumerate}}
\def\ba{\begin{array}} \def\ea{\end{array}}  
\def\intt{\mbox{Int}} \def\S{\Sigma} \def\s{\sigma} \def\a{\alpha} \def\b{\beta} \def\ti{\tilde}
\def\id{\mbox{id}}  \def\im{\mbox{Im}} \def\sign{\mbox{sign}}
  
\def\ker{\mbox{Ker}}\def\be{\begin{equation}} \def\ee{\end{equation}} 
   
  \def\Ext{\mbox{Ext}} 
   \def\eps{\epsilon}
 \def\dim{\mbox{dim}}

\def\zt{\Z[t^{\pm 1}]}    
\def\w{\omega}   
    \def\fr12{\frac{1}{2}} \def\z12{\Z[\fr12]}

\def\eul{\mbox{Eul}}

\def\spinc{\mbox{Spin}^c}

\def\max{\mbox{max}}
\def\min{\mbox{min}}
\def\lift{\mbox{Lift}}
\def\v{\varphi}

\def\S{\Sigma}
\def\s{\mathfrak{s}}
\def\t{\mathfrak{t}}

\def\bolda{\boldsymbol{\alpha}}
\def\boldb{\boldsymbol{\beta}}
\def\boldx{{\bf x}}
\def\boldy{{\bf y}}

\def\bbx{\mathfrak{l}(\boldx)}

\def\frakl{\mathfrak{l}}
\def\ts{\thinspace}
\def\HH{{G}}

\def\bbar{\overline{\beta}}
\def\E{\mathcal{E}}
\def\A{\mathcal{A}}

\def\e{\mathfrak{e}}

\begin{document}

\title{The decategorification of sutured Floer homology}

\author{Stefan Friedl}
\address{Mathematisches Institut\\ Universit\"at zu K\"oln\\ Weyertal 86-90\\ 50931 K\"oln\\ Germany}
\email{sfriedl@gmail.com}

\author{Andr\'as Juh\'asz}
\address{DPMMS\\ University of Cambridge\\Wilberforce Road\\ Cambridge CB3 0WA\\ UK}
\email{A.Juhasz@dpmms.cam.ac.uk}

\author{Jacob Rasmussen}
\address{DPMMS\\ University of Cambridge\\Wilberforce Road\\ Cambridge CB3 0WA\\ UK}
\email{J.Rasmussen@dpmms.cam.ac.uk}

\date{\today}
\begin{abstract}
We define a torsion invariant $T$ for every balanced sutured manifold $(M,\g),$ and show that it agrees with the Euler characteristic of sutured Floer homology $SFH.$  The invariant $T$ is easily computed using Fox calculus. With the help of $T,$ we prove that if $(M,\g)$ is complementary to a Seifert surface of an alternating knot, then $SFH(M,\g)$ is either $0$ or $\Z$ in every $\spinc$ structure. $T$ can also be used to show that a sutured manifold is not disk decomposable, and to distinguish between Seifert surfaces.

The support of $SFH$ gives rise to a norm $z$ on $H_2(M,\partial M;\R).$
The invariant $T$ gives a lower bound on the norm $z,$ which in turn is at most the sutured Thurston norm $x^s.$ For closed three-manifolds, it is well
known that Floer homology determines the Thurston norm, but we show that $z < x^s$ can happen in general. Finally, we compute $T$ for several wide classes of sutured manifolds.
\end{abstract}
\maketitle

\section{Introduction}

Sutured Floer homology  is an invariant of balanced sutured manifolds introduced by the second author \cite{Ju06}. It is an offshoot of the Heegaard Floer homology of Ozsv{\'a}th and Szab{\'o} \cite{OS04a}, and contains  knot Floer homology \cite{OS04c,Ra03} as a special case. The Euler characteristics of these homologies are torsion invariants of three--manifolds. For example, the Euler characteristic of the Heegaard Floer homology \(HF^+\) is given by Turaev's refined torsion \cite{Tu97,Tu02} and the Euler characteristic of knot Floer homology is given by the Alexander polynomial. In this paper, we investigate the torsion invariant which is the Euler characteristic of sutured Floer homology.

To make a more precise statement, we recall some basic facts about sutured Floer homology. Given a balanced sutured manifold \((M, \g)\) and a relative \(\spinc\) structure \(\s \in \spinc(M, \g)\), the sutured Floer homology is a finitely generated abelian group \(SFH(M, \g,\s)\). {\it A priori}, \(SFH(M, \g,\s)\) is relatively \(\Z/2\) graded. To fix an absolute \(\Z/2\) grading, we must specify a homology orientation \(\w\) of the pair \((M,R_-(\gamma))\); i.e., an orientation of the vector space $H_*(M,R_-(\g);\R).$ We denote the resulting invariant by $SFH(M,\g,\s,\w)$.

Following Turaev, we define a torsion invariant \(T_{(M, \g)}\) for {weakly} balanced sutured manifolds, which is essentially the maximal abelian torsion of the pair \((M,R_-(\gamma))\). Our construction is generally very close to Turaev's, but we use handle decompositions of sutured manifolds in place  of triangulations. This makes it easier to define a correspondence between lifts and $\spinc$ structures.
\(T_{(M,\g)}\) is a function which assigns an integer  to each \(\s \in \spinc(M, \g)\). The torsion function is well-defined up to a global factor of \(\pm 1\). Again, to fix the sign, we must specify a homology orientation \(\w\) of \((M,R_-(\gamma))\). Then we obtain a function
$T_{(M,\g,\w)} \colon \spinc(M,\g) \to \Z.$

\begin{Itheorem} \label{maintheorem}
Let $(M,\gamma)$ be a balanced sutured manifold. Then  for
any  $\s\in \spinc(M,\g)$ and homology orientation \(\w\) of \((M,R_-(\gamma))\),  $$T_{(M,\g,\w)}(\s)=\chi(SFH(M,\g,\s,\w)).$$
\end{Itheorem}

Related invariants have also been studied by Benedetti and Petronio \cite{BP01} using a slightly different approach and terminology, and by Goda and Sakasai \cite{GS08} in the case of homology products.

It is often convenient to combine the torsion invariants of \((M,\g)\) into a single generating function, which we view as an element of the group ring \(\Z[H_1(M)]\). To do so, we fix an affine $H_1(M)$--isomorphism \(\iota:\spinc(M,\g) \to H_1(M)\), and write
$$ \tau(M,\g) = \sum_{\s \in \text{Spin}^c(M,\g)} T_{(M,\g)}(\s)  \cdot \iota(\s).  $$
The invariant \(\tau(M,\g)\) is best thought of as a generalization of the  Alexander polynomial to sutured manifolds.
Like the classical Alexander polynomial, it is well defined up to multiplication by elements of the form \(\pm [h]\) for \(h \in H_1(M)\). Many  properties of the Alexander polynomial have  analogues for \(\tau(M,\g)\).

For example,   if \(Y\) is a three-manifold with toroidal boundary, then its Alexander polynomial \(\Delta(Y)\) is by definition an invariant of \(\pi_1(Y)\). Similarly, for $\tau(M,\g)$ the following holds. Note that by Lemma \ref{lem:product} we can always assume that both $R_+(\g)$ and $R_-(\g)$ are connected for the purpose of computing $\tau(M,\g).$

\begin{Iprop}
\label{prop:pi1}
Let $(M,\g)$ be a balanced sutured manifold such that $M$ is irreducible and both $R_+(\g)$ and $R_-(\g)$ are connected. Then the invariant  \(\tau(M,\g)\) can be computed from the map $\pi_1(R_-(\gamma)) \to \pi_1(M) $ using Fox calculus.
\end{Iprop}


A well-known theorem of McMullen \cite{Mc02} says that  \(\Delta(Y)\) gives a lower bound on the Thurston norm of \(Y\). There is a natural extension of the Thurston norm to sutured manifolds due to Scharlemann \cite{Sc89};  we recall its definition
in Section \ref{section:norm}.  For \(\a \in H_2(M, \partial M;\R)\), let \(x^s(\a)\) denote this {sutured Thurston norm}.

\begin{Iprop}
\label{prop:Anorm}
Suppose that $(M,\g)$ is a balanced sutured manifold such that $M$ is irreducible. Let \(\mathcal{S} \subset H_1(M)\) be the support of   \(\tau (M,\g)\). Then  \emph{
$$
\max_{\s, \mathfrak{t} \in S} \langle \s - \mathfrak{t}, \alpha \rangle  \leq x^s(\alpha).
$$}
\end{Iprop}

Proposition~\ref{prop:Anorm} can be proved analogously to McMullen \cite{Mc02} and Turaev \cite{Tu02}. However, we will not do that here since it is an immediate consequence of
Theorem~\ref{maintheorem} and the following result.


\begin{Itheorem}
\label{thm:Tnorm}
Suppose that $(M,\g)$ is a balanced sutured manifold such that $M$ is irreducible. Let \(S \subset \spinc(M,\g)\) be the support of   \(SFH(M,\g)\). Then  \emph{
$$
\max_{\s, \mathfrak{t} \in S} \langle \s - \mathfrak{t}, \alpha \rangle
  \leq x^s(\alpha).
$$}
Furthermore, equality holds if $\partial M$ consists only of tori.
\end{Itheorem}

In analogy with the situation for closed manifolds, it is tempting to guess that one always has equality in Theorem \ref{thm:Tnorm}, but using an example of Cantwell and Conlon in \cite{CC06} (cf. Proposition \ref{prop:2}) we show that this is not the case.

The unit Thurston norm ball of a link complement is always centrally symmetric. We demonstrate
in Example \ref{ex:nonsymmetric}
that $S$ and $\mathcal{S}$ can be centrally asymmetric in general.

One final property of \(\tau\) is that its   ``evaluation'' under the map \(\Z[H_1(M)] \to \Z[H_1(M,R_-(\gamma))]\) is very simple. More precisely, we show the following.
\begin{Iprop}
\label{Prop:eval}
Let  \(p_*:H_1(M) \to H_1(M,R_-(\gamma))\) be the natural map. Then
\[ p_*(\tau(M,\g)) = \pm I_{H_1(M,R_-(\gamma))},\]
where given a group \(G\)  we define
\(I_G \in \Z[G]\) to be
$$I_G = \begin{cases}  \sum_{g\in G} g \quad & |G|<\infty, \\
					0 \quad & |G| = \infty. \end{cases}
					$$	
\end{Iprop}

For example, suppose \(K\) is a knot in a homology sphere, and let \((M,\g)\) be the sutured manifold whose total space is the complement of \(K\) and whose boundary contains two sutures parallel to the meridian of \(K\). Then it can be shown that \(\tau(M,\g) = \Delta_K(t) \). On the other hand,  \(H_1(M,R_-(\gamma)) = 0\), so
\(I_{H_1(M,R_-(\gamma))} = 1\). Thus in this case the proposition reduces to the  fact that \(\Delta_K(1)=\pm 1\).

\begin{defn}
A balanced sutured manifold \((M,\g)\) is a {\it sutured \(L\)-space} if the group \(SFH(M,\g)\) is torsion-free and is supported in a single \(\Z/2\) homological grading.
\end{defn}

Examples of such manifolds are  easy  to find; e.g., if \(R \subset S^3\) is a Seifert surface of an alternating link, then we will show in Corollary \ref{cor:alt}
that the sutured manifold complementary to \(R\) is a sutured \(L\)-space.
The next result follows from Proposition~\ref{Prop:eval}.
\begin{Icor}
\label{Cor:Lspace}
If \((M,\g)\) is a sutured \(L\)-space, then for each \(\s \in \spinc(M,\g)\) the group \(SFH(M,\g,\s)\) is either trivial or isomorphic to \(\Z\).
\end{Icor}

In the last section, we compute the torsion for a variety of examples, including pretzel surface complements, and for all sutured manifolds complementary to Seifert surfaces of knots with $\le 9$ crossings. In all these examples, the sutured Floer homology is easily determined from the torsion.
As an application, we give a simple example of a phenomenon first demonstrated by Goda \cite{Go94}: There exist sutured manifolds whose total space is a handlebody, but which are not disk decomposable.
In fact, such examples are not difficult to come by; we found this one by writing a computer program which  calculates $\tau(M,\g)$ when $M$ is a genus two handlebody, and looking at the output in some simple cases.


The paper is organized as follows. In Section \ref{section:SFH}, we recall the relevant facts about sutured Floer homology.
We furthermore show how sutured Floer homology behaves under orientation reversal. Section~\ref{section:torsion} contains the definition of the torsion, and
section~\ref{section:comp} explains how to compute it using Fox calculus. Section~\ref{section:maintheorem} contains the proof of Theorem~\ref{maintheorem}. In Section~\ref{section:algebra}, we discuss some algebraic properties of the torsion, including Proposition~\ref{Prop:eval}. Section~\ref{section:norm} discusses the relation between \(SFH\) and the sutured Thurston norm. Finally, Section~\ref{examples} is devoted to examples.

The authors would like to thank Marc Lackenby for pointing out the connection to reference \cite{GS08}. The second author would also like to thank IH\'ES for its hospitality during the course of this work, and the Herchel Smith Fund for their generous support. Finally, we would like to thank Irida Altman and the anonymous referee for carefully reading our manuscript and for many helpful comments.

\begin{conv}
All 3--manifolds are understood to be oriented and compact. All homology groups are with integral coefficients unless otherwise specified. Given a 3-manifold $Y$ with boundary, we routinely identify $H_i(M)$ with $H^{3-i}(M,\partial M)$.
If $X$ is a submanifold of $Y,$ then $N(X)$ denotes an open tubular neighborhood of $X$ in $Y.$
\end{conv}

\section{Sutured Floer homology}
\label{section:SFH}

In this section, we recall some relevant facts about sutured manifolds and sutured Floer homology. For full details, we refer the reader to \cite{Ju06}.

\subsection{Balanced sutured manifolds} \label{section:bsm}
 For our purposes, a \emph{sutured manifold} $(M, \g)$ is a compact oriented 3-manifold $M$
with boundary together with a set $s(\g)$ of oriented simple closed curves on $\partial M$, called \emph{sutures}.
We fix a closed tubular neighborhood \(\g \subset \partial M\) of  the sutures, hence $\g$ is a union  of pairwise disjoint annuli.
Finally, we require that each component of $R(\g) = \partial M \sm \intt(\g)$ be oriented, and that this orientation is coherent with respect to $s(\g).$ I.e.,
if $\delta$ is a component of $\partial R(\g)$ and is given the boundary orientation,
then $\delta$ must
represent the same homology class in $H_1(\g)$ as some suture.
Define $R_+(\g)$ to be the union of those components of $R(\g)$ whose orientation is consistent with the orientation on \(\partial M\)  induced by \(M,\) and let $R_-(\gamma) = R(\g) \setminus R_+(\g).$

The notion of a sutured manifold is due to Gabai \cite{Ga83}. The description given above is slightly less general than Gabai's, in that we have omitted the possibility of toroidal sutures.

\begin{example} Let $R$ be a compact oriented surface with no closed components.
Then there is an induced orientation on $\partial R$. Let $M = R \times [-1,1]$, define $\gamma = \partial R \times [-1,1]$,
finally put $s(\gamma) = \partial R \times \{0\}$. Such a pair $(M, \gamma)$
is called a \emph{product sutured manifold}.
\end{example}

\begin{example} \label{ex:1}
Suppose \(Y\) is a closed connected oriented three-manifold. Let \(M=Y\setminus \mbox{Int}(B^3)\), and let \(s(\g)\) be an oriented simple closed curve on \(\partial B^3.\) We denote the resulting sutured manifold by \(Y(1)\).
\end{example}

\begin{example} \label{ex:2}
Suppose that $L$ is a link in the oriented three-manifold $Y.$ Then the sutured manifold $Y(L) = (M,\g)$ is given by $M = Y \setminus N(L),$ and for each component $L_0$ of $L$ we take $s(\gamma) \cap \partial N(L_0)$ to be two oppositely oriented meridians of $L_0.$
\end{example}

\begin{example} \label{ex:3}
Let \(L\) be a null-homologous link in a closed oriented three-manifold \(Y\), and let \(R\) be a Seifert surface for \(L\).
If \(U \simeq \intt(R) \times (-1,1) \) is a regular neighborhood of \(\intt(R)\), then the complement  \(M = Y \sm U\) is a sutured manifold with \(\gamma = \partial R \times [-1,1].\) The curve \(s(\gamma)\) is \(\partial R \times \{0\}.\) Then $(M,\gamma)$ is called the \emph{sutured manifold complementary to} $R,$ and is  denoted by $Y(R).$
\end{example}


\begin{defn} A \emph{weakly balanced sutured manifold} is a sutured manifold $(M,\g)$ such that for each component $M_0$ of $M$ we have $$\chi(R_+(\g) \cap M_0) = \chi(R_-(\gamma) \cap M_0).$$ A \emph{balanced sutured manifold} is a weakly balanced sutured manifold $(M, \g)$ such that $M$ has no closed components and the map $\pi_0(\g) \to \pi_0(\partial M)$ is surjective. Finally, we say that $(M,\g)$ is \emph{strongly balanced} if it is balanced and for each component $V$ of $\partial M$ we have $\chi(R_+(\g) \cap V) = \chi(R_-(\gamma) \cap V).$
\end{defn}

Balanced sutured manifolds were defined in \cite{Ju06} and strongly balanced sutured manifolds in \cite{Ju08}. The examples given  above are all strongly balanced. Since
$2\chi(M)= \chi(\partial M) = \chi(R_-(\gamma))+\chi(R_+(\g))$, for a weakly balanced sutured manifold  $$\chi(M,R_-(\gamma))=\chi(M,R_+(\g))=0.$$

Sutured Floer homology is only defined for balanced sutured manifolds. However, we can define the torsion for any weakly balanced sutured manifold.

\subsection{$\spinc$--structures on sutured manifolds}\label{section:spinc}

Suppose that  $(M, \gamma)$ is a sutured manifold.
Let $v_0$ be a nowhere zero vector field along $\partial M$ that points into $M$
along $\mbox{Int}\,R_-(\gamma)$, points out of $M$ along $\mbox{Int}\,R_+(\gamma)$, and on $\gamma$ is given by the gradient of a height function $s(\gamma) \times [-1,1] \to [-1,1]$.


\begin{defn} Let $v$ and $w$ be nowhere zero vector fields on $M$ that agree with $v_0$ on $\partial M$. We say that $v$ and $w$ are \emph{homologous} if in each component $M'$ of $M$ there is an open ball $B\subset \intt(M')$ such that
$v|(M' \setminus B)$ is homotopic to $w|(M' \setminus B) \mbox{ rel } \partial M$ through nowhere zero vector fields. We define $\spinc(M, \gamma)$ to be the set of homology classes of nowhere zero vector fields $v$ on $M$ such that $v|\partial M = v_0$.
\end{defn}

\emph{A priori}, this definition appears to depend on the choice of \(v_0\).
However, the space of such vector fields is contractible, so there is a canonical identification between equivalence classes coming from different choices of \(v_0\).
In the case of a closed, oriented 3--manifold the definition is equivalent to the standard definition given  in terms of bundles (cf. \cite{Tu97}). We expect that there is a bundle theoretic interpretation of $\spinc(M,\gamma)$, but we have not explored this question. Note that $\spinc$ structures on sutured manifolds were first introduced by Benedetti and Petronio \cite{BP01}, and they called them smooth Euler structures.

\begin{lemma}
$\spinc(M,\g) \neq \emptyset$ if and only if $(M,\g)$ is weakly balanced. Furthermore, there exists a free and transitive action of $H^2(M, \partial M) \cong H_1(M)$ on the set $\spinc(M, \gamma)$.
\end{lemma}

\begin{proof}
An analogous argument as in the proof of \cite[Proposition 3.6]{Ju10} implies that $\spinc(M,\g) \neq \emptyset$ if and only if $(M,\g)$ is weakly balanced. It follows from obstruction theory that $\spinc(M,\g)$ is an affine space over $H^2(M,\partial M;\mathbb{Z}),$ since nowhere zero vector fields can be thought of as sections of the unit sphere bundle $STM.$ If  $\s_1,\s_2 \in \spinc(M,\g),$ then $\s_1 -\s_2$ is the first obstruction to homotoping vector fields representing $\s_1$ and $\s_2.$

If $v$ is a representative of $\s$ and the simple closed curve $c$ represents $h \in H_1(M),$ then an explicit representative of $\s + h$ can be obtained by Reeb turbularization, which is described in \cite[p.639]{Tu90}.
\end{proof}

\subsection{Sutured Floer homology}

We now sketch the construction of \(SFH(M,\g)\). Our starting point is a Heegaard diagram adapted to the pair \((M, \g)\).

\begin{defn}
A \emph{balanced sutured Heegaard diagram}, in short a balanced diagram, is a triple
$(\S,\bolda,\boldb),$ where $\S$ is a compact
oriented surface with boundary and $\bolda = \{ \a_1,\dots,\a_d\}$ and $\boldb=  \{ \b_1,\dots,\b_d\}$
are two sets of pairwise disjoint simple closed curves in $\intt(\S)$ such that $\pi_0(\partial \S)\to \pi_0(\S\sm \bigcup\bolda)$
and $\pi_0(\partial \S)\to \pi_0(\S\sm \bigcup\boldb)$ are both surjective.
\end{defn}

Note that the restrictions on $\bolda$ and $\boldb$ are equivalent to the conditions that $\S$ has no closed components and that the elements of $\bolda$ and $\boldb$ are both linearly independent in $H_1(\S)$.

Every balanced diagram $(\S,\bolda,\boldb)$ uniquely defines a
sutured manifold $(M,\gamma)$ using the following construction.
Let $M$ be the 3-manifold obtained from $\Sigma\times [-1,1]$ by attaching 2--handles along the curves
$\a_i\times \{-1\}$ for $i=1,\dots,d$ and along $\b_j\times \{1\}$ for $j=1,\dots,d.$
The sutures are defined by taking $\gamma = \partial \S \times [-1,1]$ and $s(\gamma) = \partial \S \times \{0\}$.

Equivalently,  $(M,\g)$ can be constructed from the product sutured manifold \(R_-(\gamma) \times I\) by first adding \(d\) one--handles to \(R_-(\gamma) \times \{1\}\), and then \(d\) two--handles. The Heegaard surface \(\S\) is the upper boundary of the manifold obtained by adding the one--handles.  The \(\alpha\) curves are the belt circles of the one--handles, and the \(\b\) curves are the attaching circles of the two--handles.

The following proposition combines \cite[Proposition~2.9]{Ju06} and \cite[Proposition~2.13]{Ju06}.

\begin{proposition}\label{prop:suturedheegaard}
The sutured manifold defined by a balanced diagram is balanced, and for every balanced sutured manifold there exists a balanced diagram defining it.
\end{proposition}

If $(\S,\bolda,\boldb)$ is a
balanced diagram  for $(M,\gamma),$ then the \(\alpha\) and \(\beta\) curves define the tori
$\T_\a  = \a_1\times \ldots \times \a_d $ and
$\T_\b   = \b_1\times \ldots \times \b_d $
in the symmetric product $\text{Sym}^d(\S).$ We can suppose that $\bigcup\bolda$ and $\bigcup\boldb$ intersect transversally. Then \(SFH(M, \g)\) is the homology of a chain complex whose generators are the intersection points of \(\T_\a\) and \(\T_\b\). More concretely, an element of
\(\T_\a \cap \T_\b\) is a set \(\boldx = \{x_1,\ldots, x_d\}\), where each \(x_i \) is in some \(\a_j \cap \b_k\), and each
\(\a\) and \(\b\) curve is represented exactly once among the \(x_i\)'s.
Still more concretely, for each permutation $\sigma\in S_d$  we define
\[ (\T_\a\cap \T_\b)_\sigma = \{ (x_1,\dots,x_d) \, \colon\, x_i\in \a_i\cap \b_{\sigma(i)}, i=1,\dots,d\}.\]
Then
\[ \T_\a \cap \T_\b =\bigcup_{\sigma\in S_d} (\T_\a\cap \T_\b)_\sigma.\]

The differential in the chain complex is defined by counting rigid holomorphic disks in \(\text{Sym}^d (\S).\) Since we are mostly interested in the Euler characteristic of \(SFH\), we will have little need to understand these disks; in fact, the only place they appear is in the proof of Proposition~\ref{prop:duality}. For the full definition of the differential, the interested reader is referred to \cite{Ju06}.


\subsection{Orientations and Grading} \label{sec:ori}
Next, we consider the homological grading on the sutured Floer chain complex.
In its simplest form, this grading is a relative \(\Z/2\) grading given by the
sign of intersection in $\text{Sym}^d(\S)$ --- two generators have the same grading if the corresponding intersection points in $\T_\a\cap \T_\b$ have the same sign. To fix the sign of intersection, or equivalently, to turn this relative \(\Z/2\) grading into an absolute one, we must orient $\text{Sym}^d(\S),$ and the tori \(\T_\a\) and \(\T_\b.\)

The orientation of \(\Sigma\) is determined by the orientation of the sutures, as we require that $\partial \S = s(\g).$ Equivalently, \(\Sigma\) is always oriented as the boundary of the compression body determined by the \(\a\) curves (the part of $M \setminus \S$ containing $R_-(\g)$), which we view as a submanifold of $M.$ Using this orientation of $\S,$ we get the product orientation on $\text{Sym}^d(\S).$ If $\S$ is endowed with a complex structure compatible with its orientation, then the complex orientation on $\text{Sym}^d(\S)$ agrees with the product orientation. However, to get a well-defined $\Z/2$ grading on $SFH(M,\g),$ we will always consider $\text{Sym}^d(\S)$ with $(-1)^{d(d-1)/2}$ times the product orientation.

Choosing an orientation of \(\T_\a\) is the same as choosing a generator of \(\Lambda^d(A)\), where \(A \subset H_1(\Sigma;\R)\) is the \(d\)-dimensional subspace spanned by the \(\a\)'s. Similarly, an orientation of \(\T_\b\) is specified by a choice of generator for \(\Lambda^d(B)\), where \(B\) is the subspace of $H_1(\Sigma;\R)$ spanned by the \(\b\)'s. To fix the sign of intersection, we must orient the tensor product
\(\Lambda^d(A) \otimes \Lambda^d(B)\). This turns out to be  equivalent to choosing a homology orientation of  \((M,R_-(\gamma)).\)

\begin{defn}
\label{def:torussign}
Suppose we are given a balanced Heegaard diagram $(\S,\bolda,\boldb)$ for $(M,\g).$ Then we define a bijection $o$ from the set of orientations of $H_*(M,R_-(\g))$ to the set of orientations of $\Lambda^d(A) \otimes \Lambda^d(B).$

For simplicity, write $R_- = R_-(\g).$ The balanced diagram $(\S,\bolda,\boldb)$ gives a relative handle decomposition of $M$ built on $R_- \times I:$ attach one-handles $A_1,\dots,A_d$ to $R_- \times I$ with belt circles $\a_1,\dots,\a_d,$ followed by two-handles $B_1,\dots,B_d$ with attaching circles $\b_1,\dots,\b_d.$ Let $C_* = C_*(M,R_- \times I;\R)$ be the handle homology complex corresponding to this handle decomposition.

An orientation $\w$ of $H_*(M,R_-;\R)$ determines an orientation $\w'$ of $C_*$ as follows. First, choose an ordered basis $h_1^1,\ldots h_m^1,h_1^2,\ldots h_m^2$ of $H_*(M,R_-;\R)$ compatible with $\w$ such that $h_j^i \in H_i(M,R_-;\R),$ and pick chains $c_j^i$ representing the $h_j^i.$ Next, choose chains $b_1,\ldots,b_{d-m} \in C_2(M,R_-;\R)$ such that $c_1^2,\ldots,c_m^2,b_1,\ldots,b_{d-m}$ is a basis of $C_2(M,R_-;\R).$ Then $$c_1^1,\ldots c_m^1, \partial b_1,\ldots \partial b_{d-m}, c_1^2,\ldots c_m^2,b_1,\ldots, b_{d-m}$$ is an oriented basis of $C_*.$ The reader can easily verify that the corresponding orientation $\w'$ of $C_*$ does not depend on the choice of $c^i_j$ and $b_k.$

Given the orientation $\w'$ of $C_*,$ we orient $\Lambda^d(A) \otimes \Lambda^d(B)$ as follows. Suppose that the handles $A_1,\ldots, A_d,B_1,\ldots,B_d$ give an ordered basis of $C_*$ compatible with $\w'.$ This gives rise to an orientation and ordering $\a_1,\dots,\a_d, \b_1, \dots, \b_d$  of the $\a$ and $\b$ curves. Let $o(\w)$ be the orientation of $\Lambda^d(A) \otimes \Lambda^d(B)$ given by $$[\alpha_1] \wedge \dots \wedge [\alpha_d] \wedge [\beta_1] \wedge \dots \wedge [\beta_d].$$ It is easy to see that $o(-\w) = -o(\w),$ hence $o$ is indeed a bijection.
\end{defn}

\begin{defn}
Suppose that $(\S,\bolda,\boldb)$ is a balanced diagram such that both $\S$ and \(\Lambda^d(A) \otimes \Lambda^d(B)\) are oriented.
Then for $\boldx = (x_1,\dots,x_d) \in \T_{\a} \cap \T_{\b}$ let $m(\boldx)$ be the intersection sign of $\T_{\a}$ and $\T_{\b}$ at $\boldx.$

Now assume that each $\a \in \bolda$ and each $\b \in \boldb$ is oriented.  If $x \in \a \cap \b,$ then let $m(x)$ denote the intersection sign of $\a$ and $\b$ at $x.$
\end{defn}

\begin{lemma} \label{lem:m}
Let $(\S,\bolda,\boldb)$ be a balanced diagram such that \(\Lambda^d(A) \otimes \Lambda^d(B)\) is oriented.
Suppose that each $\a \in \bolda$ and each $\b \in \boldb$ is oriented such that the product orientations on $\T_{\a}$ and $\T_{\b}$ are consistent with the orientation on \(\Lambda^d(A) \otimes \Lambda^d(B).\) If $\boldx = (x_1,\dots,x_d) \in \T_{\a} \cap \T_{\b},$
and $x_i \in \a_i \cap \b_{\sigma(i)}$ for some $\sigma \in S_d,$ then
\em $$m(\boldx) =  \text{sign}(\sigma) \cdot \prod_{i=1}^d m(x_i).$$
\end{lemma}

\begin{proof}
For $1 \le i \le d,$ let $a_i \in T_{x_i}\alpha_i$ and $b_{\sigma(i)} \in T_{x_i}\beta_{\sigma(i)}$ be positive tangent vectors. Furthermore, let $s_i$ be a positive basis of $T_{x_i}\S.$ Then $s = (-1)^{d(d-1)/2} \cdot s_1 \times \dots \times s_d$ is a positive basis of
$T_{\boldx}\text{Sym}^d(\S).$ Following the notation of \cite{Tu01}, we have
$$\prod_{i=1}^d m(x_i) = \prod_{i=1}^d \text{sign}[a_ib_{\sigma(i)}/s_i] = (-1)^{d(d-1)/2} \cdot \text{sign}[a_1b_{\sigma(1)}\dots a_db_{\sigma(d)}/s] =$$ $$ = \text{sign}[a_1 \dots a_d b_{\sigma(1)} \dots b_{\sigma(d)}/ s]=\text{sign}(\sigma) \text{sign}[a_1 \dots a_d b_1 \dots b_d / s]
= \text{sign}(\sigma) \cdot m(\boldx).$$
\end{proof}

\begin{defn}
\label{def:SFHsign}
A homology orientation \(\omega\) on \(H_*(M,R_-)\) determines an absolute \(\Z/2\) grading
on the sutured Floer chain complex. Under this grading, the sign assigned to a generator \(\boldx\) is \((-1)^{b_1(M,R_-)}m(\boldx).\) \end{defn}

\begin{remark}
If \(Y\) is a closed oriented three-manifold,  then \(SFH(Y(1)) \cong \widehat{HF}(Y)\). In this case, the manifold \(Y(1)\) admits a canonical homology orientation \(\omega\), defined as follows. If \(b_1,\ldots, b_m\) is any basis for \(H_1(Y;\R)\), let \(b_1^*,\ldots,b_m^*\) be the dual basis for \(H_2(M; \R)\), which satisfies \(b_i \cdot b_j^* = \delta_{ij}\). Then $\w$ is
given by the ordered basis $b_1,\dots,b_m,b_1^*,\dots,b_m^*.$ For this orientation, the \(\Z/2\) grading defined above {\em differs} from the canonical \(\Z/2\) grading on \(\widehat{HF}(Y)\) defined in Section 10.4 of \cite{OS04b} by a factor of \((-1)^{b_1(M)}\). Although this choice is less natural from the perspective of intersections in the symmetric product, it is better behaved with respect to surgery formulas and the surgery exact triangle. For example, consider the exact triangle for surgery on a knot \(K \subset S^3\):
$$
\widehat{HF}(S^3) \to \widehat{HF}(K_n) \to \widehat{HF}(K_{n+1}) \to  \widehat{HF}(S^3).
$$
For \(n\neq -1,0\), the first map in this sequence reverses the absolute \(\Z/2\) grading on \(\widehat{HF}\), and the other two maps preserve it. However, in the case where \(n=0\) and \(K\) is the unknot, the reader can easily check that the second map reverses grading and the other two preserve it. In contrast, if we use the orientation convention of Definition~\ref{def:SFHsign}, the maps in the triangle have the same grading regardless of \(n\).
\end{remark}


\subsection{Generators and \(\spinc\) structures}
\label{Subsec:Gens}

An important property of the sutured Floer chain complex is that it decomposes as a direct sum over
\(\spinc\) structures. Definition 4.5 of \cite{Ju06}, and also Remark~\ref{rem:spinc} of the present paper, explain how to  assign a \(\spinc\) structure \(\s(\boldx)\)
to each  \(\boldx \in \T_\a\cap \T_\b\) such that if the boundary \(\partial\boldx = \sum a_i \boldy_i ,\) where each \(a_i\) is non-zero, then \(\s(\boldx) = \s(\boldy_i)\) for all \(i\).
The exact mechanics of this assignment do not concern us at the moment, but we will need to know how to compute the difference between the \(\spinc\) structures assigned to two generators.

Given   \(\boldx,\boldy \in \T_\a\cap \T_\b\),  pick  a path \(\theta\) along the \(\alpha\)'s from \(\boldx\) to \(\boldy\). More precisely, \(\theta\) is a singular 1-chain supported on the \(\alpha\)'s with \(\partial \theta = \sum_{i=1}^d y_i - \sum_{i=1}^d x_i.\)
Similarly, choose a path \(\eta\) from \(\boldx\) to \(\boldy\) along the \(\beta\)'s. The difference
\(\theta - \eta\) represents an element of \(H_1(\Sigma)\). If \(\theta'\) is a different path from \(\boldx\) to \(\boldy\) along the \(\alpha\)'s, then the difference \(\theta - \theta'\) is a linear combination of the \(\alpha\)'s in \(H_1(\S)\). Similarly, if \(\eta'\) is another path from \(\boldx\) to \(\boldy,\) then \(\eta'-\eta\) is a linear combination of \(\beta\)'s. Thus \(\theta - \eta\) represents a well defined element of  ${H_1(\Sigma)}/{L}$, where  \(L\) is the subspace spanned by the \(\alpha\)'s and \(\beta\)'s. We write
$$
\boldx-\boldy = [\theta-\eta]\in {H_1(\Sigma)}/{L} \cong H_1(M).
$$

\begin{lemma} \label{spincdiff}
\cite[Lemma 4.7]{Ju06}
We have
$ \displaystyle
\s(\boldx) - \s(\boldy) = \boldx-\boldy$ in $  H_1(M). $
\end{lemma}


Let \((M,\g)\) be a balanced sutured manifold, \(\s \in \spinc(M,\gamma)\) a relative \(\spinc\) structure, and choose a homology orientation \(\omega\) for \((M,R_-(\gamma))\). If \((\S,\bolda,\boldb)\) is a balanced diagram for \((M,\g)\), then with the absolute grading of Definition~\ref{def:SFHsign}, we have
\begin{equation*}
\chi(CF(\S,\bolda,\boldb,\s,\omega) ) = \sum_{\{\boldx \in \T_\a \cap  \T_\b  \,\colon\, \s(\boldx) = \s\}} (-1)^{b_1(M,R_-)} m(\boldx).
\end{equation*}

\begin{lemma}
If $(\S,\bolda,\boldb)$ and $(\S',\bolda',\boldb')$ both represent $(M,\g),$ then
$$\chi(CF(\S,\bolda,\boldb,\s,\omega) ) = \chi(CF(\S',\bolda',\boldb',\s,\omega) ).$$
\end{lemma}

\begin{proof}
Recall that $(\S,\bolda,\boldb)$ and $(\S',\bolda',\boldb')$ can be connected by a sequence of isotopies, handleslides,
and stabilizations/destabilizations. By \cite{Pe07}, both isotopies and handleslides correspond to isotopies of $\T_{\a}$
and $\T_{\b}$ inside $\text{Sym}^d(\S).$ If a pair of intersection points $\boldx$ and $\boldy$ of $\T_{\a}$ and $\T_{\b}$ appear/disappear during such an isotopy, then $\boldx$ and $\boldy$ can be connected by a topological Whitney disk, and hence
$\s(\boldx) = \s(\boldy).$ Invariance under stabilization/destabilization follows immediately from Lemma~\ref{lem:m}.
\end{proof}

\begin{defn}
Let $(M,\g)$ be a balanced sutured manifold, $\s \in \spinc(M,\g),$ and $\w$ a homology orientation of $(M,R_-(\g)).$ Then define $$\chi(SFH(M,\g,\s,\w)) = \chi(CF(\S,\bolda,\boldb,\s,\w)),$$ where $(\S,\bolda,\boldb)$ is any balanced diagram representing
$(M,\g).$
\end{defn}

In practice, it is convenient to combine the Euler characteristics corresponding to different \(\spinc\) structures into a single generating function, which we view as an element of the group ring \(\Z[H_1(M)]\). For this, fix an affine isomorphism \(\iota:\spinc(M,\g) \to H_1(M),\) and  let
 \begin{align} \label{eqn:1}
 \chi(SFH(M,\g,\omega))  = \sum_{\s \in \text{Spin}^c(M,\g)} \chi(SFH(M,\g,\s,\w)) [\iota(\s)]  \\ \notag = (-1)^{b_1(M,R_-)} \sum_{\boldx \in \T_\a \cap  \T_\b } m(\boldx)[\iota(\s(\boldx))].
 \end{align}
Then \(\chi(SFH(M,\g,\omega))\) is well-defined up to multiplication by an element of $H_1(M),$ viewed as a unit in \(\Z[H_1(M)]\).

%
%
%

\subsection{Duality}
Let $(M,\gamma)$ be a balanced sutured manifold, and denote by $(M,-\gamma)$  the same manifold, but with  the orientation of the suture $s(\g)$ reversed. The effect of this is to reverse the roles of  \(R_+(\g)\) and \(R_-(\gamma),\) more precisely, \(R_{\pm}(-\g) = R_{\mp}(\g)\).
The same  happens if we reverse the orientation of $M.$
In this subsection, we
show that the groups $SFH(M,\g)$ and $SFH(M,-\g)$ are isomorphic, and that they are `dual' 
to $SFH(-M,\g)$ and $SFH(-M,-\g)$. This essentially follows the same way as for ordinary Heegaard Floer homology, though it has not appeared in print before in the case of sutured Floer homology.

If $\s\in \spinc(M,\gamma)$ is represented by a nowhere vanishing vector field $v$, then we also denote by $\s$ the homology class of $v$ on $(-M,-\g).$ Furthermore, $-v$ defines a $\spinc$ structure on both $(M,-\g)$ and $(-M,\g).$ In both cases, we denote the homology class of $-v$ by $-\s.$

\begin{proposition} \label{prop:duality}
Let $(M,\gamma)$ be a balanced sutured manifold, and choose a $\spinc$ structure $\s\in \spinc(M,\gamma).$
Then  $$SFH(M,\g,\s) \cong SFH(M,-\g,-\s)$$ as relatively graded groups, and hence $$SFH(-M,-\g,\s) \cong SFH(-M,\g,-\s).$$ Moreover, $SFH(M,\g,\s)$ and  $SFH(-M,-\g,\s)$  are the homologies of dual chain complexes, so by the universal coefficient theorem
$$SFH(-M,-\g,\s) \cong \Hom(SFH(M,\g,\s),\Z) \oplus \Ext(SFH(M,\g,\s)[1],\Z).$$
\end{proposition}

\begin{proof}



Choose an admissible balanced diagram $(\S,\bolda,\boldb)$ for $(M,\g).$ Recall our orientation conventions from Section \ref{sec:ori}. The surface $\S$ is oriented such that $s(\g) = \partial\S,$ and divides $M$ into two compression bodies. The $\a$ curves bound disks in the compression body containing $R_-(\g),$ the $\b$ curves in the compression body containing $R_+(\g).$

First, consider $(-M,-\g).$ Since $s(-\g) = -s(\g) = \partial (-\S),$ the Heegaard surface is now $-\S.$
As $R_-$ and $R_+$ are the same in $(M,\g)$ and $(-M,-\g),$ the $\a$ and $\b$ compression bodies also coincide.
So $(-\S,\bolda,\boldb)$ is a balanced diagram for $(-M,-\g).$

Now flip the orientation of $M.$ For $(-M,\g),$ the surface $\S$ serves as a Heegaard surface, since the orientation of the sutures are the same. However, changing the orientation of $M$ results in flipping $R_-$ and $R_+,$ and hence the $\a$ and $\b$ compression bodies are also reversed. So $(\S,\boldb,\bolda)$ is a balanced diagram for $(-M,\g).$

Finally, in $(M,-\g),$ the orientations of the sutures are reversed, and $R_-$ and $R_+$ are also flipped. Combining the observations of the previous two paragraphs, we see that $(-\S,\boldb,\bolda)$ is a balanced diagram for $(M,-\g).$

The chain complexes $CF(\S,\bolda,\boldb),$ $CF(-\S,\bolda,\boldb),$ $CF(\S,\boldb,\bolda),$ and $CF(-\S,\boldb,\bolda)$ all have the same generators, namely $\T_{\alpha} \cap \T_{\beta}.$ Let $\boldx, \boldy $ be generators of $CF(\S,\bolda,\boldb)$ that are connected by a rigid pseudo-holomorphic Whitney disc $u \colon \mathbb{D} \to \text{Sym}^d(\S).$ Then $z \mapsto u(\overline{z})$ is a rigid pseudo-holomorphic disc connecting $\boldy$ to $\boldx$ in $CF(-\S,\bolda,\boldb).$ Furthermore, $z \mapsto u(-z)$  is a  rigid pseudo-holomorphic disk connecting $\boldy$ to $\boldx$ in $CF(\S, \boldb,\bolda).$ Finally, $z \mapsto u(-\overline{z})$  is a  rigid pseudo-holomorphic disk connecting $\boldx$ to $\boldy$ in $CF(-\S,\boldb,\bolda).$ Thus
$$CF(\S,\bolda,\boldb) = CF(-\S,\boldb,\bolda)$$ and $$CF(-\S,\bolda,\boldb) = CF(\S,\boldb,\bolda),$$ while the chain complex $CF(-\S,\bolda,\boldb)$ is dual to $CF(\S,\bolda,\boldb).$

To get the refined statement involving the $\spinc$ structures, observe that if $\boldx$ is a generator of $CF(\S,\bolda,\boldb)$ and $\boldx'$ is the corresponding generator of $CF(\S,\boldb,\bolda),$ then $\s(\boldx') = -\s(\boldx).$ On the other hand, the $\spinc$ structure assigned to $\boldx$ in $CF(\S,\bolda,\boldb)$ and in $CF(-\S,\bolda,\boldb)$ can be represented by the same vector field on $M.$
So $\spinc$ structures of corresponding generators of $CF(\S,\bolda,\boldb)$ and $CF(-\S,\boldb, \bolda)$ can be represented by opposite vector fields on $M.$
\end{proof}

\section{The definition of the torsion function}
\label{section:torsion}

In this section, we  first review the torsion of a based chain complex and the maximal abelian torsion of a pair of finite CW complexes. 
Then we define the torsion invariant for weakly balanced sutured 3--manifolds. Our approach follows closely the ideas of Turaev exposed in \cite{Tu97,Tu98,Tu01,Tu02}, see also Benedetti and Petronio \cite{BP01} for a related approach.

\subsection{Torsion of based complexes} \label{subsection:torsion}
In this subsection, we quickly recall the definition of the torsion of a based complex.
We refer to Milnor's classic paper \cite{Mi66} and Turaev's books \cite{Tu01,Tu02} for details.
Note that we follow Turaev's convention; Milnor's definition gives the multiplicative inverse of the torsion that we consider.

Throughout this section, let $\F$ be a field.
Let $V$ be a vector space over $\F,$ and  let $x=(x_1,\dots,x_n)$ and $y=(y_1,\dots,y_n)$ be two ordered bases for $V$. Then we
can write $x_i=\sum_{j=1}^na_{ij}y_j,$ and we
define $ [x/y] =  \det(a_{ij})$.

Now let
\[ 0\to C_n \xrightarrow{\partial_{n-1}} C_{n-1}\xrightarrow{\partial_{n-2}} \dots \xrightarrow{\partial_1} C_{1} \xrightarrow{\partial_0 } C_0\to 0\]
be a complex of $\F$-vector spaces. We write $H_i =H_i(C) = \ker\, \partial_{i-1}/\im\,{\partial_i}.$ For each $i,$ we pick an ordered basis $c_i$ for $C_i$ and an ordered basis $h_i$ for $H_i.$

We write $B_i=\im\{\partial_i:C_{i+1}\to C_i\},$ and we pick an ordered  basis $b_i$ for $B_i$.
Finally, we pick an ordered set of vectors $b_i'$ in $C_{i+1}$ such that $\partial_i b_{i}'=b_i$ as ordered sets.
By convention, we define $b_{-1}'$ to be the empty set.
Note that for $i=0,\dots,n,$ the ordered set $b_ih_ib'_{i-1}$ defines an ordered basis for $C_i$.
We now define the \emph{torsion} of the based complex $C$ as
\[\tau =  \prod_{i=0}^n [b_ih_ib'_{i-1}/c_i]^{(-1)^{i+1}} \in \F^*.\]
An elementary argument shows that $\tau$ does not depend on the choice of $b_0,\dots,b_{n-1},$ and it does not depend on the choice of  lifts $b_0',\dots,b_{n-1}',$ see for example \cite[Section~1]{Tu01}. Put differently, this number only depends on the choice of the  complex and the choice of the ordered bases for $C_*$ and $H_*.$ We henceforth denote this invariant by
$\tau(C_*,c_*,h_*)\in \F^*.$ If $C_*$ is acyclic; i.e., if $H_*(C)=0,$ then we just write $\tau(C_*,c_*)\in \F^{*}.$

If $c_*'$ is an ordered basis obtained from $c_*$ by swapping two basis vectors, then
it is straightforward to see that
\[ \tau(C_*,c_*',h_*)=- \tau(C_*,c_*,h_*)\in \F^*.\]

Given a chain complex $C_*$ as above, consider the numbers
\[ \ba{rcl}
 \b_i(C)&=&\sum\limits_{j=0}^i (-1)^{i-j} \dim(H_j), \\
\g_i(C)&=&\sum\limits_{j=0}^i (-1)^{i-j} \dim(C_j), \\
N(C)&=&\sum\limits_{i=0}^n \b_i(C)\g_i(C).\ea
\]
Then we define
\[ \check{\tau}(C_*,c_*,h_*)=(-1)^{N(C)} \tau(C_*,c_*,h_*)\in \F^*.\]
If $C_*$ is acyclic, then $\b_i(C)=0$ for all $i,$ and so $\check{\tau}(C_*,c_*) = \tau(C_*,c_*).$

\subsection{Lifts and Euler structures} \label{section:lifts}

Let $(X,Y)$ be a pair of finite dimensional CW complexes with $Y\subset X$ and $X$ connected. We write $H=H_1(X)$ and view $H$ as a multiplicative group.
Denote by  $\mathcal{C}$ the set of cells in $X \setminus Y.$ Let $\pi:\hat{X}\to X$ be the universal abelian cover of $X$ and write $\hat{Y}=\pi^{-1}(Y)$.

\begin{defn} \label{defn:lift}
A \emph{lift} $l$ from $(X,Y)$ to $(\hat{X},\hat{Y})$ is a choice for every $c \in \mathcal{C}$ of a cell $l(c)$ in $\hat{X}$ lying over $c.$
Note that if $l'$ is any other lift, then for every $c \in \mathcal{C}$ there is an element $g(c) \in H$ such that $l'(c) = g(c) \cdot l(c).$ We say that $l$ and $l'$ are equivalent if $$\prod_{c \in \mathcal{C}}  g(c)^{(-1)^{\text{dim}\,c}} \in H$$ is trivial. We denote the set of equivalence classes of lifts by $\lift(X,Y).$
\end{defn}

We now define an action of  $H$ on $\lift(X,Y).$ First, suppose that $X \neq Y.$
Let $h\in H$ and suppose that $\frakl\in \lift(X,Y)$ is represented by a lift $l.$ Fix an arbitrary cell $c_0 \in \mathcal{C}$ and suppose that $\dim\, c_0 = i.$ Then $h \cdot \frakl$ is represented by the lift $l'$ such that $l'(c_0) = h^{(-1)^i} \cdot l(c_0)$ and $l'(c) = l(c)$ for $c \in \mathcal{C} \setminus \{c_0\}.$ If $X = Y,$ then $|\lift(X,Y)| =1.$ Then the action of $H$ is trivial on $\lift(X,Y).$

The above definition is independent of the choice of $c_0.$ If $X \neq Y,$ then $H$ acts freely and transitively on $\lift(X,Y)$. In particular, given $\frakl_1,\frakl_2\in \lift(X,Y),$ we get a well-defined element $\frakl_1-\frakl_2\in H$.

\begin{defn} \label{defn:eul}
For each cell $c \in \mathcal{C},$ pick a point $p(c)$ in $c.$ An \emph{Euler chain} for $(X,Y)$ is a one-dimensional singular chain $\theta$ in $X$ with
\[ \partial \theta= \sum_{c \in \mathcal{C}} (-1)^{\text{dim}\,c}\,p(c).\]
Given two Euler chains $\theta,\eta,$ we define $\theta-\eta\in H$ to be the homology class of the 1-cycle $\theta-\eta$. Two Euler chains $\theta,\eta$  are equivalent if $\theta-\eta$ is trivial in $ H$. We call an equivalence class of Euler chains an \emph{Euler structure},  and denote the set of Euler structures by $\eul(X,Y)$.
\end{defn}

Note that $\eul(X,Y)\ne \emptyset$ if and only if $\chi(X,Y)=0$. Furthermore, if $\eul(X,Y)\ne \emptyset,$ then $H$ acts freely and transitively on $\eul(X,Y)$.

\begin{defn} \label{defn:E}
Suppose that $\chi(X,Y) = 0$ and $X\ne Y.$ Then we define a map $$E \colon \lift(X,Y) \to \eul(X,Y)$$ as follows. Pick a point $\hat{p} \in \hat{X}.$ Suppose that $\frakl \in \lift(X,Y),$ and choose a lift $l$ representing $\frakl.$
For every $c \in \mathcal{C},$ connect $\hat{p}$ and a point $\hat{p}(c) \in l(c)$ with an oriented path $\hat{\theta}(c)$ such that $\partial \hat{\theta}(c) = (-1)^{\text{dim}\,c}(\hat{p}(c)-\hat{p}).$ If
$\hat{\theta} = \sum_{c \in \mathcal{C}} \hat{\theta}(c),$ then $\theta = \pi(\hat{\theta})$ is an Euler chain since $\chi(X,Y) = 0.$ The Euler structure $\mathfrak{e}$ represented by $\theta$ only depends on $\frakl,$ so we define $E(\frakl) = \mathfrak{e}.$
\end{defn}

If $X \neq Y,$ then the map $E$ is an $H$-equivariant bijection. If $X = Y,$ then
$\eul(X,Y)$ is canonically isomorphic to $H,$ and the image of the unique element of $\lift(X,Y)$ under $E$ is $0 \in H.$

\subsection{Torsion of CW complexes} \label{subsection:CW}

We continue with the notation from Section~\ref{section:lifts}. In particular, let $(X,Y)$ be a pair of finite dimensional CW complexes with $Y\subset X$ and $X$ connected. We write $H=H_1(X)$ and view $H$ as a multiplicative group.
Let $\frakl\in \lift(X,Y)$ be a lift represented by $l.$
Furthermore, let $\varphi \colon \Z[H]\to \F$ be a ring homomorphism to a field $\F.$
Finally, let  $\w$  be a \emph{homology orientation}; i.e., an orientation of the vector space $H_*(X,Y;\R)=\oplus_{i\geq 0} H_i(X,Y;\R)$.
In this section, we recall the definition of the sign-refined Reidemeister--Turaev torsion  $\tau^\varphi(X,Y,\frakl,\w)\in \F$.
We refer to \cite{Tu01,Tu02} for details.

Consider the chain complex
$$C_*(X,Y;\F)=C_*(\hat{X},\hat{Y};\Z)\otimes_{\Z[H]}\F.$$
Here $H$ acts via deck transformations on $\hat{X},$ and hence on $C_*(\hat{X},\hat{Y};\Z)$, and $H$ acts on $\F$ via $\varphi$.
If this complex is not acyclic; i.e., if the  twisted homology groups $H_*(X,Y;\F)$ do not vanish, then we set $\tau^\varphi(X,Y,\frakl,\w)=0 \in \F$.
If the complex is acyclic,  then we can define the torsion $\tau^\varphi(X,Y,\frakl,\w)\in \F\sm \{0\}$ as follows.

We first pick an ordering of the cells of $X\sm Y$, and for each cell we pick an orientation.
We thus obtain an ordered  basis $c_*$ for $C_*(X,Y;\R)$.
The cells $\{\,l(c) \colon c\in \mathcal{C} \,\}$ also  define  an ordered basis of $C_*(\hat{X},\hat{Y};\Z)$ as a complex of free $\Z[H]$--modules.
This gives rise to an ordered basis $\hat{c}_*$ of $C_*(X,Y;\F)=C_*(\hat{X},\hat{Y};\Z)\otimes_{\Z[H]}\F$ in a natural way.

Finally, pick ordered bases $h_i$ for $H_i(X,Y;\R)$ with the property that
$h_0h_1\dots h_n$ is a positive basis for the oriented vector space $H_*(X,Y;\R)$.
We now consider
\[ \tau(C_*(X,Y;\F),\hat{c}_*) \cdot \sign(\check{\tau}(C_*(X,Y;\R),c_*,h_*)).\]
It follows easily from the definitions that this number is independent of the choices we made; i.e., it is independent of the choice of representative lift corresponding to $\frakl \in \lift(X,Y)$, the ordering of the cells of $X\sm Y$, their orientations and the choice of $h_*$. We refer to \cite[Section~18]{Tu01} for details.
Put differently, this number only depends on the CW pair $(X,Y)$, the lift $\frakl\in \lift(X,Y),$ the homology orientation $\w$ and the ring homomorphism $\v$.
We can thus define
\[ \tau^\varphi(X,Y,\frakl,\w):=\tau(C_*(X,Y;\F),\hat{c}_*) \cdot \sign(\check{\tau}(C_*(X,Y;\R),c_*,h_*))\in \F\sm \{0\}. \]

Now suppose that $\chi(X,Y) = 0.$ If $X \neq Y,$ then for $\mathfrak{e} \in \eul(X,Y)$ let  $$\tau^\varphi(X,Y,\mathfrak{e},\w) =
\tau^\varphi(X,Y,E^{-1}(\mathfrak{e}),\w).$$  If $X = Y,$ then recall that $\eul(X,Y)$ is canonically identified with $H.$ For $h \in H = \eul(X,Y),$ let $\tau^\varphi(X,Y,h,\w) = h$ if $\omega$ is the positive orientation of $H_*(X,Y;\R) = 0,$ and $\tau^\varphi(X,Y,h,\w) = -h$ otherwise.

\subsection{The maximal abelian torsion of a CW complex}\label{sec:MaxTorsion}
We continue with the notation from the previous sections. In particular, let $(X,Y)$ be a pair of finite CW complexes such that $X$ is connected. Furthermore, let $\frakl\in \lift(X,Y)$ and $\w$ a homology orientation.

Again, we write $H=H_1(X)$ and think of $H$ as a multiplicative group. We let $T=\tor(H)$ be the torsion subgroup.
Given a ring $R,$ we denote by $Q(R)$ the ring which is given by inverting all elements of $R$ that are not zero divisors. We write $Q(H)=Q(\Z[H]).$

A character $\chi:T\to \C^*$ extends to a ring homomorphism ${\chi}:\Q[T]\to \C$, its image is a cyclotomic field $\F_\chi$.
Two characters $\chi_1,\chi_2$ are called equivalent if $\F_{\chi_1}=\F_{\chi_2}$ and if ${\chi}_1$ is the composition of ${\chi}_2$ with a Galois automorphism of $\F_{\chi_1}$ over $\Q$. For any complete family of representatives $\chi_1,\dots,\chi_n$ of the set of equivalence classes of characters, the homomorphism
\[ ({\chi}_1,\dots,{\chi}_n):\Q[T]\to \bigoplus_{i=1}^n \F_{\chi_i}\]
is an isomorphism of rings.
We will henceforth identify $\Q[T]$ with $\bigoplus_{i=1}^n \F_{\chi_i}$.
Note that under this isomorphism $1\in \Q[T]$ corresponds to $(1,\dots,1)$.

Now let $F$ be the free abelian group $H/T,$ and pick a splitting $H=F\times T$.
Then we have the identifications
\[ \Q[H]=\bigoplus_{i=1}^n \F_{\chi_i}[F] \mbox{ and } Q(H)=\bigoplus_{i=1}^n Q(\F_{\chi_i}[F]).\]
We denote by $\varphi_i$ the ring homomorphism $$\Z[H]\to\bigoplus_{i=1}^n  \F_{\chi_i}[F]\to  \F_{\chi_i}[F]\to Q( \F_{\chi_i}[F]).$$
Following Turaev, we now let
\[ \tau(X,Y,\frakl,\w) =\sum_{i=1}^n \tau^{\varphi_i}(X,Y,\frakl,\w)\in \bigoplus_{i=1}^n Q( \F_{\chi_i}[F])=Q(H).\]
Note that $\tau(X,Y,\frakl,\w)\in Q(H)$ is independent of the choices we made, cf. \cite[Section~K]{Tu02}.
Also note that
\be \label{equ:htau} \tau(X,Y,h\cdot \frakl ,\pm \w)=\pm h\cdot \tau(X,Y,\frakl,\w).\ee
In the following, we write $\tau(X,Y,\frakl)$ for the set of torsions corresponding to all possible homology orientations. Also, if $\chi(X,Y) = 0$
and $X\ne Y$, then for $\mathfrak{e} \in \eul(X,Y)$ we define
$$\tau(X,Y,\mathfrak{e},\w) = \tau(X,Y,E^{-1}(\mathfrak{e}),\w).$$
If $X=Y$ and $\mathfrak{e}\in \eul(X,Y)$ corresponds to $h\in H$ under the canonical isomorphism $\eul(X,Y)=H$, then
we define $\tau(X,Y,\mathfrak{e},\w)=\pm h$, depending on whether $\w$ is the positive or negative orientation of the zero space.
Finally, if $\chi(X,Y)\ne 0$, then we set
$\tau(X,Y,\mathfrak{e},\w) = 0$.

In the coming sections, we will often make use of the following two lemmas.

\begin{lemma} \label{lem:simple}
Let $(X,Y)$ be a pair of finite CW--complexes, and let $\w$ be a homology orientation for $(X,Y).$
Assume that $(X',Y')$ is a CW--pair obtained from $(X,Y)$ by a simple homotopy $s.$
Since $s_* \colon H_*(X,Y;\R) \to H_*(X',Y';\R)$ is an isomorphism, the orientation $\w$ of $(X,Y)$ induces an orientation $\w' = s_*(\w)$ of $(X',Y').$
Then there exists an $H$--equivariant bijection
$b_s \colon \eul(X,Y)\to \eul(X',Y')$ such that for every $\mathfrak{e} \in \eul(X,Y)$ $$s_*\left(\tau(X,Y,\mathfrak{e},\w)\right) = \tau(X',Y',b_s(\mathfrak{e}),\w').$$
\end{lemma}

\begin{proof}
It is sufficient to show the result if $X'$ is obtained from $X$ using an elementary expansion.
Suppose that we added an $i$-cell $c$ and an ($i+1$)-cell $d$ to $X$ to get $X'.$ Choose an Euler chain $\theta$ representing $\mathfrak{e}.$ Let $\delta \subset c \cup d$ be a curve such that $\partial \delta = (-1)^i(p-q),$ where $p$ is the center of $c$ and $q$ is the center of $d.$ Then define $b_s(\mathfrak{e})$ to be the equivalence class of the Euler chain $s(\theta) + \delta.$ From here, a standard argument shows that $s_*\left(\tau(X,Y,\mathfrak{e},\w)\right) = \tau(X',Y',b_s(\mathfrak{e}),\w').$
\end{proof}

In practice, we are mostly interested in the simple case where \(C_i(X,Y) = 0\) for \(i \neq 1,2\).
We will explain how to compute the torsion in this situation, but first we digress to discuss homology orientations.

Let \(C_*\) be a chain complex with \(C_i = 0\) for \(i \neq 1,2\)  and \(\chi(C_*)=0\).
Suppose we are given a homology orientation \(\omega\) for \(H_*(C_* \otimes \R)\).  As in Section~\ref{subsection:torsion}, we let \(h_1 \) and \(h_2\) be ordered bases of \(H_1(C_* \otimes \R)\) and \(H_2(C_* \otimes \R)\) compatible with \(\omega.\)
Furthermore, choose an ordered basis $b_1$ for $\im(\partial_1) \le C_1 \otimes \R,$ and an ordered set of chains \(b_1'\) in \(C_2 \otimes \R\) such that $b_1 = \partial_1 b_1'.$ Then \({h}_2b_1'\) is a basis of \(C_2 \otimes \R.\)

\begin{defn}
Let \(c_1\) and $c_2$ be ordered bases of \(C_1\) and \(C_2,\) respectively.
We say  \(c_1\) and $c_2$
are compatible with \(\omega\) if they induce the same orientation of \(C_*\otimes \R\) as the bases \({h}_1 b_1\) and \({h}_2b_1'.\)
\end{defn}
A standard argument shows that the notion of a compatible basis depends only on the homology orientation \(\omega\), and not on the various choices involved in the definition.

 \begin{lemma}
 \label{lem:taudet}
Suppose that $X$ is a 2--complex, that $Y \neq \emptyset$, and that \(\chi(X, Y) = 0.\) Denote by $\mathcal{C}$ the set of cells in $X \setminus Y.$ Fix a homology orientation $\w$ of $(X,Y).$ Furthermore, choose a lift \(\mathfrak{l} \in \lift(X,Y),\) together with a representative $l.$ Let $c_*$ be an ordering of $\mathcal{C}$ such that $l_* = l(c_*)$ is a basis of the free $\Z[H]$-module $C_*(\hat{X},\hat{Y};\Z)$ compatible with the homology orientation \(\omega,\) and let \(A\) be the matrix representing the boundary map
$$C_2(\hat{X},\hat{Y};\Z)\to C_1(\hat{X},\hat{Y};\Z)$$
with respect to the basis $l_*.$
Then $$\tau(X,Y,\mathfrak{l},\omega) = (-1)^{b_1(X,Y)}\det A.$$
 \end{lemma}

\begin{proof}
The space $X$ is connected and $Y \neq \emptyset,$ so after collapsing 1-cells and using Lemma~\ref{lem:simple}, we can assume without loss of generality that $X\sm Y$ contains no $0$-cells. Our assumptions then imply that $C_i(\hat{X},\hat{Y};\Z)=0$ for any $i\ne 1,2$. For any ring homomorphism $\varphi:\Z[H]\to \F,$ the basis $l_*$ of $C_*(\hat{X},\hat{Y};\Z)$ induces a basis $\hat{c}_*$ of $C_*(X,Y;\F),$ and we have
$$\tau(C_*(X,Y;\F),\hat{c}_*)=\det(\varphi(A))=\varphi(\det A).$$
Consider $c_*$ as a basis of $C_* = C_*(X,Y;\R).$ Furthermore, let \(h_*\) be a basis of \(H_* = H_*(X,Y;\R)\) compatible with \(\omega.\) Note that $\text{dim}\, C_1 = \text{dim}\, C_2$ as $\chi(X,Y) = 0,$ so $N(C) \equiv (\dim\,H_1) (\dim \, C_1)$ modulo $2.$ Then
\begin{align*}
 \sign \ \check{\tau}(C_*(X,Y;\R),c_*,h_*) & = (-1)^{N(C)} \sign[b_1h_1/c_1] \cdot \sign[h_2b_1'/c_2] \\
 & = (-1)^{(\text{dim} H_1) (\text{dim} C_1) }  \sign[b_1h_1/c_1] \cdot \sign[h_2b_1'/c_2] \\
 & =(-1)^{(\text{dim} H_1)^2}\sign[h_1b_1/c_1] \cdot \sign[h_2b_1'/c_2]  \\
 & = (-1)^{\text{dim} H_1}\sign[h_1b_1h_2b_1'/c_1c_2] \\ &  = (-1)^{\text{dim} H_1},
\end{align*}
since \(\hat{c}_*\) is compatible with \(\omega.\)
Thus $ \tau^\varphi(X,Y,\mathfrak{l},\omega) = (-1)^{\text{dim} H_1} \varphi(\det A)\) for all \(\varphi\).
The lemma now follows immediately from the definition of the maximal abelian torsion.
\end{proof}

\begin{corollary}\label{cor:taupolynomial}
Suppose that $X$ is a 2--complex and that $Y \neq \emptyset.$
Then for any $\frakl \in \lift(X,Y)$ we have $\tau(X,Y,\frakl)\in \Z[H]$.
\end{corollary}

\subsection{Torsion for sutured manifolds}

Throughout this section, $(M,\g)$ will be a connected weakly balanced sutured manifold.

\begin{defn}
A \emph{sutured handle complex} $\A$ is a triple $(A,S \times I,\E),$ where
\begin{enumerate}
\item $A$ is a compact oriented three-manifold with boundary,
\item $S \times \{0\} \subset \partial A$ is a compact subsurface with boundary,
\item $S \times I$ is a submanifold of $A$ such that $\partial S \times I \subset \partial A,$ and
\item $\E = (e_1, \dots, e_n)$ is a decomposition of $A \setminus (S \times I)$ into 3-dimensional handles $e_1, \dots, e_n,$ attached to the top of $S \times I$ one after another.
\end{enumerate}
More precisely, for $0 \le i \le n,$ we write $A_i = (S \times I) \cup (e_1 \cup \dots \cup e_i)$ and $S_i = \partial A_i \setminus (S \times \{0\} \cup \partial S \times I).$ Then the handle $e_i$ is smoothly attached to $A_{i-1}$ along $S_{i-1}$ via a gluing map $f_i.$

Let $I(r)$ denote the index of the handle $e_r.$ We say that $\mathcal{A}$ is \emph{nice} if $I$ is non-decreasing, and $e_i \cap e_j = \emptyset $ whenever   $I(i) = I(j)$ and $i \neq j.$
A \emph{homology orientation} for $\A$ is an orientation of the vector space $H_*(A,S \times \{0\};\R) = H_*(A, S \times I;\R).$
\end{defn}


Notice that if $\A = (A,S\times I,\E)$ is a sutured handle complex, then $(A,\partial S \times I)$ is a sutured manifold.
Given a sutured handle complex $\A,$ we can define $\lift(\A)$ and $\eul(\A)$
in a way completely analogous to the case of a $CW$ pair $(X,Y).$ Just use the following dictionary:

\begin{center}
\begin{tabular}{|l|l|}
\hline
CW pairs & sutured handle complexes \\
\hline
$X$ & $A$ \\
$Y$ & $S \times I$ \\
$\mathcal{C}$ & $\E$ \\
cell & handle \\
\hline
\end{tabular}
\end{center}


For the reader's convenience, we translate Definitions~\ref{defn:lift} and \ref{defn:eul} to sutured handle complexes. Given a sutured handle complex $\A = (A,S \times I, \E),$ let $\pi \colon \hat{A} \to A$ be the universal abelian cover of $A,$
and we write $\widehat{S \times I} = \pi^{-1}(S \times I).$ Suppose that $\E = (e_1,\dots, e_n).$

\begin{defn}
A \emph{lift} $l$  is a choice for every $1 \le i \le n$ of a handle $l(i)$ in $\hat{A}$ lying over $e_i.$ Note that if $l'$ is any other lift, then for every $1 \le i \le n$ there is an element $g(i) \in H_1(A)$ such that $l'(i) = g(i) \cdot l(i).$ We say that $l$ and $l'$ are equivalent if $$\prod_{i = 1}^n  g(i)^{(-1)^{I(i)}} \in H_1(A)$$ is trivial. We denote the set of equivalence classes of lifts by $\lift(\A).$
\end{defn}

\begin{defn}
For $1 \le i \le n,$ pick a point $p(i)$ in $e_i.$ An \emph{Euler chain} for $\A$ is a one-dimensional singular chain $\theta$ in $A$ with
\[ \partial \theta= \sum_{i = 1}^n (-1)^{I(i)}\,p(i).\]
Given two Euler chains $\theta,\eta,$ we define $\theta-\eta\in H_1(A)$ to be the homology class of the 1-cycle $\theta-\eta$. Two Euler chains $\theta,\eta$  are called equivalent if $\theta-\eta$ is trivial in $ H_1(A)$. We call an equivalence class of Euler chains an \emph{Euler structure},  and denote the set of Euler structures by $\eul(\A).$
\end{defn}

From now on, suppose that $A$ is connected and $\chi(A,S \times I) = \chi(A,S \times \{0\}) = 0.$ As in Definition~\ref{defn:E}, we also have a map $E_{\A} \colon \lift(\A) \to \eul(\A).$ Finally, if $\w$ is an orientation of $H_*(A,S \times I; \R)$ and $\mathfrak{e} \in \eul(\A),$ then we define the maximal abelian torsion $\tau(\A,\mathfrak{e},\w) \in  Q(H_1(A)) = Q(\Z[H_1(A)])$ using the chain complex $C_*(\hat{A},\widehat{S \times I};\Z)$ arising
from the handle decomposition $\E.$

\begin{remark} \label{rem:CW}
The reader should keep in mind that $\tau(\A,\mathfrak{e},\w)$ is a \emph{relative} torsion corresponding to the pair of spaces $(A, S \times I).$ Notice that if $\A$ is nice, then we can collapse each handle to its core, starting from $e_n$ and proceeding to $e_1,$ and finally $S \times I$ to $S \times \{0\},$ to obtain a relative CW complex built upon $S \times \{0\}.$ Extend this to a CW decomposition $Y$ of $S \times \{0\}$ such that we obtain a CW pair $(X,Y).$ Then there are canonical bijections $\lift(\A)=\lift(X,Y)$ and $\eul(\A)=\eul(X,Y)$ such that given $\mathfrak{e} \in \eul(\A)=\eul(X,Y)$ and a homology orientation $\w,$ we have  $\tau(\A,\mathfrak{e},\w)=\tau(X,Y,\mathfrak{e},\w)$.
\end{remark}

\begin{defn}\label{def:handledec}
A \emph{handle decomposition} $Z$ of a sutured manifold $(M,\g)$ consists of a sutured handle complex $\mathcal{A} = (A, S \times I, \E)$ and a diffeomorphism $d \colon A \to M$ such that $d(S\times \{0\}) = R_-(\gamma)$ and $d(\partial S \times I) = \g.$ We say  that $Z$ is \emph{nice} if $\mathcal{A}$ is nice.
\end{defn}

The set $\eul(\A)$ is an affine copy of $H_1(A),$ while $\spinc(M,\gamma)$ is an affine copy of $H_1(M).$ We will show that there
is a canonical isomorphism between them corresponding to $d_* \colon H_1(A) \to H_1(M).$

\begin{claim} Given a handle decomposition $Z = (\A,d)$ of the sutured manifold $(M,\g),$ let $v_0$ be a vector field along $\partial M$ as in Subsection~\ref{section:spinc}. Then there exists a vector field $v_A$ on $A$ with the following properties: $v_A$ vanishes exactly at the centers of $e_1,\dots,e_n;$ the index of $v_A$ at the center of $e_i$ is $(-1)^{I(i)};$ finally, $v_A|\partial A = d^{-1}_*(v_0).$
\end{claim}

\begin{proof}
First, let $v_A|(S \times I)= \partial/\partial t,$ where $t$ is the coordinate on $I.$ If $e_i = D^3$ is a 0--handle with coordinates $(x,y,z),$ then let $$v_A(x,y,z) = x \cdot \frac{\partial}{\partial x} + y \cdot \frac{\partial}{\partial y} + z \cdot \frac{\partial}{\partial z}.$$ On a 3--handle, take $v_A$ to be the negative of the previous vector field. If $e_j = D^1 \times D^2$ is a 1--handle with coordinates $x$ on $D^1$ and $(y,z)$ on $D^2,$ then we define $v_A(x,y,z)$ to be $$-x \cdot \frac{\partial}{\partial x} + y \cdot \frac{\partial} {\partial y} + z \cdot \frac{\partial}{\partial z}.$$ This vector field points out of $e_j$ along $D^1 \times S^1$
and it lies on the same side of $S^0 \times D^2$ as $\{0\} \times D^2.$
We can glue $v_A|(D^1 \times D^2)$ to $v_A|(S \times I)\cup (\text{0--handles})$ by a smooth handle attachment. Similarly, on a 2--handle $D^2 \times D^1$ with coordinates $(x,y,z),$ we choose $v_A$ to be $-x \cdot \partial /\partial x - y \cdot \partial /\partial y + z \cdot \partial /\partial z.$ One can also think of $v_A$ as the gradient-like vector field for a Morse function compatible with the handle decomposition $\E.$
This concludes the proof of the claim.
\end{proof}

We will use the notation $v_Z$ for the vector field on $M$ obtained by pushing forward $v_A$ along $d.$

\begin{defn} \label{defn:1}
Let $Z = (\A,d)$ be a relative handle decomposition of the sutured manifold $(M,\g).$ For $\mathfrak{e} \in \eul(\A),$ let $s_Z(\mathfrak{e}) \in \spinc(M,\g)$ be the homology class of the vector field $d_*(v)$ that is obtained as follows. Pick an Euler chain $\theta$ representing $\mathfrak{e}$ that is a union of pairwise disjoint smoothly embedded arcs and circles inside $\text{Int}(A).$ Choose an open regular neighborhood $N(\theta)$ of $\theta.$ Then let $v = v_A$ on $A \setminus N(\theta).$ If $N_0$ is a component of $N(\theta)$ diffeomorphic to $B^3,$ then extend $v$ to $N_0$ as a nowhere zero vector field. This is possible since $v_A$ has exactly one index $1$ and one index $-1$ singularity inside $N_0.$ The homology class of $v$ is independent of the choice of extension. If $N_1$ is a component of $N(\theta)$ diffeomorphic to $S^1 \times B^2,$ then we get $v|N_1$ from $v_A|N_1$ using Reeb turbularization, as described in \cite[p.639]{Tu90}. Finally, we push forward $v$ along $d$ to obtain $d_*(v).$
\end{defn}

Note that we can avoid closed components of $\theta$ in the above definition except if $(M,\g)$ is a product and $A = S \times I.$

\begin{lemma} \label{lem:sz}
Let $Z = (\A,d)$ be a handle decomposition of $(M,\g).$ For $\mathfrak{e}_1,\mathfrak{e}_2 \in \eul(\A),$ we have $$s_Z(\mathfrak{e}_1) - s_Z(\mathfrak{e}_2) = d_*(\mathfrak{e}_1 - \mathfrak{e}_2).$$
\end{lemma}

\begin{proof}
An analogous obstruction theoretic argument as in the proof of \cite[Lemma 2.19]{OS04a} works here too. Also see \cite{Tu90}.
\end{proof}

Consequently, $s_Z$ gives an isomorphism between the affine spaces $\eul(\A)$ and $\spinc(M,\g).$

\begin{remark} \label{rem:spinc}
Note that if the handle decomposition $Z$ arises from a balanced diagram $(\S,\bolda,\boldb),$ then every $\boldx = (x_1, \dots, x_d) \in \T_\a \cap \T_\b$ defines a unique Euler structure $\mathfrak{e}(\boldx) \in \eul(Z)$ as follows. Suppose that $x_i \in \alpha_j \cap \beta_k,$ and let $a_j$ be the 1--handle corresponding to $\alpha_j$ and $b_k$ the 2--handle corresponding to $\beta_k.$ Then let $\theta_i$ be a curve that connects the center of $a_j$ to $x_i$ inside $a_j$
and then goes from $x_i$ to the center of $b_k$ inside $b_k.$ The Euler chain $\theta_1 + \dots + \theta_d$ defines $\mathfrak{e}(\boldx).$ Then the $\spinc$-structure $\s(\boldx)$ assigned to $\boldx$ is exactly $s_Z(\mathfrak{e}(\boldx)).$
\end{remark}


\begin{proposition} \label{prop:1}
Suppose that $Z = (\A,d)$ and $Z' = (\A',d')$ are nice handle decompositions of the connected weakly balanced sutured manifold $(M,\gamma).$ Suppose that $\w$ and $\w'$ are homology orientations for $\A$ and $\A',$ respectively, such that $d_*(\w) = d'_*(\w').$ Furthermore, pick a $\spinc$ structure $\s \in \spinc(M,\g),$ and write $\e = s_Z^{-1}(\s)$ and $\e' = s_{Z'}^{-1}(\s).$ Then
$$d_*(\tau(\A,\e,\w)) = d_*'(\tau(\A',\e',\w')) \in Q(H_1(M)).$$
\end{proposition}

\begin{proof}
By \cite{Ce70}, one can get from $Z$ to $Z'$ through a finite sequence of nice handle decompositions, each one obtained from the previous by one of the following basic operations:
\begin{enumerate}
\item an isotopy of $d$ through diffeomorphisms $d_t \colon A \to M$ such that $d_t(S \times \{0\}) = R_-(\g)$ and $d_t(\partial S \times I) = \g,$
\item an isotopy of a handle,
\item a handle slide,
\item or adding/cancelling a pair of handles.
\end{enumerate}
So it suffices to show the claim when $Z$ and $Z'$ are related by one of these operations.

Suppose that $\mathcal{A} = \mathcal{A'}$ and $d$ is isotopic to $d',$ as in (1).
Then we claim that $\e = \e'.$
Indeed, pick a submanifold $\theta \subset A$  that represents $\e.$ As in Definition~\ref{defn:1}, let $v$ be the vector field obtain by modifying $v_A$ along $\theta.$ Then $s_Z(\e)$ is the homology class of the vector field $d_*(v),$ while $s_{Z'}(\e)$ is represented by $d'_*(v).$ Since $d$ is isotopic to $d'$ through diffeomorphisms $d_t \colon A \to M$ such that $d_t(S \times \{0\}) = R_-(\g)$ and $d_t(\partial S \times I) = \g,$ the vector fields $d_*(v)$ and $d_*'(v)$ are homotopic through vector fields that satisfy the appropriate boundary conditions along $\partial M.$ In particular, $d_*(v)$ and $d_*'(v)$ are homologous, hence $s_{Z'}(\e) = s_Z(\e) = \s.$ From this, $\e = s_{Z'}^{-1}(\s) = \e',$ as claimed. Since $d$ and $d'$ are isotopic, they induce the same maps from $H_*(A,S \times \{0\};\R)$ to $H_*(M,R_-(\g);\R),$ so $\w = \w'.$ Similarly, $d_* = d_*' \colon H_1(A) \to H_1(M).$ Consequently,
$$d_*(\tau(\A,\e,\w)) = d_*'(\tau(\A',\e',\w')).$$

Now we consider cases $(2)$ and $(3).$ Actually, $Z$ and $Z'$ are related by isotoping a handle in both cases. This means the following. We choose a handle $e_i$ of $\mathcal{A}$ and isotope its attaching map $f_i$ to some other map $f_i'$ inside $S_{i-1}.$  Then we extend this isotopy to a diffeotopy $\{\,\varphi_t \colon t \in I\,\}$ of $S_{i-1}$ such that $f_i'= \varphi_1 \circ f_i.$ We define the $A_j'$ recursively, together with diffeomorphisms $\nu_j \colon A_j \to A_j'$ for $j \ge i-1.$  If $j \le i-1,$ then let $A_j' = A_j.$ To define $\nu_{i-1},$ choose a collar $S_{i-1} \times I \subset A_{i-1}$ of $S_{i-1}$ such that $S_{i-1} \times \{1\}$ is identified with $S_{i-1}.$ For $x \in A_{i-1} \setminus (S_{i-1} \times I),$ let $\nu_{i-1}(x) = x,$ and set $\nu_{i-1}(s,t) = (\varphi_t(s),t)$ for $(s,t) \in S_{i-1} \times I.$ If $A_{j-1}'$ and $\nu_{j-1}$ are already defined, then we obtain $A_j'$ by gluing $e_j' = e_j$ to
$S_{j-1}'$ along $\nu_j \circ f_j.$ Then $\nu_{j-1}$ naturally extends to $A_j,$ call this extension $\nu_j.$ This defines the handle complex $\mathcal{A}',$ and we set $d' = d \circ (\nu_n)^{-1},$ where $n$ is the number of handles.

Define $\nu = \nu_n,$ this is a diffeomorphism from $A$ to $A'.$ For $1 \le j \le n,$ let $p_j$ be the center of $e_j$ and $p_j'$ the center of $e_j'.$ Then $\nu(p_j) = p_j'.$ Hence $\nu$ induces a natural bijection $N \colon \eul(\A) \to \eul(\A')$ via the formula $N([\theta]) = [\nu(\theta)],$ where $\theta$ is an Euler chain in $A.$  We claim that $$s_Z(\mathfrak{e}) = s_{Z'}(N(\mathfrak{e})).$$ Indeed, suppose that $[\theta] = \e$ for some Euler chain $\theta \subset A.$ Then $N(\e)$ is represented by the Euler chain $\nu(\theta) \subset A'.$ Let $v$ be the vector field obtained by modifying $v_A$ along $\theta.$
The vector field $\nu_*(v_A)$ agrees with $v_{A'},$ except on $\nu(S_{i-1}\times I) = S_{i-1} \times I$ (recall that $A_{i-1}' = A_{i-1}$). But along $S_{i-1} \times I$ they both point up with respect to $\partial/\partial t,$  so $\nu_*(v_A)$ and $v_{A'}$ are isotopic on this collar through nowhere zero fields rel boundary.
Hence, if we modify $v_{A'}$ along $\nu(\theta),$ we obtain a vector field $v'$ isotopic to $\nu_*(v)$ rel $\partial A'.$ So
$s_Z(\e) = [d_*(v)],$ while $s_Z(N(\e)) = [d'_*(v')] = [d'_*(\nu_*(v))] = [d_*(v)],$ which proves the claim.
Consequently, $$\e' = s_{Z'}^{-1}(\s) = s_{Z'}^{-1}(s_Z(\e)) = N(\e).$$

Let $\pi \colon \hat{A} \to A$ and $\pi' \colon \hat{A}' \to A'$ be the maximal abelian covers of $A$ and $A',$ respectively. Fix an arbitrary lift $\hat{\nu} \colon \hat{A} \to \hat{A}'$ of the diffeomorphism $\nu \colon A \to A'.$
Let $\theta$ be an Euler chain representing $\e$ which is a submanifold of $A.$ We can assume that $\theta$ has no closed components. Then $\nu(\theta)$ represents $\e' = N(\mathfrak{e}).$ Let $\hat{\theta}$ be an arbitrary lift of $\theta$ to $\hat{A}.$ Then $\partial \hat{\theta}$ defines a lift $l$ that represents $E_{\A}^{-1}(\e).$ Similarly, $\hat{\nu}(\hat{\theta})$ is a lift of $\nu(\theta),$ hence $\hat{\nu}(\partial \hat{\theta})$ gives a lift $l'$ representing $E_{\A'}^{-1}(\e').$
The handle decomposition $\E$ of $A \setminus (S \times I)$ gives rise to a handle decomposition $\hat{\E}$ of
$\hat{A} \setminus (\widehat{S \times I}).$ Similarly, $\E'$ induces a handle decomposition $\hat{\E}'$ of $\hat{A}' \setminus (\widehat{S' \times I}).$ We obtain chain complexes $C_*(\hat{A},\widehat{S \times I})$ and $C_*(\hat{A}', \widehat{S' \times I})$
generated by $\hat{\E}$ and $\hat{\E}',$ respectively.
Since $\nu(p_j) = p_j'$ for $1 \le j \le n,$ we see that $\hat{\nu}$ maps the centers of the handles of $\hat{\E}$ to the centers of the handles of $\hat{\E'}.$ So $\hat{\nu}$ induces a bijection $B \colon C_*(\hat{A},\widehat{S \times I}) \to C_*(\hat{A}', \widehat{S' \times I})$ that takes the generator $l(j)$ to $l'(j).$

We are now ready to prove invariance under move (2). So suppose that the isotopy connecting $f_i$ and $f_i'$ avoids every other handle of index $I(i),$ including the attaching map of $e_j$ if $j > i$ and  $I(j) = I(i).$ Then $B$ is an isomorphism of based complexes. Indeed, isotoping the attaching map of $l(i)$ does not change the algebraic intersection number with belt circles of handles of index $I(i)-1.$ More precisely, if $I(j) = I(i)-1,$ then we have the equality of matrix coefficients
$$\langle\,\partial l(i),l(j)\,\rangle = \langle\, \partial l'(i), l'(j) \,\rangle.$$
So $\nu_*(\tau(\A,\e,\w)) =  \tau(\A',\e',\w').$ Together with $d = d' \circ \nu,$ this implies the result in case (2).

Now we consider operation (3). Suppose that we handleslide $e_i$ over a handle $e_r$ such that $I(i) = I(r).$
Consider the new basis $\underline{b} = (b_1, \dots, b_n)$ of $C_*(\hat{A},\widehat{S \times I})$ given by $b_j = l(j)$ when $j \neq i,$ and $b_i = l(i) + l(r).$ Note that this change of basis does not change the torsion.
Then $B$ is an isomorphism of based complexes from $C_*(\hat{A},\widehat{S \times I})$ with the basis $\underline{b}$ to  $C_*(\hat{A}', \widehat{S' \times I})$ with the basis $(l'(1),\dots,l'(n)).$ The result follows for case (3).

Finally, we consider operation (4). So suppose that $Z'$ is obtained from $Z$ by adding a canceling pair of handles $e$ and $f$ between $e_{i-1}$ and $e_i$ such that $Z'$ is also nice (this is not necessarily the case if $I(i)=1,$ but that can be avoided by isotoping the 2-handles beforehand). Similarly to the case of an isotopy, we recursively define $\mathcal{A}',$ together with a sequence of diffeomorphisms $\nu_j \colon A_j \to A_{j+2}'$ for $j \ge i-1.$
Let $A_j' = A_j$ for $j \le i-1.$ To define $\nu_{i-1},$ choose a collar $S_{i-1} \times I \subset A_{i-1}$ of $S_{i-1}$ as before. For $x \in A_{i-1} \setminus (S_{i-1} \times I),$ let $\nu_{i-1}(x) = x,$ while $$\nu_{i-1}(S_{i-1} \times I)
= (S_{i-1} \times I) \cup e \cup f.$$ We set $e_i' = e$ and $e_{i+1}' = f.$ If $j \ge i,$ and if the attaching map of $e_j$ in $\mathcal{A}$ is $f_j,$ then let $e_{j+2}' = e_j$ attached to $A_{j-1}$ along $\nu_{j-1} \circ f_j.$ Finally, set $\nu =\nu_n$ and define $d' = d \circ \nu^{-1}.$

The CW complexes corresponding to $\mathcal{A}$ and $\mathcal{A}'$ are related by an elementary expansion $s$, so by Lemma~\ref{lem:simple} there is a bijection $b_s \colon \eul(\A) \to \eul(\A')$ such that $$s_*\left(\tau(\A,\e,\w)\right) = \tau(\A',b_s(\e),\w')$$ for every $\mathfrak{e} \in \eul(\A),$ where $\w' = s_*(\w).$ We claim that $s_Z(\mathfrak{e}) = s_{Z'}(b_s(\mathfrak{e})).$ Indeed, let $p_j'$ denote the center of $e_j'.$ If the Euler chain $\theta$ represents $\mathfrak{e}$ and $\delta$ is an arc inside $e_i' \cup e_{i+1}'$ such that $\partial \delta = (-1)^i(p_i' - p_{i+1}'),$ then $\theta' = \nu(\theta) + \delta$ represents $b_s(\mathfrak{e}).$ Let $$K = (S_{i-1} \times I) \cup e_i' \cup e_{i+1}' \subset A',$$ then $K = \nu(S_{i-1} \times I).$ If $v'$ is a nowhere vanishing vector field on $A'$ extending $v_{A'}|(A' \setminus N(\theta')),$ then $v'|K$ is isotopic to $\nu_*(\partial/\partial t).$ On the other hand, if $v$ is the nowhere zero vector field obtained from $v_A$ and $\theta,$ then $v|(S_{i-1} \times I) = \partial /\partial t.$ Furthermore, $$\nu_*\left(v|(A \setminus (S_{i-1} \times I))\right) = v'|(A' \setminus K).$$
So $\nu_*(v)$ is homotopic to $v'$ rel $\partial A,$ hence
$$s_Z(\e) = [d_*(v)] = [d'_* \circ \nu_*(v)] = [d'_*(v')] = s_{Z'}(b_s(\e)).$$
This implies that $\e' = b_s(\e),$ and the claim follows for case (4), concluding the proof of Proposition~\ref{prop:1}.
\end{proof}

\begin{defn} \label{defn:2}
Let \(Z = (\A,d)\) be a nice handle decomposition of the connected weakly balanced sutured manifold $(M,\g).$ Given an element  \(\s \in \spinc(M,\g)\) and a homology orientation \(\w\) for the pair \((M,R_-(\gamma))\), we define
\begin{equation*}
\tau(M,\g,\s,\w) = d_*\left(\tau\left(\A,s_Z^{-1}(\s),d_*^{-1}(\w)\right)\right) \in Q(H_1(M)),
\end{equation*}
where $s_Z$ is the identification between $\eul(\A)$ and $\spinc(M,\g)$ defined in Definition~\ref{defn:1}. By Proposition~\ref{prop:1}, the torsion $\tau(M,\g,\s,\w)$ is independent of the choice of $Z.$
\end{defn}

We would like to emphasize that $Z$ is a relative handle decomposition, hence $\tau(M,\g,\s,\w)$ is essentially the torsion
of the pair $(M,R_-(\g)).$
We now extend Definition~\ref{defn:2} to disconnected weakly balanced sutured manifolds.

\begin{defn}
Suppose that $(M,\g)$ is a weakly balanced sutured manifold whose components are $(M_1,\g_1), \dots, (M_n,\g_n).$ Fix a homology orientation $\omega$ for $(M,R_-(\gamma))$ and a $\spinc$ structure $\s \in \spinc(M,\g).$ For $1 \le i \le n,$ let $\s_i = \s|M_i$ and $\omega_i= \omega|M_i.$ Then define $$\tau(M,\g,\s,\w) = \bigotimes_{i=1}^n \tau(M_i,\g_i,\s_i,\w_i) \in \bigotimes_{i=1}^n Q(H_1(M_i)) \subset Q(H_1(M)),$$ where we take the tensor product over $\Z.$
\end{defn}

Now suppose that $R_+(\g) \cap M_i \neq \emptyset$ and $R_-(\gamma) \cap M_i \neq \emptyset$ for every $1 \le i \le n.$ This is true for example if $(M,\g)$ is balanced. Then $(M,\g)$ has a handle decomposition with no 0 and 3--handles. From Corollary~\ref{cor:taupolynomial}, we get that $\tau(M,\gamma,\s,\w)\in \Z[H_1(M)].$
Following Turaev, we define the torsion function as
\[ \ba{rcl} T_{(M,\gamma,\w)}:\spinc(M,\gamma)&\to & \Z\\
\s &\mapsto &\tau(M,\gamma,\s,\w)_1\ea \]
where $\tau(M,\gamma,\s,\w)_1$ denotes the constant term of
$\tau(M,\gamma,\s,\w)\in \Z[H_1(M)]$.
Note that in light of (\ref{equ:htau}), we can recover $\tau(M,\gamma,\s,\w)\in \Z[H_1(M)]$ from the function $T_{(M,\g,\w)}.$


\begin{remark} \label{rem:conv}
For proving that the torsion $\tau(M,\g, \s,\w)$ is independent of the chosen handle decomposition $Z = (\A,d)$ of $(M,\g),$ it was important to differentiate between the sutured manifolds $(A,\partial S \times I)$ and $(M,\g).$ Indeed, if we perform a handle slide,
or if we add or remove a canceling pair of handles, then we obtain from $A$ a different (though diffeomorphic) three-manifold $A'.$
These operations would be more difficult to describe only in terms of $M.$

However, from now on we will identify $(A, \partial S \times I)$ and $(M,\g)$ via $d,$ and we will also identify the handles $e_1,\dots,e_n$ with their images $d(e_1),\dots,d(e_n)$ in $M.$ We will often use the notation $\lift(Z),\eul(Z),E_Z,$ and  $\tau(Z,\e,\w)$ for $\lift(\A), \eul(\A), E_{\A},$ and $\tau(\A,\e,\w),$ respectively.
\end{remark}

\subsection{Making $\g$ connected}
Let $(M,\g)$ be a connected balanced sutured manifold and fix a homology orientation $\omega.$ In the future, it will often be convenient to assume that $\gamma,$ and hence  $R_{\pm}(\g)$ are connected. This can be arranged by adding product 1--handles to \((M,\g)\) to produce a new sutured manifold \((M',\g')\). In this section, we describe the effect of this operation on the sutured Floer homology and the torsion. In particular, we show that \(SFH(M,\g)\) can be recovered from \(SFH(M',\g'),\) and likewise for the torsion.

In terms of Heegaard diagrams, this operation of adding a product 1--handle can be described as follows. Suppose \((\S,\bolda,\boldb)\) is a balanced sutured Heegaard diagram for \((M,\g)\), and let \(\S'\) be the result of attaching a 2--dimensional 1--handle \(h\) to \(\S\). Then \((\S',\bolda,\boldb)\) is a Heegaard diagram representing a sutured manifold \((M',\g')\) which is obtained from \(M\) by attaching the 3--dimensional 1--handle \(h\times [-1,1]\).
We can recover \((M,\g)\) from \((M',\g')\) by decomposing along the product disk \(c \times [-1,1]\), where \(c\) is the cocore of \(h\). If  \(h\) joins two different components of $\g,$
then $\g'$ will have one less component than $\g.$

The embedding $(M,R_-(\gamma)) \hookrightarrow (M',R_-(\gamma'))$ induces an isomorphism $$H_*(M,R_-(\gamma);\R) \to H_*(M',R_-(\gamma');\R).$$ Hence $\omega$ induces an orientation $\omega'$ of $H_*(M',R_-(\gamma');\R).$

From \cite[Lemma 9.13]{Ju06}, we know that \(SFH(M',\g') \cong SFH(M,\g)\). To be more precise, by \cite[Prop. 5.4]{Ju10}, there is an injection \(i:\spinc(M,\g) \to \spinc(M',\g')\) for which $$SFH(M,\g, \s) \cong SFH(M',\g',i(\s)).$$ Moreover, \(SFH(M',\g',\s')=0\) if \(\s'\) is not in the image of \(i\). The same proof shows that using the homology orientations $\omega$ and $\omega'$ we get an isomorphism of $\mathbb{Z}/2$ graded groups.

The injection  \(i\) is most easily described by observing that the generating sets \(\T_\a\cap \T_\b\) and \(\T_{\a'} \cap \T_{\b'}\) are naturally identified,
and setting \(i(\s(\boldx)) = \s(\boldx'),\) where $\boldx \in \T_{\a} \cap \T_{\b}$ and $\boldx' \in \T_{\a'} \cap \T_{\b'}$ are corresponding intersection points. This map extends to a map on $\spinc(M,\gamma)$ using the free and transitive $H_1(M)$--action,
which extends to a free action on $\spinc(M',\gamma')$ using the injection $H_1(M)\to H_1(M').$

More intrinsically, \(i\) can be defined as follows. Suppose that
$\s \in \spinc(M,\g)$ is represented by some vector field $v.$ By definition, $v|\g$ is the gradient of a height function $\gamma \to [-1,1].$ So we get a smooth vector field $v'$ on $M'$ if we let $v'|M = v$ and $v'|(h \times [-1,1]) = p^*(\partial/\partial t),$ where $p \colon h \times [-1,1] \to [-1,1]$ is the projection. We define $i(\s)$ as the homology class of $v'.$

To see that the latter description coincides with the former, let $Z$ and $Z'$ be the handle decompositions arising from $(\S,\bolda,\boldb)$ and $(\S',\bolda,\boldb),$ respectively. The inclusion $M \hookrightarrow M'$ naturally induces an affine injection $i_0 \colon \eul(Z) \to \eul(Z')$, and \(i\) is  the composition $ s_{Z'} \circ i_0 \circ s_Z^{-1}$. It is easy to see that \(i_0(\mathfrak{e}(\boldx))  = \mathfrak{e}(\boldx')\), and this implies that \(i(\s(\boldx)) = \s(\boldx')\).

An analogous  statement is satisfied by the torsion:

\begin{lemma} \label{lem:product}
Suppose that $(M',\g')$ is obtained from the connected balanced sutured manifold $(M,\g)$ by adding a product 1--handle. Then  \(T_{(M,\g,\omega)}(\s) = T_{(M',\g',\omega')}(i(\s))\) for every $\s \in \spinc(M,\g).$ Moreover, \(T_{(M',\g',\omega')}(\s')=0\) if \(\s'\) is not in the image of \(i.\)
\end{lemma}

\begin{proof}
In this proof, we use the conventions of Remark~\ref{rem:conv}. Let $Z$ be the handle decomposition of $(M,\g)$ given by $(\S,\bolda,\boldb),$ and $Z'$ the decomposition of
$(M',\g')$ given by $(\S',\bolda,\boldb).$ The universal abelian cover \(p':\hat{M}'\to M'\) can be constructed as follows. Start with a disjoint union $ \displaystyle \sqcup_{i \in \Z} \hat{M}_i,$
where each \(\hat{M}_i\) is homeomorphic to the universal abelian cover \(\hat{M}\) of $M$. Now join \(\hat{M}_i\) to \(\hat{M}_{i+1}\) by 1--handles, one for each element of \(H_1(M)\). These 1--handles are all thickenings of 1--handles in  \(\hat{R}'_-=(p')^{-1}(R_-(\gamma'))\),
so they do not contribute to  \(C_*(\hat{M}',\hat{R}'_-) \). Choose a basis for  \(C_*(\hat{M}',\hat{R}'_-) ,\) all of whose handles are contained in \(\hat{M}_0\); let \(\frakl\) and \(\frakl'\) be the associated lifts for \(M\) and \(M'\).    With respect to such a basis,
\begin{equation*}
C_*(\hat{M}',\hat{R}'_-) \cong C_*(\hat{M},\hat{R}_-)\otimes_\Z \Z[t,t^{-1}].
\end{equation*}
It follows that \(\tau(Z',\frakl',\omega') = i_*(\tau(Z,\frakl,\omega))\), where \(i_*:H_1(M) \to H_1(M')\) is the inclusion. It is easy to see that  $i_0(E_Z(\frakl)) = E_{Z'}(\frakl').$  Since $i = s_{Z'} \circ i_0 \circ s_Z^{-1}$, it follows
that $\tau(M',\g',i(\s),\omega') = i_*(\tau(M,\g,\s,\omega)).$ It is now straightforward to see that this implies the statement of the lemma.
\end{proof}

\section{The proof of Theorem \ref{maintheorem}}
\label{section:maintheorem}

We are now in a position to prove that the Euler characteristic of \(SFH(M,\g,\s,\omega)\) coincides with the torsion \(T_{(M,\g,\omega)}(\s)\). Before giving the proof, we recall some basic facts about handle decompositions, presentations of \(\pi_1\), and Fox calculus.

\subsection{Balanced diagrams and presentations}
\label{Subsection:DiagtoPres}
We begin by explaining how to find a presentation of \(\pi_1(M)\) compatible with a nice handle decomposition \(Z\) having no 0 and 3-handles, or equivalently, with a balanced  diagram \((\S,\bolda,\boldb)\) representing \((M,\g)\). Choose a 2-dimensional handle decomposition of \(R_-\) consisting of one 0-handle and \(l\) 1-handles; this naturally gives  a 3-dimensional handle decomposition of \(R_-\times I\), again with one 0-handle and \(l\) 1-handles. Without loss of generality, we may assume that the attaching disks of the 1-handles of \(Z\) are disjoint from the belt circles of the 1-handles of \(R_-\times I\), and thus (after an isotopy) that the 1-handles of \(Z\) are attached to the 0-handle of \(R_- \times I\).

Fix a basepoint \(p \in R_- \times \{0\}\) which is contained in the 0-handle. Then \(\pi_1(M,p)\) is generated by loops \(\a^*_1, \ldots, \a^*_d, c_1^*,\ldots, c_l^*\), where  \(\a^*_i\) runs through the \(i\)-th 1-handle of \(Z\), and \(c_k^*\) runs through the \(k\)-th 1-handle of \(R_-\times I\). (Note that \(\a^*_i\) depends on the choice of handle decomposition for \(R_-\) as well as on \(Z\).) Each 2-handle in \(Z\) gives rise to a relation as follows. Choose a path \(q_j\) from \(p\) to the attaching circle \(\beta_j\) of the \(j\)-th 2-handle; then \(\bbar_j = q_j \beta_j q_j^{-1}\) is a loop which represents a trivial element of \(\pi_1(M,p)\).
We have a presentation
\[ \pi_1(M,p) =\ll  \a_1^*,\ldots,\a_d^*, c_1^*,\ldots,c^*_l \, |\, \bbar_1,\ldots,\bbar_d\rr.\]

To read off this presentation from a sutured Heegaard diagram \((\S,\bolda,\boldb)\) compatible with \(Z\), we proceed as follows. First, surger \(\S\) along the \(\alpha\)-curves to produce a surface homeomorphic to \(R_-\) and containing \(2d\) marked disks (the traces of the surgery). Next, choose a system of disjoint properly embedded arcs \(c_1,\ldots,c_l\) in \(R_-\) whose complement is homeomorphic to a disk. (This amounts to choosing a handlebody decomposition of \(R_-\).) Without loss of generality, we may assume that the \(c_k\)'s are disjoint from the marked disks, so they lift to arcs \(c_1,\ldots, c_l\) in \(\S\).  To write down the word \(\bbar_j,\) we simply traverse \(\bbar_j\) and record its intersections with the \(\alpha_i\)'s and the \(c_k\)'s as we go.

More precisely, an intersection point \(x\) between \(\alpha_i\) and \(\bbar_j\) is recorded by \(\alpha_i^*\) if the sign of intersection \(\alpha_i \cdot \bbar_j\) at \(x\) is positive, and by \((\a_i^*)^{-1}\) if the sign of the intersection is negative. Note that in doing this, we have implicitly chosen orientations on the \(\alpha_i\)'s , the \(\beta_j\)'s, and the \(c_k\)'s.

In what follows, it will be convenient to choose the path \(q_j\) such that its image in \(\S\)  does not intersect any of the \(\alpha_i\) or \(c_k\). (This is always possible, since the complement of the \(\alpha\)'s and \(c\)'s is connected.) In this case, we write the resulting relation as \(\beta_j\); it is obtained by traversing the curve \(\beta_j\) and recording the intersections with the \(\alpha\)'s and the \(c\)'s.

\subsection{Torsion from a balanced diagram}
Given a nice handle decomposition of \((M,\g)\) as in the previous section, the torsion \(\tau(M,\g)\) can be computed as follows.
As usual, let $H = H_1(M).$ Pick a lift $\hat{p} \in \hat{M}$ of the basepoint $p.$ Then the based curves $c_1^*,\dots,c_l^*$ give rise to a basis \(\hat{C}_1^*,\ldots, \hat{C}_l^*\)
of the free $\Z[H]$-module $C_1(\hat{R}_-;\Z).$ Similarly, the based curves ${\a}_1^*,\dots,{\a}_d^*$  give a basis
\(\hat{A}_1^*,\ldots, \hat{A}_d^*\) of the free $\Z[H]$-module
$C_1(\hat{M},\hat{R}_-;\Z)$. Let \(B_j\) denote the 2-handle attached along \(\b_j\). Then the basings  $q_1,\dots,{q}_n$ give rise
to a choice of lifts $\hat{B}_1,\dots,\hat{B}_d$ of $B_1,\dots,B_d$. These give a basis for the free $\Z[H]$-module
\(C_2(\hat{M}; \Z).\)

\begin{proposition} \label{prop:comptau}
Let $\frakl\in \lift(Z)$ be the element corresponding to the bases $\hat{A}_1^*,\dots,\hat{A}_d^*$ and
$\hat{B}_1,\dots,\hat{B}_d$, and let \(\omega\) be the compatible homology orientation.
Then
\[ \tau(M,\gamma,\frakl,\omega) =  (-1)^{b_1(M,R_-)}\det\left(\varphi\left(\frac{\partial {\b}_j}{\partial {\a}_i^*}\right)\right)\in \Z[H],\]
where $\varphi \colon \Z[\pi_1(M,p)] \to \Z[H]$ is the homomorphism induced by abelianization.
\end{proposition}

\begin{proof}
Consider the following diagram of free $\Z[H]$--modules. We write our choice of basis under the free modules.
\[ \ba{cccccccccccccc}
&&0&\to& C_2(\hat{M};\Z)&\to& C_2(\hat{M},\hat{R}_-;\Z)&\to&0\\
&&&& \ba{c} \hat{B}_1,\dots,\hat{B}_d\ea && \ba{c} \hat{B}_1,\dots,\hat{B}_d\ea&&\\
&&\downarrow&& \downarrow &&\downarrow &&\\
0&\to&C_1(\hat{R}_-;\Z)&\to&C_1(\hat{M};\Z)&\to& C_1(\hat{M},\hat{R}_-;\Z)&\to &0\\
&&\hat{C}_1^*,\dots,\hat{C}_l^*&\to&\ba{c} \hat{A}_1^*,\dots,\hat{A}_d^*\\ \hat{C}_1^*,\dots,\hat{C}_l^*\ea &&
\ba{c} \hat{A}_1^*,\dots,\hat{A}_d^*\ea \\
&& \downarrow &&\downarrow &&\downarrow&&\\
0&\to& C_0(\hat{R}_-;\Z)&\to& C_0(\hat{M};\Z)&\to&0.&&\\
&&\hat{p}&&\hat{p}&&&&\ea \]
Fox calculus tells us  that the boundary map $C_2(\hat{M};\Z)\to C_1(\hat{M};\Z)$ is given by
\[ \bp
 \varphi\left(\frac{\partial {\b}_j}{\partial {\a}_i^*}\right)\\[2mm]
 \varphi\left(\frac{\partial {\b}_j}{\partial c_k^*
}\right)
   \ep. \]
The proposition now follows  from Lemma~\ref{lem:taudet}.
\end{proof}



\subsection{Equality in \(\Z[H]/\pm H\)}
We now turn to the proof of Theorem~\ref{maintheorem}.
First, observe that it is enough to prove the equality in the case where  \(R_-(\gamma)\) is connected. Indeed, if \(R_-(\gamma)\) is not connected, then we can add product one-handles to obtain a new sutured manifold  \((M',\g')\) with \(R_-(\g')\) connected. Lemma~\ref{lem:product} and the discussion preceding it show that if
$
\chi(SFH(M',\g',i(\s),\omega')) = T_{(M',\g',\omega')}(i(\s)),
$
then
$
\chi(SFH(M,\g,\s,\omega)) = T_{(M,\g,\omega)}(\s).
$

For the rest of this section, we assume \(R_-(\gamma)\) is connected.  Our next step is to show that the torsion and the Euler characteristic of $SFH$ agree up to multiplication by \(\pm [h]\) for some \(h \in H\). In light of Proposition~\ref{prop:comptau}, it suffices to prove the following.

\begin{proposition}
\label{Prop:chi=tor}
$$
 \chi(SFH(M,\gamma)) \sim \displaystyle \det \varphi\left(\frac{\partial \beta_j}{\partial \alpha^*_i}\right),
 $$
 where $\varphi: \Z[\pi_1(M)]\to \Z[H_1(M)]$ is the homomorphism induced by abelianization
 and
 where \(\sim\) indicates equality up to multiplication by \(\pm [h]\).
 \end{proposition}

 \begin{proof}
We argue along the lines of Chapter 3 in \cite{Ra03}.
First, observe that
there are natural bijections
\begin{equation*}
\{\text{elements of \(\alpha_i \cap \beta_j\)}\}  \leftrightarrow \{ \text{appearances of \( \alpha^*_i\) in \(\beta_j\)}\}  \leftrightarrow \left\{\text{monomials in \(\varphi\left(\frac{\partial \beta_j}{\partial \alpha^*_i}\right)\)}\right\},
\end{equation*}
where  the free derivative \(\partial \b_j / \partial \a_i^*\) has been expanded {\it without canceling any terms.}
Equivalently,
\begin{equation}
\label{eq:FoxD}
\varphi\left(\frac{\partial \beta_j}{\partial \alpha^*_i}\right) = \sum _{x \in \alpha_i \cap \b_j} m(x)[A(x)],
\end{equation}
where each \(m(x) = \pm 1\) and  each \(A(x)\) is an element of \(H_1(M)\). The sign \(m(x)\) is given by the exponent of the corresponding appearance of \(\alpha_i^*\) in \(\beta_j\), or equivalently, by the sign of intersection \(\alpha_i \cdot \beta_j\) at \(x\).

Recall that the chain complex computing \(SFH\) is generated by  \(d\)-tuples of intersection points \(\boldx = \{x_1,\ldots, x_d\}\), where \(x_i \in \alpha_i \cap \beta_{\sigma(i)}\) for some permutation \(\sigma \in S_d\). On the other hand, the determinant of a \(d\times d\) matrix \((B_{ij})\) can be expanded as
\begin{equation*}
\det (B_{ij}) = \sum_ {\sigma \in S_d} \sign(\sigma) B_{1\sigma(1)}\cdots B_{d\sigma(d)}.
\end{equation*}
Thus we get a bijection
\begin{equation*}
 \T_\a\cap \T_\b  \longleftrightarrow \left\{\text{monomials in \(\det \varphi\left(\frac{\partial \beta_j}{\partial \alpha^*_i}\right)\)}\right\}.
\end{equation*}
Again, all terms in the determinant are to be expanded {\it without cancelation.}

Together with Lemma~\ref{lem:m}, these imply that 
\begin{equation*}
\tau(M,\g) \sim \det \varphi\left(\frac{\partial \beta_j}{\partial \alpha^*_i}\right) = \sum _{\boldx \in \T_\alpha \cap \T_\beta}{m(\boldx)} [A(\boldx)].
\end{equation*}
Here $A(\boldx) = \sum_{i=k}^d A(x_k)$
and  \(m(\boldx)\) is the sign of the intersection \(\T_\alpha \cdot \T_\beta\) at \(\boldx\).  The  orientations on \(\T_\alpha\) and \(\T_\beta\) are induced by the orderings \(\langle \alpha_1, \ldots, \alpha_d \rangle\) and \( \langle \beta_1, \ldots, \beta_d \rangle.\) On the other hand, we know from Equation~(\ref{eqn:1}) at the end of Section~\ref{Subsec:Gens} that
\begin{equation*}
\chi(SFH(M,\g))  \sim  \sum_{\boldx \in \T_\a \cap  \T_\b } m(\boldx) [\iota(\s(\boldx))],
\end{equation*}
where \(\iota: \spinc(M,\g) \to H_1(M)\) is an affine isomorphism.
Comparing this with Lemma~\ref{spincdiff}, we see that the two expressions will agree up to  multiplication by \(\pm[h]\) if we can show the following.
\begin{lemma}
$A(\boldx) - A(\boldy) = \boldx - \boldy$.
\end{lemma}

\begin{proof}
The coefficients \(A(x)\) appearing in Equation~(\ref{eq:FoxD}) can be interpreted geometrically as follows.
In the universal abelian cover \(\hat{M}\),  we have fixed lifts \(\hat{\a}_i\) of \(\a_i\) and \(\hat{\b}_j\) of \(\b_j\) coming from the basings. An intersection point \(x \in \a_i \cap \b_j\) lifts to  a unique
\(\hat{x} \in \hat{\b}_j\), and this point \(\hat{x}\) is contained in
\( (A(x) \cdot \hat{\a}_i )\cap \hat{\b}_j\).

More generally, suppose that \(\tilde{\b}_j\) is an arbitrary
lift of \(\b_j\) and that  \(x \in \a_i\cap \b_j\) and \(x' \in
\a_{i'} \cap \b_j\). Let  $\tilde{x}$ and  $\tilde{x}'$ be the lifts of
\(x\) and \(x'\) which lie on \(\tilde{\b}_j\). If we choose \(h\)
and \(h'\) so that
\(\tilde{x} \in  (h \cdot \hat{\a}_i )\cap
\tilde{\b}_j\) and \(\tilde{x}' \in  (h' \cdot \hat{\a}_{i'})\cap
\tilde{\b}_j\), then  it is easy to see that
\begin{equation*}
h-h' = A(x)-A(x').
\end{equation*}


 Now suppose we are given generators
\(\boldx = \{x_1,\ldots, x_d\}\) and \(\boldy = \{y_1,\ldots, y_d\}\)
in \(\T_\a\cap \T_\b\). To compute the difference \(\boldx-\boldy\), we
choose a 1--cycle \(\theta\) which runs from \(\boldx\) to \(\boldy\)
along the \(\alpha\) curves  and a 1--cycle  \(\eta\) which runs from
\(\boldx \) to \( \boldy \) along the \(\beta\) curves. Let \(\delta\)
be one component of the closed
1--cycle \(\theta-\eta\). After relabeling the \(\a\)'s, \(\b\)'s,
\(x\)'s and \(y\)'s, we may assume that \(x_1 \in \delta\), and that
if we traverse \(\delta\) starting at \(x_1\), we successively
encounter \(\a_1,y_1,\b_1,x_2,\a_2,\ldots,\a_r, y_r\) and \(\b_r \) before
returning to \(x_1\).

Let \(\tilde{\delta}\) be the lift of \(\delta\) to \(\hat{M}\) which
starts at \(\hat{x}_{1}\).
 As we traverse \(\tilde{\delta}\), we successively encounter
\(\tilde{\a}_1,\tilde{y}_1,\tilde{\b}_1,\tilde{x}_2,\tilde{\a}_2,
\ldots,\tilde{\a}_r, \tilde{y}_r\) before arriving at the endpoint
\(\tilde{x}_{r+1} = [\delta] \cdot \hat{x}_1\). Here \(\tilde{\a}_i\) denotes some lift
of \(\a_i\), {\it etc.} Let us write \(\tilde{\a}_i = h_i \cdot \hat{\a}_i\).
To compute the
difference \(h_{i+1}-h_{i}\), we observe that \(\tilde{y}_i\) (which
lies on  \(\tilde{\a}_i\)) and \( \tilde{x}_{i+1}\) (which lies on
\(\tilde{\a}_{i+1}\)) both lie on \(\tilde{\b_{i}}\). It follows that
$$
h_{i+1}-h_i = A(x_{i+1})-A(y_i).
$$

We now compute
\begin{align*}
[\delta] & = h_{r+1}-h_1 \\
& = \sum_{i=1}^r \left(h_{i+1}-h_i \right)\\
& = \sum_{i=1}^r \left({A}(x_{i+1})-{A}(y_i)\right) \\
& = \sum_{i=1}^r \left({A}(x_{i})-{A}(y_i)\right),
\end{align*}
since \(\tilde{x}_{r+1}\) is a lift of \(x_1\).
A similar relation holds for each component of \(\theta - \eta\). Adding them all together, we find that
\begin{equation*}
\boldx-\boldy = [\theta-\eta]=\sum_{i=1}^n \left(A(x_i)-A(y_i)\right) = A(\boldx)-A(\boldy).
\end{equation*}
\end{proof}

\noindent This completes the proof of Proposition~\ref{Prop:chi=tor}.
\end{proof}

\subsection{\(\spinc\) structures}
Our next step is to show that \(\chi(SFH(M,\g,\s)) = \pm T_{(M,\g)}(\s)\) for every $\s \in \spinc(M,\g).$
To see this,  fix a single generator \(\boldx = \{x_1,\ldots, x_d\} \in \T_\a \cap \T_\b\) and consider the associated lift \(\bbx =E_Z^{-1}(\mathfrak{e}(\boldx))\in \lift(Z).\)
 We will show that the term in
$$
\det \varphi \left(\frac{\partial \beta_j} {\partial \alpha_i^*}\right)
$$
corresponding to \(\boldx\) contributes to \(T_{(M,\g)}(\s(\boldx))\), or equivalently, that it contributes \(\pm 1\)
 to the torsion \(\tau(M,\g,\bbx)\). It will then follow from Proposition~\ref{Prop:chi=tor} that the same relation holds for every generator \(\boldy \in \T_\a \cap \T_\b\).

Without loss of generality, we may assume
\(x_i \in \alpha_i \cap \beta_i\).
Consider the matrix
$$ F_{ij}=   \varphi \left(\frac{\partial \beta_j} {\partial \alpha_i^*}\right)$$ of Fox derivatives. Fix lifts \(\tilde{A}_i^*\) of the 1-handles to the universal abelian cover. Then after an appropriate normalization (multiplying each column in the matrix by a unit in \(\Z[H_1(M)]\)), the column vector
\begin{equation*}
\left(\varphi \left(\frac{\partial \beta_j}{\partial \alpha_1^*} \right), \ldots, \varphi \left(\frac{\partial \beta_j} {\partial \alpha_d^*}\right)\right)^T
\end{equation*}
expresses the boundary of a lift \(\tilde{B}_j\) of the two-handle \(B_j\) in terms of the lifts
\(\tilde{A}_i^*\).
It is well-defined up to multiplication by a unit in \(\Z[H_1(M)]\), corresponding to changing the lift \(\tilde{B}_j\). Let us choose the \(\tilde{B}_j\) such that the monomial in \(F_{ii}\) which corresponds to \(x_i\)
  is \(\pm 1\).

The lifts \(\tilde{A}^*_i\) and \(\tilde{B}_j\) determine an element
 \(\mathfrak{l} \in \lift(Z).\) We claim that
 \(\mathfrak{l} = \mathfrak{l}(\boldx)\). To see this, we construct
 an Euler chain for \(\mathfrak{l}\). We fix a point \(\hat{p} \in
 \hat{M}\), and choose paths \(\theta_i^1\) from \(\hat{p}\) to the
 center of \(\tilde{A}^*_i\). Next, we should
 choose paths \(\theta_j^2\) from \(\hat{p}\) to the centers of the
 2-cells \(\tilde{B}_j\).  Now the basis \(\tilde{B}_j\) was chosen
  so that the lift \(\tilde{x}_i\) of \(x_i\)
which lies in \(\tilde{A}^*_i\) is also contained in
\(\tilde{B}_j\). Thus
there is a path \(\theta_i\) from the center of
 \(\tilde{A}^*_i\) to the
center of \(\tilde{B}_i\) passing through \(\tilde{x}_i\)
which is contained in \(\tilde{A}^*_i \cup
\tilde{B}_i\). We let \(\theta^2_i\) be the join of \(\theta_i\) with
\(\theta_i^1\). Then the Euler chain associated to \(\mathfrak{l}\) is
\(\sum (\theta^2_i-\theta^1_i) = \sum \theta_i\). By
Remark~\ref{rem:spinc}, this is exactly the Euler chain which determines
the lift \(\mathfrak{l}(\boldx)\).
Applying Proposition~\ref{prop:comptau}, we see that
\begin{equation*}
\tau(M,\gamma,\bbx) = \pm \det \varphi \left(\frac{\partial \beta_j} {\partial \alpha_i^*}\right),
\end{equation*}
and the monomial \(A(\boldx)\) contributes to \(\tau(M,\gamma,\bbx)\) with coefficient \(\pm1\). In other words, it is assigned to the \(\spinc\) structure \(\s(\boldx)\).

\subsection{Homology orientations}
To complete the proof of Theorem~\ref{maintheorem}, it remains to  check that if we fix a homology orientation \(\w\) for \(H_*(M,R_-(\gamma);\R),\)
then we have  \(\chi(SFH(M,\g,\s,\w)) =  T_{(M,\g,\w)}(\s)\). By Proposition~\ref{prop:comptau}, we know that the torsion is given by
\[ \tau(M,\gamma,\frakl,\omega) =  (-1)^{b_1(M,R_-)}\det\left(\varphi\left(\frac{\partial {\b}_j}{\partial {\a}_i^*}\right)\right).\]
On the other hand, by Definition~\ref{def:torussign}, \(\omega\) determines an orientation on \(\Lambda^d(A) \otimes \Lambda^d(B)\). By Lemma~\ref{lem:m}, with respect to this orientation, the local intersection sign \(\T_\a \cap \T_\b\) at a generator \(\boldx\) is
$$\text{sign}(\sigma) \cdot \prod_{i=1}^dm(x_i).$$
This is    the sign of the term in the determinant corresponding to \(\boldx\). Finally, we recall (Definition~\ref{def:SFHsign}) that the \(\Z/2\) grading induced by \(\omega\) on the sutured Floer homology was defined to be \((-1)^{b_1(M,R_-)}\) times the intersection sign in the symmetric product.
\qed

\section{The torsion function via Fox calculus} \label{section:comp}

In this section, we explain how to compute $T_{(M,\gamma)}$ using Fox calculus.
We assume throughout that  $(M,\gamma)$ is balanced and that the subsurfaces  $R_\pm(\gamma)$ are connected.
In light of Lemma~\ref{lem:product}, this restriction is a very mild one. For brevity, we write  \(R_\pm\) for  \(R_\pm(\g)\) and \(H\) for \(H_1(M)\).

Suppose we are given a basepoint \(p\in R_-\) and a presentation $$ \pi_1(M,p) \cong \langle a_1, \dots, a_m \ts | \ts b_1, \ldots, b_n \rangle.$$ We say that the presentation is {\it geometrically balanced} if \(m-n = g(\partial M)\). For example, any presentation coming from a handle decomposition of \(M\) is geometrically balanced.

We now consider the map \(\iota_*: \pi_1(R_-,p) \to \pi_1(M,p)\) induced by the embedding $\iota \colon R_- \to M.$
Since \(R_-\) is connected and has nonempty boundary, \(\pi_1(R_-,p)\) is free. Choose  elements \(d_1,\ldots, d_l\) which freely generate \(\pi_1(R_-,p)\), and let \(e_k = \iota_*(d_k)\) be their images in \(\pi_1(M,p)\) for $1 \le k \le l.$  Expressing each \(e_k\) as a word in the \(a_i\), we have the Fox derivatives
\(\frac{\partial e_k}{\partial a_i} \in \Z[\pi_1(M)]\). Denote by \(\varphi:\Z[\pi_1(M)] \to \Z[H_1(M)]\) the homomorphism induced by abelianization.
Then we can form the matrix
\[ A= \left( \varphi \left( \frac{\partial b_j}{\partial a_i} \right) \
 \varphi\left( \frac{\partial e_k}{\partial a_i} \right)
   \right). \]

If  our presentation of \(\pi_1(M,p)\) and the sutured manifold \((M,\g)\) are both balanced, then
$1-l = \chi(R_-)  = 1 - g(\partial M) = 1- m+n,$  so
\(m=n+l\), and \(A\) is a square matrix. The main result of this section is the following proposition.

\begin{proposition}
\label{prop:AMatrix}
Let $(M,\gamma)$ be a balanced sutured manifold such that  $M$ is irreducible.
If   $A$ is a square matrix, then
\[   \tau(M,\g)=\det(A) \]
up to multiplication by an element in $\pm H$.
\end{proposition}

\begin{remark}
The above proposition was already proved by Goda and Sakasai in \cite[Proposition~4.6]{GS08} when $(M,\g)$ is a homology cylinder. Recall that a sutured manifold $(M,\g)$ is called a homology cylinder (or homology product)
if the map $H_*(R_-)\to H_*(M)$ is an isomorphism,  and a rational homology cylinder (or rational homology product) if the map
$H_*(R_-;\Q) \to H_*(M;\Q)$ is an isomorphism.
\end{remark}

\begin{proof}
If $M = D^3,$ then the proposition is obviously true. Otherwise, since $M$ is irreducible, no component of $\partial M$ can be homeomorphic to $S^2,$ thus $b_1(M) > 0.$

As in Section~\ref{Subsection:DiagtoPres}, pick a nice handle decomposition $Z$ for $(M,\gamma),$ and another handle decomposition for $R_-.$ Again, denote by $c_1^*,\dots,c_l^*$ the corresponding curves in $R_-$ based at a point $p.$
Let $Y$ be the 1--complex given by the 0--cell $p$ and the 1--cells $c_1^*,\dots,c_l^*.$
Let $X$ be the 2--complex given by collapsing the handles of $Z$ to cells.
Note that $X$ is homotopy equivalent to $M$ and that $Y$ is a subcomplex of $X$. Furthermore, $\tau(X,Y)=\tau(M,\gamma),$ cf.  Remark~\ref{rem:CW}.

Now let $\Pi=\langle a_1, \dots, a_{m} \ts | \ts b_1, \ldots, b_{n} \rangle$ be a geometrically balanced presentation of $ \pi_1(M,p),$
and let $d_1,\dots,d_l$ be any basis for the free group $\pi_1(R_-)$.
A classical theorem of Nielsen in \cite{Ni24} on automorphisms of free groups implies that any  two bases   \(\langle e_1,\ldots, e_l \rangle \) and \(\langle f_1,\ldots, f_l \rangle \) of a free group $\pi_1(R_-)$ are related by a sequence of the following moves. We either 1) replace \(e_k\) by \((e_k)^{-1}\) or  2) replace \(e_k\) by \(e_k e_{k'}\) for some \(k\neq k'\).  Move 1) simply multiplies the \(n+k\)-th column of \(A\) by \(-\varphi(e_k^{-1})\), while move 2) adds \(\varphi(e_k)\) times the \(n+k'\)-th column to the \(n+k\)-th column. In both cases, the determinant of $A$ changes only by multiplication by an element of \(\pm [H]\). Thus
 we may assume that $d_k=[c_k^*]$ for $k=1,\dots,l$.

Let  $D$ be the 2--complex with one 0--cell $p$ and 1--cells $A_1,\dots,A_m$ corresponding to  $a_1,\dots,a_m,$
finally 2--cells $B_1,\dots,B_n$ glued along $b_1,\dots,b_n,$ respectively.
We denote by $X'$ the 2--complex obtained from $D$ by adding extra 1--cells $C_1^*,\dots,C_l^*$ corresponding to $c_1^*,\dots,c_l^*,$ and by adding 2--cells $W_1,\dots,W_l$ such that $\partial W_i$ is glued along $e_*([c_i^*])\cdot (C_i^*)^{-1}.$
Note that $X'$ is a 2--complex that is simple homotopy equivalent to $D$.
We can and will view $Y$ as a subcomplex of $X'$.

We have an obvious isomorphism $\psi:\pi_1(X')\to \pi_1(X)$ which induces the identity on $\pi_1(Y),$ viewed as a subgroup of both $\pi_1(X')$ and $\pi_1(X).$ Since $M$ is irreducible and $\pi_1(M)$ is infinite,
$M$, and hence $X$, are aspherical. It follows that there exists a map $f:X'\to X$ with $f_*=\psi$ and such that $f|_Y=\id_Y.$

\begin{claim}
The map $f:X'\to X$ induces an isomorphism $H_2(X')\to H_2(X)$.
\end{claim}

Let $F$ be the homotopy fiber corresponding to the map $f:X'\to X$. Note that $\pi_1(F)=0$, since $f_*:\pi_1(X')\to \pi_1(X)$ is an isomorphism
and  $\pi_2(X)=0$. It follows that $H_1(F)=0$. Now consider the Leray-Serre spectral sequence corresponding to the fibration $F\to X'\to X$:
We have a spectral sequence with $E^2_{p,q}=H_p(X;H_q(F))$ where the degree of the $n$-th boundary map $d_n$
equals $(n,1-n)$ which converges to $H_*(X').$ Note that $E^2_{p,q}=0$ for $p>2$ since $X$ is a 2--complex, and we also have $E^2_{p,1}=0$ since $H_1(F)=0$. It follows that there is a short exact sequence
\[ 0\to H_0(X;H_2(F))\to H_2(X')\to H_2(X)\to 0,\]
it is well-known that the right hand map is the canonical map induced by the projection.
In particular, the map $H_2(X')\to H_2(X)$ is surjective. Since $X$ and $X'$ are 2--complexes, both $H_2(X)$ and $H_2(X')$ are free abelian groups.
We also have $H_i(X')\cong H_i(X)$ for $i=0,1,$ and
\[ \chi(X')=1-(m+l)+n+l=1-m+n=\chi(R_-)=\chi(M)=\chi(X).\] Hence the map $H_2(X')\to H_2(X)$ is an epimorphism between torsion free abelian groups of the same rank, so it is an isomorphism.

\begin{claim}
The map $f:X'\to X$ induces a homotopy equivalence of pairs $(X',Y)\to (X,Y)$.
\end{claim}

First note that $f_*:\pi_1(X')\to \pi_1(X)$ is an isomorphism. We write $\pi=\pi_1(X)$ and  identify $\pi_1(X')$ with $\pi$ via $f_*$.
Recall that $\pi_i(X)=0$ for $i>1$ since $X$ is aspherical.
By the Hurewicz theorem  this implies that  $H_i(X;\Z[\pi])=0$ for $i>1 $, so it suffices  to show that $H_i(X';\Z[\pi])=0$ for all $i>1$.

Let $U$ be the mapping cylinder of $f:X'\to X$. Then $U$ is a 3--complex which is homotopy equivalent to $X$ and contains $X'$. We can identify $\pi$ with $\pi_1(U).$
From the long exact sequence of the pair $(U,X')$ we now obtain that $H_2(X';\Z[\pi])\cong H_3(U,X';\Z[\pi]).$ So it remains to show that $H_3(U,X';\Z[\pi])=0$.

By the previous claim, the map $f_*:H_2(X')\to H_2(U)=H_2(X)$ is an isomorphism. It follows that $H_3(U,X')=0.$ In particular, the map of free $\Z[\pi]$--modules $C_3(U,X';\Z[\pi])\to C_2(U,X';\Z[\pi])$ becomes injective after tensoring over $\Z[\pi]$ with $\Z.$
On the other hand, by \cite[Corollary~6.2]{Ho82} the group $\pi$ is locally indicable.
According to \cite[Theorem~1]{HS83}, this implies  that $C_3(U,X';\Z[\pi])\to C_2(U,X';\Z[\pi])$ is injective, and in particular $H_3(U,X';\Z[\pi])$ is trivial. This concludes the proof of the claim.
\vskip0.05in
By hypothesis, \(M\) is irreducible and \(b_1(M)>0\), so  $M$ is Haken. By  a theorem of Waldhausen \cite[Theorem~5]{Wa78}, the Whitehead group of $\pi=\pi_1(M)$ vanishes. Thus the map
$f:(X',Y)\to (X,Y)$ is a simple homotopy equivalence, and  the maximal abelian torsions of $(X',Y)$ and $(X,Y)$ agree.
 We have now shown that
\[ \tau(M,\gamma)=\tau(X,Y)=\tau(X',Y).\]
The following claim therefore concludes the proof of the proposition.

\begin{claim}
We have
$\tau(X',Y)=\det(A)$.
\end{claim}

As before, we pick lifts of the cells in $(X',Y)$ to the universal abelian cover. These lifts are, as usual, decorated by a `hat'.
Consider the following diagram of free $\Z[H]$--modules.
\[ \ba{cccccccccccccc}
&&0&\to& C_2(X';\Z[H])&\to& C_2(X',Y;\Z[H])&\to&0\\
&&&& \ba{c} \hat{B}_1,\dots,\hat{B}_n\ea && \ba{c} \hat{B}_1,\dots,\hat{B}_n\ea&&\\
&&&& \ba{c} \hat{W}_1,\dots,\hat{W}_l\ea && \ba{c} \hat{W}_1,\dots,\hat{W}_l\ea&&\\
&&\downarrow&& \downarrow &&\downarrow &&\\
0&\to&C_1(Y;\Z[H])&\to&C_1(X';\Z[H])&\to& C_1(X',Y;\Z[H])&\to &0\\
&&\hat{C}_1^*,\dots,\hat{C}_l^*&\to&\ba{c} \hat{A}_1^*,\dots,\hat{A}_m^*\\ \hat{C}_1^*,\dots,\hat{C}_l^*\ea &&
\ba{c} \hat{A}_1^*,\dots,\hat{A}_m^*\ea \\
&& \downarrow &&\downarrow &&\downarrow&&\\
0&\to& C_0(Y;\Z[H])&\to& C_0(X';\Z[H])&\to&0.&&\\
&&\hat{p}&&\hat{p}&&&&\ea \]
Fox calculus tells us  that the boundary map $C_2(X';\Z[H])\to C_1(X';\Z[H])$ is given by
\[ \bp
 \varphi\left(\frac{\partial {b_j}}{\partial {a}_i}\right)&\varphi\left(\frac{\partial {e_*([c_j^*])}}{\partial {a}_i}\right) \\[2mm]
 0& * \ep.
 \]
Thus \(A\) is the matrix of the boundary map \(C_2(X',Y;\Z[H]) \to C_1(X',Y; \Z[H])\), and the claim follows.
\end{proof}

\section{Algebraic properties of the torsion}
\label{section:algebra}

In this section, we collect some algebraic properties of the torsion function and their consequences. We begin by describing  some known examples which appear as special cases of the torsion for sutured manifolds. We then turn our attention to the ``evaluation homomorphism'' \(H_1(M) \to H_1(M,R_-(\gamma))\) and prove
  Proposition~\ref{Prop:eval} from the introduction. Finally,  we discuss sutured L-spaces and give the proof of Corollary~\ref{Cor:Lspace}.

\subsection{Special Cases of the Torsion} In this section, we summarize some useful special cases of the torsion. These are all ``decategorifications" of known facts about sutured Floer homology, although in many cases they admit more elementary proofs as well.

\begin{lemma}
Let $(M,\gamma)$ be a product sutured manifold. Denote by $\s_0$ the canonical vertical $\spinc$--structure of $(M,\gamma)$ and take $\omega$ to be the positive orientation of $H_*(M,R_-(\gamma);\R) = 0.$
Then
\[  T_{(M,\gamma,\omega)}(\s)=\left\{ \ba{rl} 1, &\mbox{ if }\s=\s_0, \\ 0, &\mbox{ if }\s\ne \s_0.\ea \right.\]
\end{lemma}

\noindent
This decategorifies the
 fact that \(SFH(M,\g,\s)\) is isomorphic to \( \Z\) if \(\s=\s_0\) and is trivial otherwise \cite{Ju06}.

\begin{proof}
Let $Z$ be the handle decomposition of $(M,\g)$ with underlying sutured handle complex $(A, S\times I)$ such that $S = R_-(\gamma)$
and $A = S \times I.$ Then $s_Z^{-1}(\s_0) \in \eul(Z)$ can be represented by the 1-cycle $\theta = \emptyset.$ Since there are no handles, we have a canonical identification between $\eul(Z)$ and $H_1(M),$ and by the definitions of Subsection~\ref{subsection:CW}, in this case $\tau(Z,h,\omega) = h$ for every $h \in \eul(Z) \cong H_1(M).$
The result follows.
\end{proof}

\begin{lemma}
Let $Y$ be a closed 3--manifold and let $Y(1)$ be the balanced sutured manifold defined in Example~\ref{ex:1}.
Then for every $\s\in \spinc(Y(1))$ we have
\[  T_{Y(1)}(\s)=\left\{ \ba{rl} 1 &\mbox{ if }b_1(Y)=0, \\ 0 &\mbox{ if }b_1(Y)>0.\ea \right.\]
\end{lemma}

\begin{proof}
Since $R_-(\gamma) = D^2,$ the map $p_* \colon H_1(M) \to H_1(M,R_-(\gamma))$ is an isomorphism. So the result follows immediately from Proposition~\ref{Prop:eval}.
\end{proof}
\noindent This also  follows from the isomorphism \(SFH(Y(1)) \cong \widehat{HF}(Y)\) from \cite{Ju06}, together with the corresponding calculation of  \(\chi(\widehat{HF}(Y,\s))\) in \cite{OS04a}.

Let $L\subset S^3$ be an ordered oriented $k$--component link. Let $S^3(L)$ be the corresponding balanced sutured manifold as defined in Example~\ref{ex:2}. Then $H = H_1(S^3\setminus N(L))$ is the
free abelian multiplicative group generated by $t_1,\dots,t_k$, where \(t_k\) is represented by the meridian of the \(k\)-th component of $L.$

\begin{lemma} \label{lem:alex}
If $L \subset S^3$ is a $k$-component link, then given $\s \in \spinc(S^3(L)),$ we have
\[ \tau(S^3(L))= \sum_{h\in H} T_{S^3(L)}(h + \s) \cdot h \sim \left\{ \ba{ll} \Delta_L(t_1) &\mbox{ if $k = 1$}\\
 \Delta_L(t_1,\ldots,t_k) \cdot \prod_{i=1}^k(t_i-1) &\mbox{ if $k\geq 2$.}\ea \right.\]
\end{lemma}

\begin{proof}
Let $Q = \Q(t_1,\dots,t_k).$ Note that we have a short exact sequence of chain complexes
\[ 0\to C_*(R_-(\gamma);Q)\to C_*(S^3\sm N(L);Q)
\to C_*(S^3\sm N(L),R_-(\gamma);Q)\to 0.\]
From the multiplicativity of torsion (cf. \cite[Theorem~1.5]{Tu01}) it follows that
\[ \tau(S^3\sm N(L),R_-(\gamma);Q) = \tau(S^3\sm N(L);Q)\cdot \tau(R_-(\gamma);Q)^{-1}.\]
For $i=1,\dots,k,$ we denote by $R_i$ the component of $R_-(\gamma)$ corresponding to the $i$-th component of the link $L.$
It follows easily from the definition of the torsion that $\tau(R_i;Q)=(t_i-1)^{-1}$.
We therefore deduce that
$$\tau(S^3\sm N(L),R_-(\gamma);Q) = \tau(S^3\sm N(L);Q)\cdot \tau(R_-(\gamma);Q)^{-1}=$$
$$=\tau(S^3\sm N(L);Q)\cdot \prod\limits_{i=1}^k \tau(R_i;Q)^{-1}=\tau(S^3\sm N(L);Q)\cdot \prod _{i=1}^k (t_i-1).$$
The lemma now follows from the following relation between the torsion and the Alexander polynomial
(cf. \cite[Theorem~11.8]{Tu01}):
\[ \tau(S^3\sm N(L); Q)=\left\{ \ba{ll} \Delta_L(t_1)(t_1-1)^{-1}, &\mbox{ if $k=1$}\\
 \Delta_L(t_1,\ldots,t_k), &\mbox{ if $k\geq 2$.}\ea \right.\]
\end{proof}

The following lemma can be viewed as the decategorification of  \cite[Theorem~1.5]{Ju08}.

\begin{lemma}\label{lem:topcoeff}
Let $K\subset S^3$ be a knot and let $R$ be a genus minimizing Seifert surface for $K.$
Then $$ \sum_{\s \in \text{Spin}^c(S^3(R))} T_{Y(R)}(\s)$$ is equal to the coefficient of  $t^{g(R)}$ in the symmetrized Alexander polynomial of $K.$
\end{lemma}

\begin{proof}
Recall that there exists a  Mayer--Vietoris sequence
\[ H_1(R)\otimes \zt \xrightarrow{A_--tA_+} H_1(S^3\sm N(R))\otimes \zt \to H_1(S^3\sm N(K);\zt)\to 0.\]
Here we endow $H_1(R)$ with any basis and $H_1(S^3\sm N(R))$ with the dual basis.
Then the symmetrized Alexander polynomial of $K$ equals $$p(t)=\det(A_--tA_+)t^{-g(R)}.$$

The coefficient of \(t^{g(R)}\) in $p(t)$ is $$\det(A_+)=|H_1(S^3\sm N(R),R_-(\gamma))|=|H_1(S^3(R),R_-(\gamma))|,$$
where we write $|G|=0$ if $G$ is infinite. With this convention, $\eps(I_G)=|G|$
for any group $G$, where $\eps:\Z[G]\to \Z$ is the augmentation map.
The lemma now follows immediately from Lemma~$\ref{lem:alex}$ and Proposition~\ref{Prop:eval}.
\end{proof}

\subsection{Evaluation}
In this section, we study the behavior of the torsion under the natural map \(p:H_1(M) \to H_1(M,R_-(\gamma))\).
We prove the following statement, which is clearly equivalent to  Proposition~\ref{Prop:eval} of the introduction.
\begin{proposition}
If \((M,\g)\) is a balanced sutured manifold, then
$$p_*(\tau(M,\g)) = \pm I_{H_1(M,R_-(\gamma))},$$ where
  \(I_G \in \Z[G]\) is the sum of all elements in \(G\) if \(G\) is finite, and is \(0\) otherwise.
\end{proposition}

\noindent Equivalently, if  $K = PD[\text{ker} \ p_*] \subset H^2(M,\partial M),$
then there is an $\varepsilon \in \{-1,1\}$ such that for every \(\s \in \spinc(M,\g)\)
\[  \sum_{h\in K}  T_{(M,\gamma)}(\s+h)=\left\{ \ba{rl} \varepsilon &\mbox{ if }H_1(M,R_-(\gamma))\mbox{ is finite}, \\ 0 &\mbox{ otherwise}.\ea \right.\]

This result generalizes the fact that  \(\Delta_K(1) = \pm 1\) whenever \(K\subset Y\) is a knot in a homology sphere. Indeed, if we take \((M,\g)\) to be the manifold \(Y(K)\) from Example~\ref{ex:2}, then \(H_1(M,R_-(\gamma)) = 0\), and \(p_*\) induces the map \(  \Z[t^{\pm1}] \to  \Z\) which is evaluation at \(t=1\). We also refer to \cite[Corollary~II~5.2.1]{Tu02}, where it is shown that for a null-homologous knot $K$ in a rational homology sphere $Y$ we have $\Delta_K(1)= \pm |H_1(Y)|.$

\begin{proof}
By adding product 1-handles, we reduce to the case where \(R_-(\gamma)\) is connected. Indeed, if \((M',\g')\) is obtained from \((M,\g)\) by adding a product 1-handle, then \(H_1(M',R_-(\g')) \cong H_1(M,R_-(\gamma))\), and the diagram
\begin{equation*}
\begin{CD}
H_1(M) @>{p_*}>> H_1(M,R_-(\gamma)) \\
@V{i_*}VV @VV{\cong}V \\
H_1(M') @>{p_*'}>> H_1(M',R_-(\g')) \\
\end{CD}
\end{equation*}
commutes. Applying Lemma~\ref{lem:product}, we see that if  we know that $p_*'(\tau(M',\g')) = I_{H_1(M',R_-(\g'))},$ then  \(p_*(\tau(M,\g))=I_{H_1(M,R_-(\gamma))}\) as well.

From now on, we assume that \(R_- = R_-(\gamma)\) is connected.
Denote the group \(H_1(M,R_-)\) by \(\HH\). Let \(\psi: \pi_1(M) \to \HH\) be the composition of \(p_*\) with the abelianization map, and consider the connected covering map
  \(\pi:\widetilde{M} \to M\) corresponding to the kernel of \(\psi\).
   We write $\ti{R}_-=\pi^{-1}(R_-)$. Let \(\overline{\tau}\) be the maximal abelian torsion of \(C_*(\widetilde{M}, \widetilde{R}_-)\), viewed as a module over \(\Z[\HH]\).  Then
$$C_*(\widetilde{M},\widetilde{R}_-) \cong C_*(\hat{M},\hat{R}_-)\otimes_{\Z[H_1(M)]}\Z[\HH]. $$
It now follows immediately from the proof of Lemma~\ref{lem:taudet} that
 \(\overline{\tau} = p_*(\tau(M,\g))\). Thus it suffices to show that \(\overline{\tau} = {I}_{\HH}\).

As in Subsection~\ref{sec:MaxTorsion}, let \(T \subset G\) be the torsion subgroup, and pick a splitting
\(G = F \times T\), where \(F\) is a free abelian group. Under the isomorphism
$$
\Q[T] \cong \bigoplus_i \F_{\chi_i}
$$
\({I}_T\) maps to the element whose  $\F_{\chi_i}$ component is \(0\) for all non--trivial $\chi_i$ and whose component in
$\F_{\chi_{\text{id}}} \cong \Q$ is   $|T|$. (Here \(\chi_{\text{id}}\) denotes the trivial character.)
To show  that \(\overline{\tau}\) is a multiple of \({I}_G\), it suffices to show that
the torsion \(\tau^{\varphi_i}({M},{R}_-)\) vanishes whenever either the group \(F\) or the character \(\chi_i\) is nontrivial.
Equivalently, we must prove that the complex
\(C_*(\widetilde{M}, \widetilde{R}_-) \otimes Q(\F_{\chi_i}[F])\) has nontrivial homology.

To this end, we consider the groups \(H_0(\widetilde{M})\) and \(H_0(\widetilde{R}_-)\). By construction \(H_0(\widetilde{M}) = \Z\).
Recall (see \cite[Section~VI.3]{HS97}) that
\[ H_0({M}, Q(\F_{\chi_i}[F])) = Q(\F_{\chi_i}[F])/\{ gv-v \, |\, g\in \pi_1(M), v\in Q(\F_{\chi_i}[F])\}.\]
It follows that $H_0({M}, Q(\F_{\chi_i}[F]))=0$ if and only if the $\pi_1(M)$ action on the field
$Q(\F_{\chi_i}[F])$ is non-trivial. We thus see that $H_0({M}, Q(\F_{\chi_i}[F]))=0$  unless \(F=0\) and \(\chi_i\) is the trivial character.
On the other hand,
if \(e_*: \pi_1(R_-) \to \pi_1(M)\) is the map induced by the embedding $e,$ then \(\psi \circ e_*\) factors as a composition
$$
\pi_1(R_-) \to H_1(R_-) \to H_1(M) \to H_1(M,R_-)
$$
and is therefore the zero map. It follows that  \(\pi: \widetilde{R}_- \to R_-\) is a trivial covering. We assumed  \(R_-\) is connected, so \(p_*\) is a surjection and the deck group of \(\widetilde{M}\) is isomorphic to \(G\). Thus \(H_0( \widetilde{R}_-) \cong \Z[G]\), from which it follows that
\(H_0( {R}_-, Q(\F_{\chi_i}[F]))
\cong Q(\F_{\chi_i}[F])\) for any character \(\chi_i\).

We now consider the long exact sequence of the pair \((\widetilde{M},\widetilde{R}_-)\):
$$
\to H_1({M},{R}_-;Q(\F_{\chi_i}[F]) ) \to H_0({R}_-;Q(\F_{\chi_i}[F])) \to
H_0({M};Q(\F_{\chi_i}[F])) \to
$$
The middle group in this sequence has rank \(1\), but the last group is trivial unless \(F=0\) and \(\chi_i=1\). It follows that \( H_1({M},{R}_-;Q(\F_{\chi_i}[F]) )\) is nontrivial unless \(F=0\) and \(\chi_i\) is trivial.

To finish the proof, we need only compute the torsion \(\tau^{\varphi_{\text{id}}}({M},{R}_-)\)
when \(F=0\) and \(\varphi_{\text{id}}\) is the homomorphism induced by the trivial character. But in this case
\(C_*(\widetilde{M},\widetilde{R}_-) \otimes Q(\F_{\chi_{\text{id}}}[F])\) reduces to the ordinary chain complex
\(C_*(M, R_-;\Q)\). This complex is trivial for \(i \neq 1,2\), so the torsion is \(\det(d)\), where
\(d:C_2(M,R_-) \to C_1(M,R_-)\) is the boundary map. In other words,
$
 \tau^{\varphi_{\text{id}}}({M},{R}_-) = \pm |H_1(M,R_-)| = \pm |G|
$
as desired.
\end{proof}

\subsection{Sutured \(L\)-spaces}
From the introduction, we recall the following definition.

\begin{definition}
We say that \((M, \g)\) is a \emph{sutured L-space} if \(SFH(M,\g)\) is torsion free and supported in a single \(\Z/2\) homological grading.
\end{definition}

For sutured $L$--spaces,  the sutured Floer homology is determined by its Euler characteristic; i.e., by the torsion function. In fact, the sutured Floer homology of a sutured \(L\)-space has an especially simple form.

\begin{corollary} If \((M, \g)\) is a sutured L-space, then  for each \(\s\in \spinc(M,\g)\) the group \(SFH(M,\g,\s)\) is either trivial or isomorphic to \(\Z\).
\end{corollary}

\begin{proof}
Since \(SFH\) is supported in a single homological grading,
\(T_{(M,\g)}(\s)\) is equal to the rank of \(SFH(M,\g,\s)\). Then Proposition~\ref{Prop:eval} shows that
\[  \sum_{h \in K} \text{rank} \ SFH(M,\g,\s + h) \leq 1.\] This clearly implies the statement above.
\end{proof}

\begin{lemma}
Suppose that \((M,\g)\) is a balanced sutured manifold and that \(S \subset M\) is a nice decomposing
surface (cf. \cite[Definition 3.22]{Ju10}). Let \((M',\g')\) be the result of decomposing \((M,\g)\) along \(S.\) Furthermore, let \(O_S\) be the set of outer \(\spinc\) structures for \(S\), as defined in \cite[Definition 1.1]{Ju08}. If
$$
\bigoplus_{s\in O_S} SFH(M,\g,\s)
$$
is torsion free and supported in a single \(\Z/2\) grading, then \((M',\g')\) is a sutured L-space.
\end{lemma}

\begin{proof}
By the decomposition formula \cite[Theorem 1.3]{Ju08},
$$
SFH(M',\g') \cong \bigoplus_{s\in O_S} SFH(M,\g,\s).
$$
The proof there also gives that $SFH(M',\g')$ is supported in a single $\Z/2$ grading.
\end{proof}

\begin{corollary}
If \((M,\g)\) is a sutured \(L\)-space and  \((M',\g')\) is obtained by decomposing \((M,\g)\) along a nice surface, then \((M',\g')\) is also a sutured \(L\)-space.
\end{corollary}

\begin{corollary}
\label{cor:hfl}
Suppose that \(R\) is a minimal genus Seifert surface of an oriented $k$-component link \(L \subset S^3.\) If \(\widehat{HFK}(L,g(R)+k-1)\) is torsion free and supported in one \(\Z/2\) grading, then \(S^3(R)\) is a sutured \(L\)-space.
\end{corollary}

\begin{remark}
Here we require from any Seifert surface that it intersects a meridian of each component of $L$ geometrically once.
\end{remark}

\begin{proof}
The sutured manifold \(S^3(R)\) is obtained by decomposing \(S^3(L)\) along \(R\). By \cite[Proposition 9.2]{Ju06}, there is an isomorphism \(SFH(S^3(L)) \cong \widehat{HFK}(L).\) As shown in \cite[Theorem 8.4]{Ju08}, the part corresponding to the outer \(\spinc\) structures with respect to \(R\) is exactly \(\widehat{HFK}(L,g(R)+k-1).\)
\end{proof}

\begin{corollary}\label{cor:alt}
If \(L \subset S^3\) is a non-split alternating link and \(R\) is a minimal genus Seifert surface of \(L,\) then \(S^3(R)\) is a sutured \(L\)-space. \end{corollary}
\begin{proof}
The main theorem of \cite{OS03} implies that \(\widehat{HFK}(L,g(R)+k-1)\) is torsion free and supported in a single homological grading, so the result follows from Corollary~\ref{cor:hfl}.
\end{proof}

\noindent We remark that there are many non-alternating links which also satisfy the hypothesis of Corollary~\ref{cor:hfl}. For example, it is satisfied by all knots of ten or fewer crossings
and it is satisfied by the class of quasi-alternating knots (we refer to \cite{OS05} and \cite{MO08} for definitions and proofs).

\section{The Thurston norm for sutured manifolds}
\label{section:norm}

Let \((M,\g)\) be a sutured manifold. In \cite{Sc89}, Scharlemann introduced a
natural seminorm on \(H_2(M,\partial M;\R)\)  which generalizes the usual Thurston norm  of \cite{Th86}.
In this section, we investigate the relation between this norm and the sutured Floer homology.

\begin{defn}
Let $(M,\gamma)$ be a sutured manifold.
Given a properly embedded, compact, connected surface $S \subset M,$ let
\[ x^s(S) = \mbox{max}\{ 0,\chi(S\cap R_-(\gamma)) -\chi(S)\},\]
and extend this definition to disconnected surfaces by taking the sum over the components.
Note that $S\cap R_-(\gamma)$ is necessarily a one--dimensional manifold and
$\chi(S\cap R_-(\gamma))$ equals the number of components of $S\cap R_-(\gamma)$ which are not closed.
Equivalently, we have
\[ x^s(S) = \mbox{max}\left\{ 0, -\chi(S)+\frac{1}{2} |S\cap s(\gamma)|\right\}.\]
For $\alpha \in H_2(M,\partial M),$ let
\[ x^s(\alpha)=\mbox{min}\{ x^s(S)\, \colon \, S \subset M \mbox{ is properly embedded and}\,\, [S,\partial S] =  \alpha \}.\]
\end{defn}

\begin{theorem} \cite{Sc89}
Let $(M,\gamma)$ be a sutured manifold. Then the function
\[x^s: H_2(M,\partial M) \to  \Z_{\ge 0} \]
defined above has the following two properties:
\begin{enumerate}
\item $x^s(n\alpha) = |n|\cdot x^s(\alpha)$ for all $n\in \Z$ and $\alpha \in H_2(M,\partial M),$
\item $x^s(\alpha+\beta) \leq  x^s(\alpha)+x^s(\beta)$ for all $\alpha,\beta \in H_2(M,\partial M).$
\end{enumerate}
\end{theorem}

It follows that $x^s$ extends to a continuous map $x^s:H_2(M,\partial M;\R)\to \R_{\ge 0}$ which is convex and linear on rays from the origin. Put differently,  $x^s$ is a seminorm on $H_2(M, \partial M;\R)$. It is called the \emph{sutured Thurston norm}.

\begin{example*}
Let $Y$ be a closed 3--manifold and let $(M,\g) = Y(1)$ be $Y$ with an open ball removed and having a connected suture. Then we can identify $H_2(Y;\R)$ with $H_2(Y(1),\partial Y(1);\R).$
It is straightforward to see that under this identification the Thurston norms of $Y$ and $Y(1)$ agree.
\end{example*}

\begin{example*}
Let $K\subset S^3$ be a knot and let $(M,\g) = S^3(K)$ be the associated sutured manifold with two meridional sutures.
If $\alpha \in H_2(M,\partial M)$ is a generator, then $x^s(\alpha)=2g(K).$ Note that this differs from the usual Thurston norm $x$ of $M$, which satisfies $x(\alpha) = 2g(K)-1$ for a non-trivial knot.
\end{example*}

The previous example is in fact a special case of the following elementary lemma.

\begin{lemma} \label{lem:linknorm}
Let $Y$ be a closed 3--manifold and let $L\subset Y$ be an $l$--component link with meridians $\mu_1,\dots,\mu_l$. Suppose that $L$ has no trivial components.
Let $Y(L) = (M,\g)$ be the corresponding sutured manifold.
Given $h\in H_2(M,\partial M),$ we have
\[ x^s(h)=x(h) + \sum_{i=1}^l |\langle\,h,\mu_i\,\rangle|,\]
where $x(h)$ denotes the ordinary Thurston norm on $H_2(M,\partial M).$
\end{lemma}

The following proposition should be compared with \cite[Theorem~6.1]{Ju10} and \cite[Proposition~8.5]{Ju10}.
For all the necessary definitions also see \cite{Ju10}.

\begin{proposition}
Suppose that the taut balanced sutured manifold $(M,\gamma)$ is reduced, horizontally prime, and $H_2(M)=0.$ Then $x^s$ is a norm.
\end{proposition}

\begin{proof}
Assume there exists an $\alpha \neq 0$ in $H_2(M,\partial M)$ with $x^s(\alpha)=0.$ This implies that there exists a connected homologically non-trivial orientable surface $(S,\partial S) \subset (M,\partial M)$
with $x^s(S)=0$. Hence $$-\chi(S) \le -\chi(S)+\chi(S\cap R_-(\gamma)) \le 0.$$ So $\chi(S)\geq 0$ and $S$ is either $S^2, T^2, D^2,$ or $S^1 \times I.$ Since $[S,\partial S] \neq 0$ and $H_2(M)=0,$ the surface $S$ is not $S^2$ or $T^2.$ Furthermore, we can assume that $S \cap \g$ consists of arcs connecting $R_-(\gamma)$ and $R_+(\g).$

Now suppose that $S = D^2.$ Then $\chi(S \cap R_-(\gamma))$ is $0$ or $1.$ In the latter case $S$ would be a homologically non-trivial product disc, contradicting the assumption that $(M,\g)$ is reduced. In the former case $S$ is a compressing disk for $R(\g).$ Since $(M,\g)$ is taut, $R(\g)$ is incompressible, so $\partial S$ bounds a disk $S'$ in $R(\g).$ Now $S \cup (-S')$ is a sphere, which has to bound a $D^3$ since $M$ is irreducible. But then $[S,\partial S] =0,$ a contradiction.

Finally, assume that $S$ is an annulus. Then $\chi(S \cap R_-(\gamma)) = 0.$ Since $(M,\gamma)$ is reduced, we know that $S$ can not be a product annulus. So suppose that $\partial S \subset R,$ where $R$ is either $R_-(\gamma)$ or $R_+(\gamma).$
Pick a product neighborhood $S\times [0,1]$ and let
\[ R'=R\sm (\partial S\times (0,1)) \, \cup (S\times 0) \cup (S\times 1).\]
Then $R'$ is homologous to $R$, $\partial R'=\partial R,$ and $\chi(R')=\chi(R).$ Hence $R'$ is a horizontal surface. Note that $R'$ is not parallel to $R.$ If $R'$ were parallel to $R(\g) \sm R,$ then $\partial S \times [0,1]$ would give rise to a non-trivial product annulus. So the existence of $R'$ would contradict our assumption that $(M,\gamma)$ is horizontally prime.
\end{proof}


\begin{defn}
Let $S(M,\g) = \{\s \in \spinc(M,\g) \,\colon \, SFH(M,\g,\s) \neq 0 \}$  be the support of \(SFH(M,\g).\) If $\alpha \in H_2(M,\partial M;\R)$, we define
\[ z(\alpha) =\max \{ \langle\,\s - \t, \alpha \,\rangle \, \colon \, \s,\t\in S(M,\g)\}.\]
\end{defn}

\begin{remark}
In \cite[Section~8]{Ju10} another ``seminorm'' $y$ on $H_2(M,\partial M;\R)$ was constructed using sutured Floer homology.
The function $y$ satisfies all properties of a seminorm except that $y(\alpha) \neq y(-\alpha)$ can happen. It is straightforward to see that $$z(\alpha)=\frac{1}{2}(y(\alpha)+y(-\alpha)).$$
\end{remark}

The following proposition proves the second statement of Theorem~\ref{thm:Tnorm}.

\begin{proposition}  \label{thm:torusboundary}
Let $(M,\gamma)$ be an irreducible balanced sutured manifold such that all boundary components of $M$ are tori. Then $z=x^s$.
\end{proposition}

\begin{proof}
First note that by standard arguments it suffices to show the equality of norms for integral classes.

Now assume that each component of $\partial M$ has exactly two sutures.
Observe that we can assign to $(M,\g)$ a link $L$ in a 3-manifold $Y$ which is obtained by Dehn filling $\partial M$ such that the $\mu_i$ become meridians of the filling tori, see \cite[Example 2.4]{Ju06}. Then by \cite[Remark 8.5]{Ju08} we can assign to each \(\s \in \spinc(M,\g)\) a relative first Chern class \(c_1(\s) \in H^2(M,\partial M) \) in such a way that the set
\( \{c_1(\s) \, \colon \, \s \in S(M,\g)\}\) is symmetric about the origin. Then by
\cite[Remark~8.5]{Ju08} for every $h \in H_2(M,\partial M)$
\begin{equation*}
\max \{\langle \,c_1(\s), h\,\rangle \, \colon \, \s \in S(M,\g)\,\} = x(h) + \sum_{i=1}^l |\langle\,h,\mu_i\,\rangle|.
\end{equation*}
Since the image of \(S(M,\g)\) is centrally symmetric, this is equivalent to saying
\begin{equation*}
\max \{ \langle\,\s - \t, h \,\rangle \, \colon \, \s, \t \in S(M,\g)\,\} = x(h) + \sum_{i=1}^l |\langle\,h,\mu_i\,\rangle|.
\end{equation*}
Note that the left hand side is just $z(h).$ If $L$ is not the unknot, then $L$ does not have trivial components since $M$ is irreducible. So by Lemma~\ref{lem:linknorm}, the right hand side is exactly $x^s(h).$ The proposition is obviously true if $L$ is the unknot.

Now consider the general case when each component of $\partial M$ has at least two sutures.
We can reduce this case to the case treated above using  \cite[Proposition 9.2]{Ju10}.
\end{proof}

\begin{remark}
Link Floer homology of a link $L \subset S^3$ was defined in \cite{OS08a}. It agrees with the sutured Floer homology of the sutured manifold $S^3(L)$ introduced in Example~\ref{ex:2}.
In \cite{OS08b} it is shown that if $L$ has no trivial components, then the link Floer homology of $L$ determines the Thurston norm of the link complement. In light of the above theorem it is perhaps a better point of view to observe  that twice the link Floer polytope equals the dual of the sutured Thurston polytope of $S^3(L),$ which then determines the ordinary Thurston polytope of $S^3(L)$ via Lemma~\ref{lem:linknorm}.
\end{remark}


The following theorem is exactly the first part of Theorem~\ref{thm:Tnorm}.

\begin{theorem}\label{thm:Tnorm1}
Let $(M,\gamma)$ be an irreducible balanced sutured manifold. Then $z\leq x^s$.
\end{theorem}

We will later see in Proposition~\ref{prop:2} that the inequality of Theorem~\ref{thm:Tnorm1} is strict in general.

In order to prove Theorem~\ref{thm:Tnorm1},  we consider the double of the sutured manifold
\((M,\g)\) along \(R(\g)\). More precisely, the double \(DM\) of \((M,\g)\)  is obtained from the disjoint union of $M$ and $-M$ by identifying the two copies of  $R(\g)$ via the identity map. The boundary of \(DM\) is a union of tori; each torus is the double of a component of \(\g\).
In the context of sutured manifolds this operation was first used by Gabai  \cite{Ga83}.

A theorem of Cantwell and Conlon relates the sutured Thurston norm on \((M,\g)\) to the Thurston norm of the double. To be precise, suppose \((M,\g)\) is a sutured manifold, and let \(X=DM\) be the double of \(M\) along \(R(\g)\). There is a natural ``doubling map'' \(D_*:H_2(M,\partial M;\R) \to H_2(X,\partial X;\R).\)
Note that the doubling map takes the homology class represented by a surface $(S,\partial S)\subset (M,\partial M)$ to the homology class represented by its double. In particular,
we immediately see that $x(D_*(\alpha)) \leq 2 x^s(\alpha)$. The following theorem shows that in fact equality holds.

\begin{theorem} \cite[Theorem 2.3]{CC06} \label{thm:CC}
We have
$x(D_*(\alpha)) = 2 x^s(\alpha).$
\end{theorem}
\noindent Here \(x\) denotes the usual Thurston norm on \(H_2(X,\partial X;\R)\) and \(x^s\) is the sutured Thurston norm on \(H_2(M,\partial M;\R).\)

\begin{defn}
We make $X=DM$ into a sutured manifold $(X,\g_X)$ in the following canonical way. Let the components of $\g$ be $\g_1,\dots,\g_l.$ For each component $\g_i$ of $\g$ choose two parallel, oppositely oriented arcs $m_i$ and $m_i'$ that connect $R_+(\g)$ and $R_-(\gamma).$ Then on the torus $D\gamma_i \subset \partial X$ the sutures are $\mu_i = m_i \cup (-m_i)$ and $\mu_i' = m_i' \cup (-m_i').$ These sutures are well defined up to isotopy. Let $\g_X$ be a regular neighborhood of $\bigcup_{i=1}^l (\mu_i \cup \mu_i')$ inside $\partial X.$
\end{defn}

\begin{lemma}
\label{Lem:Double}
Let $(M,\g)$ be an irreducible balanced sutured manifold and let \((X,\g_X)\) be its double. Then for all  \(\a \in H_2(M,\partial M;\R)\), we have \emph{
$$
2x^s(\a) = \max \{ \langle\,\s - \t, D_*(\alpha) \,\rangle \, \colon \, \s,\t\in S(X,\g_X)\}.
$$}
\end{lemma}

\begin{proof}
First note that by standard arguments it suffices to show the equality when $\a$ is an integral class.

Since $X$ is irreducible and has only toroidal boundary components, by Proposition~\ref{thm:torusboundary} we have
\begin{equation*}
\max \{ \langle\,\s - \t, h \,\rangle \, \colon \, \s, \t \in S(X,\g_X)\,\} =x^s(h)
\end{equation*}
for any $h\in H_2(X,\partial X).$

We claim that if $h =D_*(\alpha)$ for some $\alpha \in H_2(M,\partial M)$, then $\langle\, h, \mu_i \,\rangle = 0$ for every $1 \le i \le l.$ To see this, choose a surface \(S\) representing \(\alpha\). We may assume that \(S \cap \g_i\) consists of a collection of parallel arcs. If we take \(m_i\) and \(m_i'\) parallel to these arcs, then \(\partial S \cap \mu_i \) and \(\partial S \cap \mu_i'\)  are empty. Consequently, $x^s(h) = x(h).$
Combining this with the fact from Theorem \ref{thm:CC} that $x(h) = 2x^s(\alpha),$ we obtain the statement of the lemma.
\end{proof}

The surface
$R(\g)$ defines an oriented  surface $R \subset X$. Note that \(R\) has the orientation coming from \(R (\g)\), \emph{not} the induced orientation coming from \(\partial M\). In particular, the homology class represented by \(R \) is twice the class of \( R_-(\gamma)\). It is easy to see that \(R\) is a nice decomposing surface for $(X,\g_X)$ in the sense of \cite[Definition 3.22]{Ju10}. Let \(O_R \subset \spinc(X,\gamma_X)\) be the set of outer \(\spinc\) structures for \(R\).

\begin{lemma}
\label{Lem:Face}
Let $(M,\g)$ be a taut balanced sutured manifold.  Then for all  \(\a \in H_2(M,\partial M;\R)\), we have \emph{
\begin{equation*} \label{eqn:z}
2z(\alpha) = \max \left\{\, \left\langle\, \s - \t, D_*(\alpha) \,\right\rangle \, \colon \, \s,\t \in O_R \cap S(X,\g_X)\,\right\}.
\end{equation*}}
\end{lemma}

\begin{proof}
First note that by standard arguments it suffices to show the equality when $\a$ is an integral class.

Now note that if we decompose $(X,\g_X)$ along $R,$ we get $(M',\g') = (M,\g) \sqcup (-M,\g).$ The set \(\spinc(M',\g')\) is naturally identified with \(\spinc(M,\g) \times \spinc(-M,\g).\) By \cite[Prop. 5.4]{Ju10}, there is a gluing map \(f_R: \spinc(M',\g') \to O_R \) such that \(f_R(\s) - f_R(\t) = i_*(\s-\t)\) for every $\s,\t \in \spinc(M',\g'),$ where \(i:M' \to X\) is the inclusion. (Here we view \(\s-\t\) as an element of \(H_1(M'),\) or equivalently, use \(PD \circ i_* \circ PD\) in place of $i_*.$) The decomposition theorem \cite[Proposition 5.4]{Ju10} implies that
$$f_R(S(M',\g'))  = O_R \cap S(X,\g_X).$$
Clearly,
$$ S(M',\g') =  S(M,\g) \times  S(-M,\g).
$$
Recall that there is a bijection \(\spinc(M,\g) \to \spinc(-M,\g)\) which sends a nonvanishing vector field \(v\) to \(-v\). By Proposition~\ref{prop:duality}, we have \(S(-M,\g) = - S(M,\g)\). Thus each element of \(O_R \cap S(X,\g_X)\) can be written as \(\s = f_R(\s_1,-\s_2)\), where \(\s_1,\s_2 \in S(M,\g)\). So for $\s,\t \in O_R \cap S(X,\g_X),$ we have
 \begin{align*}
 \langle \s - \t, D_*(\a) \rangle & = \langle f_R(\s_1,-\s_2) - f_R(\t_1, - \t_2), D_*(\a) \rangle \\
 & = \langle  i_*(\s_1-\t_1,\t_2-\s_2), D_*(\a) \rangle \\
 & =  \langle  \s_1-\t_1, \a \rangle +  \langle  \s_2-\t_2, \a \rangle.
 \end{align*}
In particular, we see that
 \begin{equation*}
 \max \left\{\, \left\langle\, \s - \t, D_*(\alpha) \,\right\rangle \, \colon \, \s,\t \in O_R \cap S(X,\g_X)\,\right\}
= 2 \cdot \max \{ \langle\,\s - \t, \alpha \,\rangle \, \colon \, \s,\t\in S(M,\g)\}.
 \end{equation*}
The right hand side is by definition $2z(\a).$
\end{proof}

We are now finally ready to complete the proof of Theorem~\ref{thm:Tnorm1}.

\begin{proof}[Proof of Theorem~\ref{thm:Tnorm1}]
If \((M,\g)\) is taut, this is an immediate consequence of Lemmas~\ref{Lem:Double} and \ref{Lem:Face}.  Now suppose that \((M,\g)\) is not taut. Since $M$ is irreducible, \cite[Proposition 9.18]{Ju06} implies that \(SFH(M,\g) = 0.\) So \(z =0,\) and the inequality is obviously true.
\end{proof}

Since the set of \(\spinc\) structures appearing in
Lemma~\ref{Lem:Face} is a proper subset of the set in
Lemma~\ref{Lem:Double}, it seems plausible that there should be sutured
manifolds for which  \(x^s\) is strictly larger than \(z\). We
explain how to find such a manifold.

The form of the support \(S(X,\g_X)\) is constrained
by the fact that \((X,\g_X)\) is a double. There is a natural
``reflection'' \(r:X \to X\) which exchanges the two copies of \(M\)
in the \(X\). The action of \(r_*\) decomposes \(H_1(X;\R)\)
into a direct sum of \(\pm 1\) eigenspaces \(A_{\pm}\).
For simplicity, let us suppose
that \((M,\g)\) is a rational homology product.

\begin{lemma}
If \((M,\g)\) is a rational homology product, then \(A_-\) is one
dimensional, and \(A_+  \cong H_1(R_-)\). Moreover,
\(A_-\) is the annihilator of \(\im (D_*)\) under the intersection
 pairing $H_1(X;\R)\times H_2(X,\partial X;\R)\to \R$, and
\(A_+\) is the annihilator of \([R]\).
\end{lemma}
\begin{proof}
The fact that \((M,\g)\) is a rational homology product implies
 that the map \(i_*:H_1(R_-;\R) \to H_1(X;\R)\)
is injective and \(b_1(X) = 1+b_1(R_-)\). The action of \(r\) fixes
\(R_-\) pointwise, so \(\im(i_*) \subset A_+\). To prove the first
claim, it is enough to construct a nonzero element of \(A_-.\)
 Choose an arc in \(M\) joining \(R_-(\gamma)\) to \(R_+(\g)\), and let
\(\alpha \subset X\) be its double, Then \(r\) acts by reflection on
\(\alpha\), so \(\alpha \in A_-\). Finally, \(\alpha \cdot [R] =2\),
so \([\alpha] \neq 0\).

To prove the second claim, we observe that
the subspace spanned by \([R]\) is the \(+1\) eigenspace for the
action of \(r_*\) on \(H_2(X,\partial X;\R)\), while \(\im(D_*)\) is the
\(-1\) eigenspace. Since \(r_*\) acts by multiplication by \(-1\)
on \(H_3(X,\partial X;\R)\), the \(\pm 1\) eigenspace in \(H_1(X;\R)\) pairs
trivially with the \(\pm 1\) eigenspace in \(H_2(X,\partial X;\R).\)
\end{proof}

Since the intersection pairing is perfect, \(A_+ = \im(i_*)\)
can be naturally identified with the dual of
\(\im(D_*)\). Let \(B_{x^s}^* \subset H^2(M,\partial M;\R) \cong H_1(M;\R)\) be the
dual unit norm ball of \(x^s.\) As in \cite{Ju10}, let
$$P = P(X,\g_X) = \text{conv}\, \{\, c_1(\s) \colon \s \in S(X,\gamma_X) \,\} \subset H^2(X,\partial X;\R)$$
be the sutured Floer homology polytope of \((X,\g_X),\) where $c_1(\s)$ is the relative real Chern class appearing in \cite[Remark 8.5]{Ju08}. Then Lemma~\ref{Lem:Double} implies that
$$2B^*_{x^s} = \pi(P),$$
where \( \pi: H_1(X;\R) \to A_+\) is the projection with kernel $A_-,$  and we have made use of the identification
\(A_+ \cong H_1(R_-;\R) \cong H_1(M,\R) \) coming from the fact that \((M, \g)\) is a rational homology product.
Similarly, let $$P_{\pm R} = \text{conv}\, \{\, c_1(\s) \colon \s \in O_{\pm R} \cap S(X,\gamma_X)  \,\} \subset H^2(X,\partial X;\R).$$ By \cite[Proposition 4.13]{Ju08}, $P_R$ and $P_{-R}$ are the two extremal faces of the polytope $P$ in the $A_-$ direction. Then Lemma~\ref{Lem:Face} says that $$2B^*_{z} = \pi(P_R).$$

The polytope \(P\) is invariant under the action of \(r_*\). Indeed, $r$ is orientation reversing and $r(s(\g_X)) = -s(\g_X),$ so $r$ is an diffeomorphism from $(X,\g_X)$ to $(-X,-\g_X).$ By Proposition \ref{prop:duality}, the sutured Floer homology polytope of $(-X,-\g_X)$ is also $P.$ This implies that $r_*(P) = P.$

It follows that \(r_*\) exchanges the  faces
\(P_{-R}\) and \(P_R\) of \(P\), so \(P_{-R}\) can be identified with
\(P_{R}\) by translating it in the direction given by
\(A_-\). Equivalently, \(\pi(P_R) = \pi(P_{-R}).\) From this, we conclude
that \(B_z^*=B_{x^s}^*\) if and only if \(P\) is isomorphic to \(P_R
\times I\) for some interval \(I.\)

\begin{remark}
It is clear from the definition that \(B_z^*\) is symmetric about the origin. The corresponding symmetry of \(P_R\) can be realized by composing \(r_*\) with the map \(P \to P\) which sends \(x\) to \(-x\) (note that $X$ has toroidal boundary,
hence $P = -P$).
\end{remark}

\begin{figure}[ht]
\begin{center}
\includegraphics{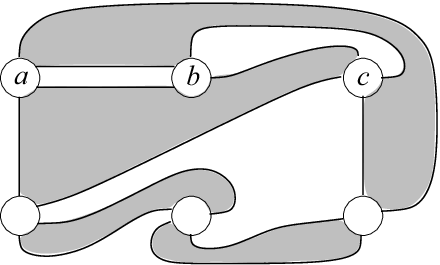}
\caption{Boundary of Cantwell and Conlon's example, showing
  sutures and compressing disks. The shaded region is \(R_-\).}
\label{fig:cc1}
\end{center}
\end{figure}

The Thurston norm of doubled manifolds was investigated by Cantwell
and Conlon in \cite{CC06}. They give an example
of a homology product such that the dual Thurston ball of the  doubled  manifold   is not  a product
\(P_R \times I\). (This is Example 2 in Section 5 of \cite{CC06}.) Figure~\ref{fig:cc1} shows the sutured manifold
\((M,\g)\) which is doubled to give their example. The
underlying manifold \(M\) is a genus three handlebody. The surface
\(\partial M\) is shown in the figure; it is obtained by identifying
the three circles labeled \(a,b,c\) with their counterparts in the
lower half of the diagram by reflecting across a horizontal line. The
circles \(a,b,c\) bound compressing disks in \(M\).
There are four
sutures, and  the subsurfaces \(R_{\pm}(\g)\) are
four-punctured spheres.

\begin{proposition} \label{prop:2}
For the sutured manifold $(M,\g)$ described above, \(B_{x^s}\) is a
proper subset of \(B_z\).
\end{proposition}

\begin{proof}
The Thurston ball of \((M,\g)\) was determined by Cantwell and Conlon
in Section 5.1 of \cite{CC06}. For convenience, we describe their
result in terms of the dual Thurston ball \(B_{x^s}^*\). Let \(A,B,\)
and \(C\) be the compressing disks in \(M\) bounded by the
corresponding circles, and oriented so that the induced orientation
on the boundary corresponds with the standard orientation on the upper
circle in each pair. Let \(a,b,c\) be the geometrically dual basis of
\(\pi_1(M)\). (In other words, \(a\) intersects \(A\) once positively and
misses \(B\) and \(C\), {\it etc.}) Finally, let \(e_a,e_b,e_c\) be
the corresponding basis of \(H_1(M)\).
 Then by \cite[Section~5.1]{CC06} the vertices of the dual
Thurston ball are
$$ \pm e_a,\  \pm e_b,\ \pm e_c,\ \pm (e_a-e_b),\ \pm(e_a-e_c),\
\pm(e_b-e_c),\ \  \text{and} \ \pm (e_a-e_b-e_c).$$

\begin{figure}[ht]
\begin{center}
\includegraphics{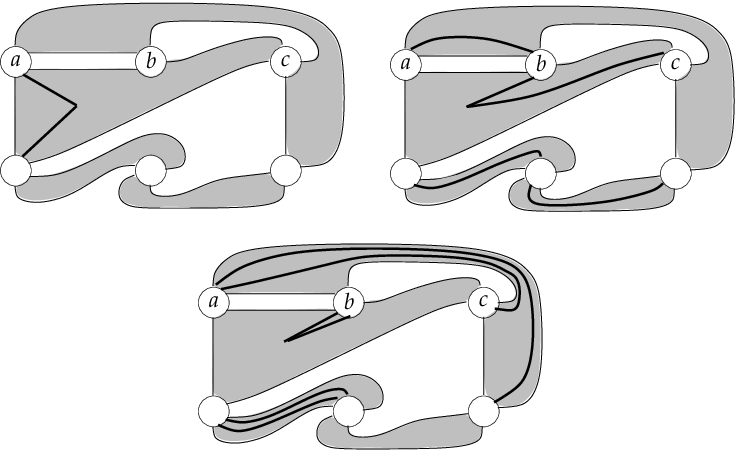}
\caption{A basis for \(\pi_1(R_-)\).}
\label{fig:cc2}
\end{center}
\end{figure}

Next, we determine the sutured Floer polytope of \((M,\g)\). To do so, we
first compute \(\tau(M,\g)\). The three loops shown in
Figure~\ref{fig:cc2} form a basis for \(\pi_1(R_-)\). The
corresponding words in \(\pi_1(M)\) are \(a,ba^{-1}bc^{-1}\), and
\(ba^{-1}cab^{-1}\), so by Proposition~\ref{prop:AMatrix}
\begin{equation*}
\tau(M,\g) \sim \det \begin{bmatrix}
1 & -ba^{-1} & -ba^{-1}+ba^{-1}c \\
0 & 1+ ba^{-1} & 1-c \\
0 & -b^2a^{-1}c^{-1} & ba^{-1}
\end{bmatrix}.
\end{equation*}
(To save space in writing the matrix, we have applied the
abelianization map, but omitted the \(\phi\)'s, so \(a\) is to be read as
\(\phi(a)\), {\it etc.}) We find that
$$ \tau(M,\g) \sim ac + bc  + ba - abc, $$
so up to a global translation the vertices of the torsion polytope are
 \(-e_a,-e_b,-e_c\) and \(0\).

Let \(Q\) and \(T\) denote the sutured Floer polytope
and torsion polytope of \((M,\g)\). We claim that \(Q=T\).
 To see this, consider the disk \(A \subset M\). \(\partial A\) intersects the sutures four times, so \(x^s([A])=1\). Applying Theorem~\ref{thm:Tnorm1}, we see that
  the projection of \(Q\) to the \(e_a\) axis takes at most two values.
 Comparing with \(T\), we see that these values must be
\(-1\) and \(0\). A similar argument applied to \(B\) and \(C\) shows that the \(e_b\) and \(e_c\) coordinates of a point in \(Q\) must be \(0\) or \(-1\).

Following Cantwell and Conlon, we observe
that there is another embedded disk \(D \subset M\) whose boundary
intersects the sutures in four points. The boundary of this disk
separates the three upper circles in Figure~\ref{fig:cc1} from the
lower circles. The homology class of \(D\) is \([A]+[B]+[C]\), so by
adjunction, the linear functional \(e_a^*+e_b^*+e_c^*\) can take at
most two values on \(Q\). Comparing with \(T\), we see that these
values must be \(-1\) and \(0\). So \(T\) and \(Q\)
are the same.

Finally, we recall that the polytope \(B_z^*\) is the convex hull of the set of
points of the form \(x-y\), for \(x,y \in Q\). The vertices of
\(B_z^*\) are
$$ \pm e_a,\  \pm e_b,\ \pm e_c,\ \pm (e_a-e_b),\ \pm(e_a-e_c),\
\text{and} \ \pm(e_b-e_c).$$
As promised, \(B_z^*\) is a proper subset of
\(B_{x^s}^*\); it does not contain the vertex \(e_a-e_b-e_c\).
\end{proof}

In the above example \(b_1(X) = 4\), so
the polytope $P$ is four-dimensional. The lattice polytope \(P \cap H_1(X;\Z)\) is composed
of three ``layers'' distinguished by the value of their intersection
number with \([R]\). Two of these layers are the outer faces \(P_{R}\)
and \(P_{-R},\) both of which are isomorphic to \(2B_z^*\).
The middle layer is larger --- it is isomorphic to $2B_{x^s}^*.$
The fact that there are (at most) three layers in $P$ is a direct
consequence of the adjunction inequality applied to \(R_- \subset
X\).

\begin{remark} Observe that if \(\partial M\) had genus less than three, we would necessarily have
 \(B_z^*=B_{x^s}^*\). Indeed, the corresponding lattice polytope would have at most two layers by adjunction. These must be  \(P_{R}\) and \(P_{-R}\), so   \(P = P_R \times I\).
\end{remark}

Let $(M,\g)$ be a balanced sutured manifold such that $H_2(M)=0.$ In \cite[Proposition 8.10]{Ju08} the second author showed that if $(M,\g) \rightsquigarrow^S (M',\g')$ is a decomposition along a ``nice" product annulus $S,$ then $SFH(M,\g) \cong SFH(M',\g').$ Theorem~\ref{thm:Tnorm1} permits us to extend this result to irreducible balanced sutured manifolds with arbitrary second homology.

\begin{proposition}
Suppose that $(M,\g)$ is an irreducible balanced sutured manifold. Let $S \subset (M,\g)$ be a product annulus such that at least one component of $\partial S$ is non-zero in $H_1(R(\g)),$ or both components of $\partial S$ are boundary-coherent
in $R(\g)$ (see \cite[Definition 1.2]{Ju08}). If $S$ gives a surface decomposition $(M,\g) \rightsquigarrow^S (M',\g'),$ then $$ SFH(M,\g) \cong SFH(M',\g').$$
\end{proposition}

\begin{proof}
In both cases we can orient $S$ such that $\partial S$ is boundary-coherent in $R(\g).$
Note that $(M',\g')$ is taut if and only if $(M,\g)$ is, and both are irreducible. So if
$(M',\g')$ is not taut, then by \cite[Proposition 9.18]{Ju06} we have  $$SFH(M,\g) = 0 = SFH(M',\g').$$

Now assume that $(M',\g')$ is taut, then $SFH(M',\g') \neq 0$ by \cite[Theorem 1.4]{Ju08}. \cite[Theorem 1.3]{Ju08} implies that
$$SFH(M',\g') \cong \bigoplus_{\s \in O_S} SFH(M,\g,\s).$$ So $O_S \cap S(M,\g) \neq \emptyset.$ Let $\s_0 \in O_S \cap S(M,\g),$ then for any $\s \in \spinc(M,\g)$ we have $\langle\, \s -\s_0, [S] \,\rangle = 0$ if and only if $\s \in O_S.$ Obviously, $x^s([S]) = 0,$ so by Theorem~\ref{thm:Tnorm1} $z([S]) = 0.$ Thus $\langle\, \s - \s_0,[S] \,\rangle = 0$ for every $\s \in S(M,\g).$ So $S(M,\g) \subset O_S,$ and hence $SFH(M,\g) \cong SFH(M',\g').$
\end{proof}

\section{Examples and Applications} \label{examples}

We conclude with some sample computations of the torsion and/or the sutured Floer homology, with emphasis on  the case where \((M,\g)\) is the complement of a Seifert surface \(R \subset S^3\).

\begin{example}
Suppose that \(R \subset S^3\) is an embedded annulus. Then \(\partial R\) consists of two parallel copies of a knot \(K\) with some linking number \(n\) corresponding to the framing of the annulus.
The complementary sutured manifold \(S^3(R)\) is homeomorphic to \(S^3 \setminus N(K)\). Its boundary is a torus with two sutures, each representing the homology class \(\ell + n m\) with respect to the canonical basis on \(H_1(\partial(S^3 \setminus N(K))\). Let \(K_n\) be the manifold obtained by filling this homology class (i.e., by performing \(n/1\) Dehn surgery on \(K\)), and let \(K(n) \subset K_n\) be the core circle of the filling. Then \(SFH(S^3(R))\) is isomorphic to \(\widehat{HFK}(K(n))\).
Its Euler characteristic is given by
$$\tau(S^3(R))  \sim \Delta_K(t) \cdot \frac{t^n-1}{t-1},$$
{\it cf.} Proposition 3.1 of \cite{Ra07}.
Note that when \(n=0\), the torsion vanishes, regardless of what \(K\) is.
The group \( \widehat{HFK}(K(n))\) has been studied by Eftekhary \cite{Ef05} (in the case \(n=0\)) and Hedden \cite{He07}, who gives a complete calculation in terms of the groups \(HFK^-(K).\) See also Section 10 of \cite{LOT08}.
In particular,   \( \widehat{HFK}(K(0))\) is nontrivial unless \(K\) is the unknot.
\end{example}

\begin{example} \label{ex:4}
Suppose \(M\) is a solid torus, and that \(\g\) consists of \(2n\) parallel curves on \(\partial M\), each of which represents \(p\) times the generator of \(H_1(M)\). The group \(SFH(M,\g)\) was computed by the second author in \cite{Ju10}; its Euler characteristic is given by
$$ \tau(M,\g) \sim  \frac{(t^p-1)^{n}}{t-1}.$$
The homology in each \(\spinc\) structure is a free module of rank equal
to the Euler characteristic in that \(\spinc\) structure. An important special case is when \(n=1\). In this case \((M,\gamma)=S^3(R)\), where \(R\) is a twisted band with \(p\) full twists (in other words, an unknotted annulus in \(S^3\) with framing \(p\).) \(SFH(M,\g)\) is supported in \(p\) consecutive \(\spinc\) structures, each containing a single copy of \(\Z\).
\end{example}

\begin{example}
\label{ex:2bridge}
Suppose \(K=K(p,q) \subset S^3\) is the two-bridge knot or link corresponding to the fraction \(p/q\). The set of minimal genus Seifert surfaces for \(K\)  has been classified up to isotopy by Hatcher and Thurston \cite{HT85}. Any such surface \(R\) is obtained as a Murasugi sum of twisted bands. By \cite[Cor. 8.8]{Ju08} and \cite[Theorem 5.11]{Ju10}, the sutured Floer homology of  the complement of a Murasugi sum is the tensor product of the \(SFH\) of the complements of the summands.  Thus the sutured Floer polytope \(SFH(S^3(R))\) is a rectangular prism whose dimension is given by the number of bands.
The length of its sides is determined by the number of twists in the different bands.

 The number of bands is
 the number of  terms in the unique continued fraction expansion of \(p/q\) all of whose terms are even, and the number of twists in a given band is half the corresponding coefficient in the continued fraction expansion \cite{HT85}. For example, the knot \(K(56,15)\) has continued fraction expansion
$$\frac{56}{15} = 4 - \cfrac{1}{4 - \frac{1}{4}}.$$
Any Seifert surface  \(R\) of \(K\) is a Murasugi sum of three twisted bands, each with two full twists. \(SFH(S^3(R)) \cong \Z^8\) is supported at the vertices of a \(2 \times 2 \times 2\) cube.

Whenever two of the bands have more than one twist, \(K\) will have more than one Seifert surface. The calculation above shows that these surfaces can not be distinguished by their sutured Floer polytope alone. In contrast, we have the following.

\begin{thm*} \cite{HJS08}  
There exist two minimal genus Seifert surfaces $R_1$ and $R_2$ for \(K(17,4)\) which can be distinguished by combining sutured Floer homology with the Seifert form. More precisely, there does not exist an orientation--preserving diffeomorphism between the pairs $(S^3,R_1)$ and $(S^3,R_2)$.
\end{thm*}

By Corollary~\ref{cor:alt},  the groups $SFH(S^3(R_i),\s)$ for $i=1,2$ and $\s\in \spinc(S^3(R_i))$ are determined by $T_{S^3(R_i)}(\s)$ since two-bridge knots are alternating. Hence it is  straightforward to modify the proof of the theorem  to show that the Seifert surfaces can also be distinguished by the torsion and the Seifert form.

\end{example}

\begin{figure}[ht] \begin{center}
\begin{tabular}{cc}
\includegraphics[trim = 0mm 228mm 121mm 5mm, clip,
width=6.6cm]{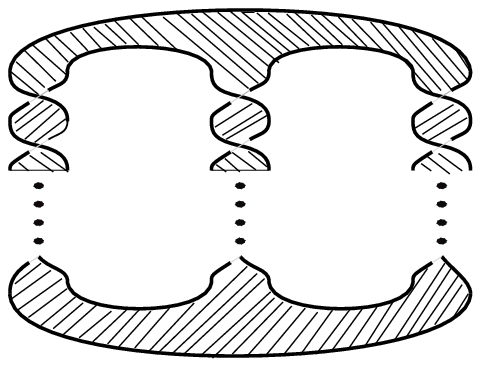} &
\includegraphics{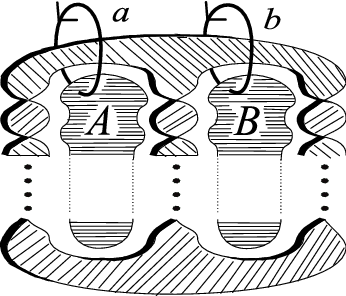}\\
(a) & (b)
\end{tabular}
 \caption{The pretzel knot $P(2r+1,2s+1,2t+1)$ with $2r+1, 2s+1,$ and $2t+1$ half twists.}
\label{fig:Pretzel}
\end{center}
 \end{figure}


\begin{example}

The pretzel knot \(P(2r+1,2s+1,2t+1)\) has an obvious Seifert surface \(R\), as shown in Figure~\ref{fig:Pretzel}(a).
A natural pair
of compressing disks \(A\) and \(B\)  for the handlebody $M = S^3 \setminus N(R)$ is shown in Figure~\ref{fig:Pretzel}(b).
Cutting \(M\) along these disks and using the Seifert--van Kampen theorem give an isomorphism between \(\pi_1(M)\) and the free group generated by \(a\) and \(b\).  If \(\alpha\) is a curve on \(\partial M\), we can read off the word it represents in \(\pi_1(M)\) by traversing \(\alpha\) and recording its intersections with \(\partial A\) and \(\partial B\).

Suppose   that \(R\) is oriented so that the region visible in the bottom of the  figure belongs to \(R_-(\gamma)\). Put \(p\) in this region, and let \(\alpha\) be a loop which runs from \(p\) up the left-hand strip, and back down via the middle strip. Similarly, let \(\beta\) be a loop which runs up the left-hand strip and back down the right, so that  \(\pi_1(R_-(\gamma),p)\) is generated by \(\alpha\) and \(\beta\). The reader can easily verify that
\begin{equation*}
\a  = a^{r}(a(b^{-1}a)^s), \quad \quad
\b  = a^rb^{t+1}.
\end{equation*}
Set $x = a^r,$ $y = a(b^{-1}a)^s,$ and $w = b^{t+1},$ then $\a = xy$ and $\b = xw.$ To compute the torsion, we evaluate
$$
\det
\begin{bmatrix}
 \partial \alpha/\partial a  & \partial \beta / \partial a \\ \partial  \alpha / \partial b &\partial  \beta/ \partial b
\end{bmatrix}
= \det
\begin{bmatrix}
d_a x + \phi(x) d_a y & d_a x \\
\phi(x) d_b y & \phi(x) d_b w
\end{bmatrix},
$$
where we have written \(d_ax\) for \(\phi(\partial x / \partial a)\), {\it etc.}
We find that
\begin{equation*}
\tau(S^3(R))  \sim d_a x \thinspace d_b w + \phi(x) d_ay \thinspace d_b w - d_b y \thinspace d_a x.
\end{equation*}
After evaluating the Fox derivatives and clearing fractions, we obtain
\begin{multline*}
(1-a)(1-b)(1-ab^{-1}) \tau(S^3(R))
= (1-a^r)(1-b^{t+1})(1-ab^{-1})     \\  \phantom{XXXXXXXXXXXXXXXXXX} + a^r(1-(ab^{-1})^{s+1})(1-b^{t+1})(1-a) \\ + ab^{-1}(1-(b^{-1}a)^s)(1-a^r)(1-b) .
\end{multline*}
Expanding the right-hand side, we get
\begin{multline*}
-a^{r+s+1}b^{-s}(1-ab^{-1})  + a^{r+s+1}b^{t-s}(1-a) - a^{r+1}b^t(1-b) \\
-b^{t+1}(1-ab^{-1}) + (1-a) - a^{s+1}b^{-s-1}(1-b).
\end{multline*}
To compute the torsion, we must divide this expression by \((1-a)(1-b)(1-ab^{-1})\).
If \(r,s,\) and \(t\) are all positive, we find that the torsion
 is supported on a hexagon, as illustrated in Figure~\ref{fig:+Pretzel}. (The easiest way to see this is to start with the sum of all monomials corresponding to vertices in the hexagon, and then multiply by
 \((1-a)(1-b)(1-ab^{-1})\). Regardless of \(r,s\) and \(t\), the product will have 6 pairs of terms, each supported near a vertex of a hexagon. These are the pairs appearing in the equation above.)

 With respect to the natural basis given by \(a\) and \(b\), the sides
 of the hexagon have slope \(0,-1\) and \(\infty\). Parallel sides
 have the same length, and the sides are of length \(r+1,t+1\), and
 \(s+1\). The coefficient of the torsion at each lattice point in the
 hexagon is \(1\), and the sutured Floer homology consists of a single
 copy of \(\Z\) at each lattice point since the pretzel knot is
 alternating.
\end{example}

 \begin{figure}[ht] \begin{center}
\includegraphics{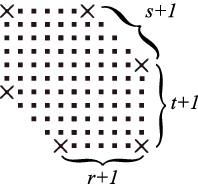}  
 \caption{Support of $\tau(S^3(R))$ for $r,s,t>0$.}
\label{fig:+Pretzel}
\end{center}
 \end{figure}

The case where \(r,t>0\) and \( s<0\) can be treated similarly. We distinguish two subcases, depending on whether \(|2s+1|\) is less than \(\min (2r+1,2t+1)\), or greater. In the first, the coefficients of \(\tau(S^3(R))\) take on both positive and negative signs. The torsion is supported on a ``bowtie'', as shown in Figure~\ref{fig:-Pretzela}. The coefficient of the torsion is \(1\) at each lattice point in the rectangle, and \(-1\) at each lattice point in the two triangles. In the second case, the support is a nonconvex hexagon, as illustrated  in  Figure~\ref{fig:-Pretzelb}. The coefficient of the torsion is \(-1\)
at each lattice point in the hexagon. To determine the sutured Floer homology, we compare with the calculation of \(\widehat{HFK}(P(2r+1,2s+1,2t+1))\) given in \cite{OS04c}. In both cases,  the top group  in the knot Floer homology is torsion free and its rank is equal to the number of vertices in the support of $\tau(S^3(R)).$ It follows that $SFH(S^3(R))$ has rank one at each vertex in the support and is trivial elsewhere.

\begin{figure}[ht] \begin{center}
\includegraphics{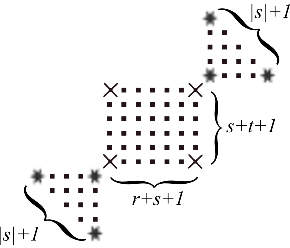}  
\caption{Support of $\tau(S^3(R))$ for $r,t>0$ and $-\min(r,t)\leq s+1 \leq 0$.}
\label{fig:-Pretzela}
\end{center}
 \end{figure}

 \begin{figure}[ht] \begin{center}
\includegraphics{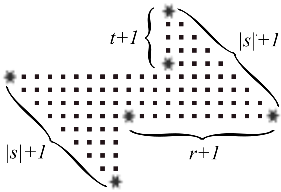}  
\caption{Support of $\tau(S^3(R))$ for $r,s>0$ and $s+1\leq -\min(r,t)$.}
\label{fig:-Pretzelb}
\end{center}
 \end{figure}

\begin{example} \label{ex:nonsymmetric} The three-component pretzel link \(P(2r,2s,2t)\) has a Seifert surface \(R\) similar to that shown in Figure~\ref{fig:Pretzel}. We compute \(\tau(S^3(R))\)  as in the previous example. When \(r,s,t\) are all positive, the torsion is again supported on a hexagon with sides of slopes \(0,-1\) and \(\infty\). (In this case, the relevant words in \(\pi_1(M)\) are  \(\a = a^r(b^{-1}a)^s\) and \(\b = a^rb^t\).) However, in this case parallel sides of the hexagon do not have the same lengths. Instead, the sides have lengths \(r+1,t,s+1,r,t+1,s\) as we go around the hexagon. This gives a simple family of examples for which the torsion does not exhibit any symmetry. The phenomenon is already evident for \(P(2,2,2)\). In this case, the hexagon degenerates to a triangle supported at three vertices in the plane. With respect to the standard basis \(a,b\), these vertices can be taken to be \((0,0)\), \((1,0)\) and \((1,1)\). Note that for \(r,s,t>0\),  \(P(2r,2s,2t)\) is an alternating link, and hence \(SFH(S^3(R))\) has the same support as the torsion. Thus the sutured Floer polytope is asymmetric as well.
\end{example}

\begin{example} Seifert surfaces of small knots.
Let \(K\) be a knot in \(S^3\) and suppose \(R\) is a Seifert surface for \(K\).  Among knots with nine crossings or fewer, most are either two-bridge or fibred. (See {\it e.g.} the tables in \cite{Ka96} or \cite{CL09}.) If \(K\) is fibred, \(SFH(S^3(R)) \cong \Z \); if it is two-bridge, \(SFH(S^3(R))\) was determined in Example~\ref{ex:2bridge}.  The remaining knots all have a unique Seifert surface \(R\) by \cite{Ka05}. We  briefly describe the groups
\(SFH( S^3(R))\). They fall into two broad classes, as well as a few knots with more interesting homology.

 \begin{itemize}
 \item The knots \(9_{16},9_{37},\) and \(9_{46}\)  have Seifert surfaces that decompose as Murasugi sums of a single twisted band with two full twists together with some other twisted bands with one full twist. Thus \(SFH(S^3(R)) \cong SFH(A_2)\), where \(A_2\) is an unknotted annulus with two full twists.
 \item The knots \(8_{15},9_{25},9_{39},9_{41}\), and \(9_{49}\)  have Seifert surfaces which are Murasugi sums of once-twisted bands and a single copy of \(R(2,2,2)\) ---  the Seifert surface of the \((2,2,2)\) pretzel link. Thus
  \(SFH(S^3(R) \cong SFH(S^3(R_{2,2,2}))\) is supported on a triangle.
   \item The knot \(9_{35}\) is \(P(3,3,3)\). Its sutured Floer polytope is a hexagon with sides of length \(2\).
 \item The knot \(9_{38}\) is the only knot with fewer than 10 crossings whose sutured Floer polytope is \(3\)-dimensional. The polytope is contained in a \(2\times2\times2\) cube, with \(\Z\) summands at five of the vertices: \((1,0,0)\), \((0,1,0)\), \((0,0,1),(1,1,0)\), and \((1,0,1)\). (We computed this directly using Proposition~\ref{prop:AMatrix}.)
 \end{itemize}
 For all of these knots, the top group in knot Floer homology is torsion free and supported in a single homological grading, so \(SFH\) is determined by the torsion.
\end{example}

\begin{example}
The four-strand pretzel link \(L=P(n,-n,n,-n)\) has a genus one Seifert surface analogous to the one shown in Figure~\ref{fig:Pretzel}. The multivariable Alexander polynomial of this link is \(0\), but a calculation similar to the one in Example 4
 shows that the torsion polytope is a ``pinwheel'' which is a disjoint union of four square pyramids, each with side length \(n\). It follows that the rank of \(HFK(L)\) in the top Alexander grading is at least
$$ 4 \sum _{k=1}^n k^2 = \frac{2n(n+1)(2n+1)}3. $$
\end{example}

\begin{figure}[ht]
\begin{center}
\includegraphics{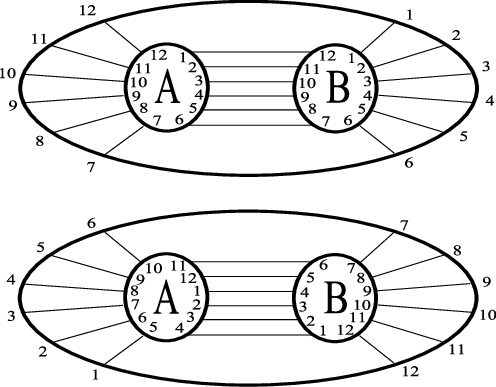}
\caption{Sutured genus two handlebody with no disk decomposition.}
\label{Fig:NoDisk}
\end{center}
\end{figure}

\begin{example} \label{ex:disk}
We conclude by using the torsion to give an example of a phenomenon first observed by Goda \cite{Go94}. Namely, there exist sutured manifolds whose total space is a genus two handlebody, but which are not disk-decomposable.
Consider the  two pairs of pants   illustrated in Figure~\ref{Fig:NoDisk}.
We consider the genus two surface obtained by gluing the two pairs of paints along
the corresponding boundary curves  by identifying the corresponding numbers $1,2,\dots,12$.
Let \(M\) be the handlebody in which the curves labeled \(A\) and \(B\) bound compressing disks, and let \(s(\g)\) be the multi-curve shown in the figure. Then we easily compute
\begin{equation*}
\tau(M,\g) \sim 2a - 3 + 2a^{-1}.
\end{equation*}
\begin{proposition}
\((M,\g)\) is not disk-decomposable.
\end{proposition}
\begin{proof}
Suppose we decompose  \((M, \g)\) along a disk \(D\) to obtain a sutured manifold \((M', \g')\). If \(\partial D\) is a non-separating curve in \(\partial M\), then   \(M'\) is homeomorphic to \(S^1 \times D^2\). If \((M',\g')\) were taut, then by \cite{Ju10}, \(SFH(M',\g')\) would be isomorphic to the restriction of \(SFH(M,\g)\) to those \(\spinc\) structures which are extremal with respect to evaluation on \([D]\) . It follows that either \(\tau(M',\g') = 0\), \(\tau(M',\g') \sim  2\), or \(\tau(M',\g') \sim 2t-3+2t^{-1}\). Comparing with Example 2, we see that none of these are the torsion of a taut sutured manifold whose total space is the solid torus.

Similarly, if \(\partial D\) is a separating curve, then \(M'\) is homeomorphic to the disjoint union of two solid tori \(M_1\) and \(M_2\), and $$SFH(M,\g) \cong SFH(M', \g') \cong SFH(M_1,\g_1) \otimes SFH(M_2,\g_2).$$ Again, comparing \(\tau(M,\g)\) with Example~\ref{ex:4} shows that this is not possible.
\end{proof}
\end{example}

\end{document}